\documentclass{agt}
\usepackage{hyperref}
\hypersetup{nesting=true,debug=true,naturalnames=true}
\usepackage{graphicx,amssymb,upref}
\usepackage{amsmath,amssymb,amsthm} \usepackage[mathscr]{eucal}
\usepackage{amscd} \usepackage{setspace} \usepackage{epsfig}
 \usepackage{dsfont}\usepackage{tensor}
\usepackage{mathtools,placeins}
\usepackage{multicol,hyperref,enumitem,pinlabel}
\usepackage[usenames,dvipsnames]{color}
\usepackage{float}
\usepackage{tikz}
\usetikzlibrary{matrix}

\newcommand\vv{\vec{v}}
\newcommand\y{\vec{y}}
\newcommand\x{\vec{x}}

\newcommand\uu{\vec{u}}
\newcommand\CP{\mathbb{CP}}
\newcommand\ep{\varepsilon}
\newcommand\lla{\left\langle}
\newcommand\rra{\right\rangle}

\newcommand\red[1]{\color{red}#1\color{black}}
\newcommand\yellow[1]{\color{yellow}#1\color{black}}
\newcommand\FG[1]{\color{ForestGreen}#1\color{black}}
\newcommand\violet[1]{\color{black}#1\color{black}}

\newcommand\Navy[1]{\color{black}#1\color{black}}
\newcommand\NavyBlue[1]{\color{NavyBlue}#1\color{black}}
\newcommand\cube{\raisebox{-1pt}{\includegraphics[height=8pt]{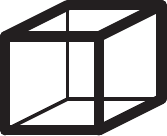}}\hspace{1pt}}

\theoremstyle{plain}
\newtheorem{theorem}{Theorem}[section]
\newtheorem{lemma}[theorem]{Lemma}
\newtheorem{obs}[theorem]{Observation}
\newtheorem{prop}[theorem]{Proposition}
\newtheorem{cor}[theorem]{Corollary}
\newtheorem*{Conj:n04}{Conjecture \ref{Conj:n04}}
\newtheorem*{P:n24}{Proposition \ref{P:n24}}
\newtheorem*{Q:n24}{Question \ref{Q:n24}}
\newtheorem*{T:Main}{Theorem \ref{T:Main}}
\newtheorem*{P:SmoothBoth}{Proposition \ref{P:SmoothBoth}}
\newtheorem*{P:Smooth}{Proposition \ref{P:Smooth}}
\newtheorem*{P:SmoothEven}{Proposition \ref{P:SmoothEven}}
\newtheorem*{C:SmoothEven}{Corollary \ref{C:SmoothEven}}
\newtheorem*{P:SmoothEfficient1}{Proposition \ref{P:SmoothEfficient1}}
\newtheorem*{C:SmoothEfficient2}{Corollary \ref{C:SmoothEfficient2}}
\newtheorem*{C:SOdd}{Corollary \ref{C:SmoothHandleOdd}}
\newtheorem*{C:SEven}{Corollary \ref{C:SmoothHandleEven}}
\newtheorem*{L:XI}{Lemma \ref{L:XI}}

\theoremstyle{definition}
\newtheorem{convention}[theorem]{Convention}
\newtheorem{notation}[theorem]{Notation}
\newtheorem{definition}[theorem]{Definition}
\newtheorem{question}{Question}
\newtheorem{conjecture}[question]{Conjecture}

\newtheorem{example}[theorem]{Example}

\theoremstyle{remark}
\newtheorem*{rem}{Remark}

\begin{document}

\title[Multisections of tori]{Efficient multisections of odd-dimensional tori}

\author{Thomas Kindred}

\address{Department of Mathematics and Statistics, Wake Forest University \\
Winston Salem, NC 27109, USA} 

\email{kindret@wfu.edu}
\urladdr{www.thomaskindred.com}

\begin{abstract}
Rubinstein--Tillmann generalized the notions of Heegaard splittings of 3-manifolds and trisections of 4-manifolds by defining {\it multisections} of PL $n$-manifolds, which are decompositions into $k=\lfloor n/2\rfloor+1$ $n$-dimensional 1-handlebodies with nice intersection properties. 
For each odd-dimensional torus $T^n$, we construct a multisection which is {\it efficient} in the sense that each 1-handlebody has genus $n$, which we prove is optimal; each multisection is {\it symmetric} with respect to both the permutation action of $S_n$ on the indices and the $\Z_k$ translation action along the main diagonal. We also construct such a trisection of $T^4$, lift all symmetric multisections of tori to certain cubulated manifolds, and obtain combinatorial identities as corollaries.
\end{abstract}

\maketitle

\section{Introduction}\label{S:Intro}

Every closed 3-manifold\footnote{Unless stated otherwise, all {\bf manifolds} are piecewise-linear (PL), compact, connected, and orientable. A manifold $X$ is {\bf closed} if $\partial X=\varnothing$. A general reference is \cite{rs}.} $X$ admits a decomposition into two 3-dimensional 1-handlebodies\footnote{A ${d}$-dimensional {\bf $\boldsymbol{h}$-handlebody} is a $d$-manifold obtained by gluing $d$-dimensional $r$-handles for various $r=0,\hdots, h$. Since we work in the PL category, the gluing maps must be PL and the attaching regions must be PL submanifolds.} glued along their boundaries. 
Gay--Kirby extended this classical notion of {\it Heegaard splittings} by proving that every closed 4-manifold admits a {\it trisection}, i.e. a decomposition $X=\bigcup_{i\in\Z_3}X_i$ where each $X_i$ is a 4-dimensional 1-handlebody, each $X_i\cap X_{i+1}$ is a 3-dimensional 1-handlebody, and $X_0\cap X_1\cap X_2$ is a closed surface. 
Rubinstein-Tillmann \cite{rt} then extended these decompositions to arbitrary dimension by proving that every closed (PL) manifold of arbitrary dimension admits a {\it PL multisection}:
 
\begin{definition}\label{Def:Multi}
A {\it PL multisection} of a closed manifold $X$ of dimension $n=2k-1$ (resp. $2k-2$) is a decomposition $X=\bigcup_{i\in\Z_k}X_i$, where:
\begin{itemize}
\item Each $X_i$ is an $n$-dimensional 1-handlebody.
\item 
$\bigcap_{i\in \Z_k}X_i$ is a closed $(n+1-k)$-dimensional submanifold.
\item $\bigcap_{i\in I}X_i$ is an $(n+1-|I|)$-dimensional $|I|$- (resp. ($|I|-1$)-) handlebody for each $I\subset\Z_k$ with $2\leq |I|\leq k-1$.\footnote{Rubinstein-Tillmann state this condition differently, requiring that each $\bigcap_{i\in I}X_i$ is an $(n+1-|I|)$-dimensional submanifold with an $|I|$- (resp. $(|I|-1)$-) dimensional spine, where a {spine} of a manifold $N$ is a subpolyhedron $P\subset\text{int}(N)$ onto which $N$ collapses. Certainly any $h$-handlebody has an $h$-dimensional spine. Conversely, given a spine $P$ of $N$, we may assume that $N$ is triangulated and $P$ is a simplicial subcomplex which admits no elementary collapses; then $N$ is PL homeomorphic to a regular neighborhood $R$ of $P$ in $N$, and $R$ has handle decomposition consisting of one $r$-handle for each $r$-simplex in $P$.}
\end{itemize}
\end{definition}


One may define {\it smooth} multisections of {\it smooth} manifolds analogously: the only extra condition is that for each nonempty $I\subset \Z_k$, the inclusion of $X_I=\bigcap_{i\in I}X_i$ into $X$ is a smooth embedding, with corners.\footnote{More precisely, for nonempty $I\subset \Z_k$, the set of corner points of $X_I$ must be:
\centerline{$\displaystyle{\text{corners}(X_I)=\bigcup_{i,j\notin I;~i\neq j}X_I\cap X_i\cap X_j.}$}} Lambert-Cole--Miller proved that every smooth 5-manifold admits a smooth trisection \cite{lcm}. 
In dimensions $n\geq 6$, the topic is wide open. In particular:

 \begin{question}\label{Q:1}
 Does every closed smooth manifold of arbitrary dimension admit a smooth multisection?
 \end{question}

The distinction between PL multisections and smooth ones comes down to that of PL and smooth handle decompositions.\footnote{Note that, while any smooth structure determines a (smooth) handle decomposition, and conversely, a PL handle decomposition does not necessarily determine a smooth structure.}  This is because any PL multisection $X=\bigcup_{i\in \Z_k}X_i$ gives rise to a nice PL handle decomposition (see Proposition \ref{P:SmoothBoth}) coming from handle decompositions of the various $X_I$; requiring each inclusion $X_I\hookrightarrow X$ to be smooth (with corners) ensures that the gluings in this handle decomposition are smooth. Henceforth, unless stated otherwise, all {multisections} are PL.



The topology of a closed manifold $X$ of dimension $n\neq 2$ bounds the {\it efficiency} of its multisection as follows. 
Let $g(X_i)$ denote the {\bf genus} of $X_i$.\footnote{$X_i$ is an $n$-dimensional 1-handlebody, so we have $X_i\cong \natural^g(S^1\times D^{n-1})$ for some $g=g(X_i)$. (Throughout, we denote PL homeomorphism by $\cong$.)} 

\begin{definition}\label{D:Efficiency}
The {\bf efficiency} of a multisection $X=\bigcup_{i\in\Z_k}X_i$ is 
\[\frac{1+\text{rank }\pi_1(X)}{1+\max_ig(X_i)}.\]
A multisection is {\bf efficient} if its efficiency is 1. 
\end{definition}
We will show:
\begin{C:SmoothEfficient2}
In any dimension $n\neq 2$, no multisection of any manifold has efficiency greater than 1, and in any efficient multisection $X=\bigcup_{i\in \Z_k}X_i$, all $X_i$ have the same genus, $g(X_i)= \textnormal{rank }\pi_1(X)$.%
\footnote{In dimension two, efficiency is strictly bounded above by 2; this bound is sharp, since any surface of even genus $g$ admits a multisection with efficiency $\frac{1+2g}{1+g}$.}
\end{C:SmoothEfficient2}

This notion of an {\it efficient} multisection generalizes a notion introduced by Lambert-Cole--Meier in \cite{lcm}.  They call a trisection of a simply-connected 4-manifold $X$ {\it efficient} if the genus of the central surface $\Sigma$ equals $b_2(X)$. Indeed, one always has $g(\Sigma)\leq b_2(X)$, and equality holds if and only if each piece of the trisection is a 4-ball.  

We close the introduction with an outline of the paper.  
\begin{itemize}[label=$\bullet$]
\item \textsection\ref{S:Smooth} establishes several general properties of multisections.

\item \textsection\ref{S:Motivation} begins a detailed investigation of multisections of {\it odd-dimensional tori}, starting with detailed descriptions the multisections of $T^n$ for $n=3,4,5$.
 Roughly stated, the main result is:

\begin{T:Main}
Each $n=(2k-1)$-torus admits an efficient multisection which is symmetric with respect to the $S_n$ permutation action on the indices and the $\Z_k$ translation action along the main diagonal.
\end{T:Main}

\noindent The full version of Theorem \ref{T:Main} gives a simple expression (\ref{E:X0ThmIntro}) for each piece $X_i$ of this multisection. The hard part is describing a handle decomposition of $X_I=\bigcap_{i\in I}X_i$ for arbitrary $n$ and $I\subsetneqq \Z_k$. 

\item \textsection\ref{S:StarShaped} introduces three types of building blocks; under our main construction, each handle of each $X_I$ 
will be a product of such blocks. 

\item \textsection\ref{S:Examples2} describes further examples of $X_I$ under our construction, each featuring a new complication in its handle decomposition.

\item \textsection\ref{S:Cutoff} proves several combinatorial facts about our main construction. In particular, \textsection\ref{S:DisjointCover} proves that $T^n=\bigcup_{i\in\Z_k}X_i$, and \textsection\ref{S:EXI} establishes a closed expression (\ref{E:XI}) for arbitrary $X_I$. Also, \textsection\ref{S:CombCor} establishes two combinatorial corollaries, which may be of independent interest.

\item \textsection\ref{S:MainGen} describes a handle decomposition of arbitrary $X_I$ from our main construction, confirms the details of this decomposition, shows that the central intersection $\bigcap_{i\in\Z_k}X_i$ is a closed $k$-manifold, and puts everything together to prove Theorem \ref{T:Main}

\item \textsection\ref{S:Cube} extends Theorem \ref{T:Main} to certain cubulated manifolds.  

\item Appendix 1 features tables, several detailing follow-up examples for the complications introduced in \textsection\textsection\ref{S:Motivation} and \ref{S:Examples2}, others detailing aspects of the handle decomposition described in \textsection\ref{S:Handles}. 

\item Appendix 2 describes four other ways one might try to multisect $T^n$. 

\end{itemize}

{\bf Thank you} to Mark Brittenham, Charlie Frohman, Hugh Howards, Peter Lambert-Cole, and Maggie Miller for helpful discussions.  Thank you to the anonymous referee for numerous suggestions to improve the clarity and exposition of the paper.  Special thank you to Alex Zupan for helpful discussions throughout the project, especially during its early stages, when we collaborated to find efficient trisections of $T^4$ and $T^5$.

\section{Multisections and their efficiency}\label{S:Smooth}

In this section, we describe a way of obtaining a (PL) handle decomposition of a manifold given a multisection (see Proposition \ref{P:SmoothBoth}), and we deduce, with the exception of 2-manifolds, that no multisection has efficiency greater than 1 (see Corollary \ref{C:SmoothEfficient2}).  
We begin, however, by describing examples of multisections in arbitrary dimension.

\subsection{Simple examples of multisections}

\begin{example}
For $n=2k-1$, the $n$-sphere 
\[S^n=\partial \prod_{i=0}^{k-1}D^2=\bigcup_{i=0}^{k-1}
\left(
\prod_{j=0}^{i-1}D^2
\times S^1\times
\prod_{j=i+1}^{k-1}D^2
\right)
\]
admits a multisection in which each 
\[X_i=
\prod_{j=0}^{i-1}D^2
\times S^1\times
\prod_{j=i+1}^{k-1}D^2
\]
is an $n$-dimensional 1-handlebody of genus 1. In dimension 3, this is the genus 1 Heegaard splitting of $S^3=D^2\times D^2$ with central surface $S^1\times S^1$. In arbitrary dimension $n$, the central intersection is the $k$-torus $\prod_{j=0}^{k-1}S^1$, and more generally, for each $I\subset \Z_k$ with $1\leq |I|=\ell\leq k-1$, the intersection
\[X_I=\bigcap_{j\in I}X_i=
\prod_{j=0}^{k-1}\left.\begin{cases}
S^1&j\in I\\
D^2&j\notin I \\
\end{cases}\right\}\cong \prod_{j=0}^{\ell-1} S^1\times \prod_{j=\ell}^{k-1}D^2\cong T^\ell\times D^{2(k-\ell)}
\]
is a thickened $\ell$-torus.
In dimension 5, Lambert-Cole--Miller use this construction and a second trisection of $S^5$, whose central intersection is a 3-sphere rather than a 3-torus, to show that, unlike Heegaard splittings of 3-manifolds and trisections of 4-manifolds, trisections of a given 5-manifold need not be stably equivalent \cite{lcm}.
\end{example}

\begin{example}\cite{rt}
Using homogeneous coordinates $[z_0:\cdots:z_{k-1}]$ on $\CP^{k-1}$, one can define a multisection by 
\[X_i=\left\{[z_0:\cdots:z_{k-1}]~\left|~|z_i|\geq |z_j|\text{ for }j=0,\hdots,{k-1}\right\}\right..\]
Then each $X_I$ with $|I|=\ell$ is related by permutation to a thickened torus
\begin{equation*}
\begin{split}
\bigcap_{i=0}^{\ell-1}X_i&=\left\{[1:z_1:\cdots:z_{k-1}]~\left|~\begin{matrix}|z_j|=1\text{ for }j=1,\hdots,\ell-1,\\ |z_j|\leq 1\text{ for }j=\ell,\hdots,k-1\end{matrix}\right.\right\}\\
&\cong T^{\ell-1}\times D^{2(k-\ell)}.
\end{split}
\end{equation*}
In particular, the central intersection is the $k$-torus
\[\left\{[1:z_1:\cdots:z_{k-1}]:~|z_1|=\cdots=|z_{k-1}|=1\right\}.\]
These symmetric multisections are also efficient, since each $X_i$ has genus 0.
\end{example}

\subsection{General properties of multisections}

\begin{prop}\label{P:4}
Let $Z_i$ be an $n$-dimensional $h_i$-handlebody, $i=1,2$, and let $\phi:Y_1\to Y_2$ glue compact $Y_i\subset \partial Z_i$, such that $Y_1\cong Y_2$ is an $h$-handlebody. 
Then $Z=Z_1\cup_\phi Z_2$ is an $h'$-handlebody for $h'=\max\{h_1,h_2,h+1\}$.
\end{prop}

\begin{proof}
By taking a regular neighborhood $N$ of $Y=\phi(Y_1)=\phi(Y_2)$ in $Z$,  where $N\equiv Y\times I$
, we may identify $Z\setminus\text{int}(N)$ with $Z_1\sqcup Z_2$, which is a 2-component $h''$-handlebody where $h''=\max\{h_1,h_2\}$.  Then, for each $i$-handle $H\equiv D^i\times{n-1-i}$ in $Y$, $0\leq i\leq h$, we can glue on  
$H\times I$ 
along $\partial(D^i\times I)\times D^{n-1-i}\cong S^i\times D^{n-1-i} $, and so attaching $H\times I$ is the same as attaching an $(i+1)$-handle, where $i+1\leq h+1$.  
\end{proof}

\begin{prop}\label{P:5}
Let $X=\bigcup_{i\in\Z_k}X_i$ be a multisection of a closed manifold of dimension $n=2k-1$ (resp. $n=2k-2$).
Then for each $1\leq j\leq i\leq k-1$:
\[
\bigcup_{t=0}^{j-1}X_t
\cap\bigcap_{t=j}^{i}X_t\]
is a $
(2k+j-i-2)$-dimensional $(i+j)$-handlebody (resp. $(2k+j-i-3)$-dimensional $(i+j-1)$-handlebody).
\end{prop}

\begin{proof}
We address the odd-dimensional case, arguing by lexicographical induction on $(i,j)$. The even-dimensional case follows analogously. When $(i,j)=(1,1)$, the proposition is true by definition, since $X_0\cap X_1$ is a 
2-handlebody. 

Let $(i,j)>(1,1)$. Assume for each $(r,s)<(i,j)$ that $(X_0\cup \cdots\cup X_{s-1})\cap X_s\cap\cdots\cap X_r$  is a $(2k+s-r-2)$-dimensional $(r+s)$-handlebody. 
Let 
\[Z_1=\bigcup_{t=0}^{j-2}X_t\cap\bigcap_{t=j}^iX_t\]
 and 
 \[Z_2=\bigcap_{t=j-1}^{i}X_t,\] so that 
\[
\bigcup_{t=0}^{j-1}X_t
\cap\bigcap_{t=j}^{i}X_t=Z_1\cup Z_2\]
Then, by induction, $Z_1$ is a $(
2k+j-i-2)$-dimensional $(
i+j-2)$-handlebody, and, by the definition of multisection, $Z_2$ 
is a $(
2k+j-i-2)$-dimensional $(i+1-j)$-handlebody. Further,
\[Z_1\cap Z_2=
\bigcup_{t=0}^{j-2}X_t
\cap\bigcap_{t=j-1}^{i}X_t,
\]
which, by induction, is a $(
2k+j-i-3)$-dimensional $(i+j-1)$-handlebody. Therefore, by Proposition \ref{P:4}, $Z_1\cup Z_2$ is a $(2k+j-i-2)$-dimensional $h$-handlebody, where 
\[\pushQED{\qed}h=\max\{i+j-2,i+1-j,i+j\}=i+j.\qedhere\]
\end{proof}

\begin{prop}\label{P:SmoothBoth}
Let $X=\bigcup_{i\in\Z_k}X_i$ be a multisection of a closed manifold of dimension $n=2k-1$ (resp. $n=2k-2$). Then $X$ admits a handle decomposition in which each $X_j$ contributes only $r$-handles for $r\leq 2j+1$ (resp. $r\leq 2j$).
\end{prop}

\begin{proof}
We address the odd-dimensional case; the even-dimensional case follows analogously.  Arguing by induction on $i$, we will show that $X_0\cup \cdots\cup X_i$ admits a handle decomposition in which each $X_j$ contributes only $r$-handles for $r\leq 2j+1$. The base case is trivial.  For the induction step, consider 
\[(X_0\cup \cdots\cup X_{i-1})\cup_{(X_0\cup \cdots\cup X_{i-1})\cap X_{i}}X_{i}.\]
By induction, $X_0\cup \cdots\cup X_{i-1}$ admits a handle decomposition in which each $X_j$ contributes only $r$-handles for $r\leq 2j+1$. Extend this to the required handle decomposition of $X_0\cup \cdots\cup X_{i}$ as follows. Let $N$ be a collared neighborhood of $(X_0\cup \cdots\cup X_{i-1})\cap X_{i}$ in $X_{i}$. As in the proof of Proposition \ref{P:4}, first construct the disjoint union 
\[(X_0\cup \cdots\cup X_{i-1})\sqcup (X_i\setminus\text{int}(N)),\] thereby contributing 0- and 1-handles to $X_i$, as $X_i\setminus\text{int}(N)$ is PL homeomorphic to $X_i$; second, glue in $N$, thereby contributing $r$-handles for $r=1,\hdots,2i+1$, since $(X_0\cup \cdots\cup X_{i-1})\cap X_i$ is a $2i$-handlebody by Proposition \ref{P:5}.
\end{proof}



\subsection{Efficiency of multisections}

Next, we consider the efficiency of multisections in light of Proposition \ref{P:SmoothBoth}. Recall Definition \ref{D:Efficiency}.

\begin{prop}\label{P:SmoothEfficient1}
In dimension $n\neq2$, any multisection $X=\bigcup_{i\in\Z_k} X_i$ obeys
\[\min_{i\in\Z_k}g(X_i)\geq \textnormal{rank }\pi_1(X).\]
\end{prop}

\begin{proof}
Given a multisection of $X$, label the pieces so that $g(X_{k-1})\leq g(X_i)$ for all $i$.  Construct a handle structure on $X$ as guaranteed by Proposition \ref{P:SmoothBoth}.  All the $n$- and $(n-1)$-handles are in $X_{k-1}$, since $n\neq 2$. Flip $X$ upside down.  Now all the 0- and 1-handles are in $X_{k-1}$, so 
\[\pushQED{\qed}
\text{rank }\pi_1(X)\leq\text{rank }\pi_1(X_{k-1})=g(X_{k-1})=\min_{i\in\Z_k}g(X_i).\qedhere\]  
\end{proof}

\begin{cor}\label{C:SmoothEfficient2}
In any dimension $n\neq 2$, no multisection of any manifold has efficiency greater than 1, and in any efficient multisection $X=\bigcup_{i\in \Z_k}X_i$, all $X_i$ have the same genus, $g(X_i)= \textnormal{rank }\pi_1(X)$.
\end{cor}

\section{Motivating examples}\label{S:Motivation}

In this section, we describe our multisections of $T^3$, $T^4$, and $T^5$ in detail.  We also establish notation that will be used throughout the rest of the paper.

\subsection{Intuitive approach to $T^3$, $T^4$, and $T^5$}

\begin{figure}
\begin{center}
\includegraphics[width=\textwidth]{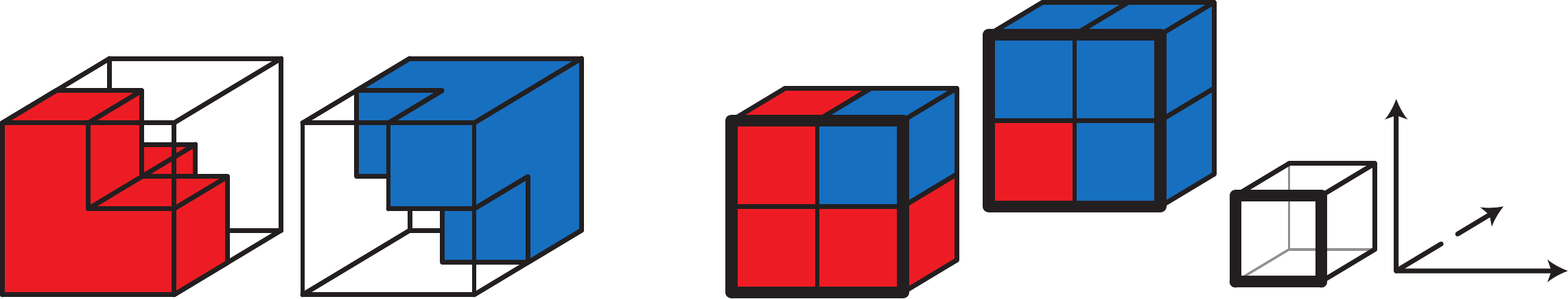}
\caption{A Heegaard splitting of $T^3$.}
\label{Fi:T3Intro}
\end{center}
\end{figure}

Figure \ref{Fi:T3Intro} illustrates an efficient Heegaard splitting of the 3-torus, which suggests viewing $T^3$ as $(\R/2\Z)^3$; then the splitting is determined by a partition of the eight unit cubes with vertices in the lattice $(\Z/2\Z)^3$.
Moreover, {\it this} partition satisfies two symmetry properties: first, the permutation action of $S_3$ on the indices in $T^3$ fixes each piece of the splitting, and second, the $\Z_2$ translation action along the main diagonal of $T^3$ switches the two pieces: $X_i+(1,1,1)=X_{i+1}$.

\begin{figure}
\begin{center}
\includegraphics[width=\textwidth]{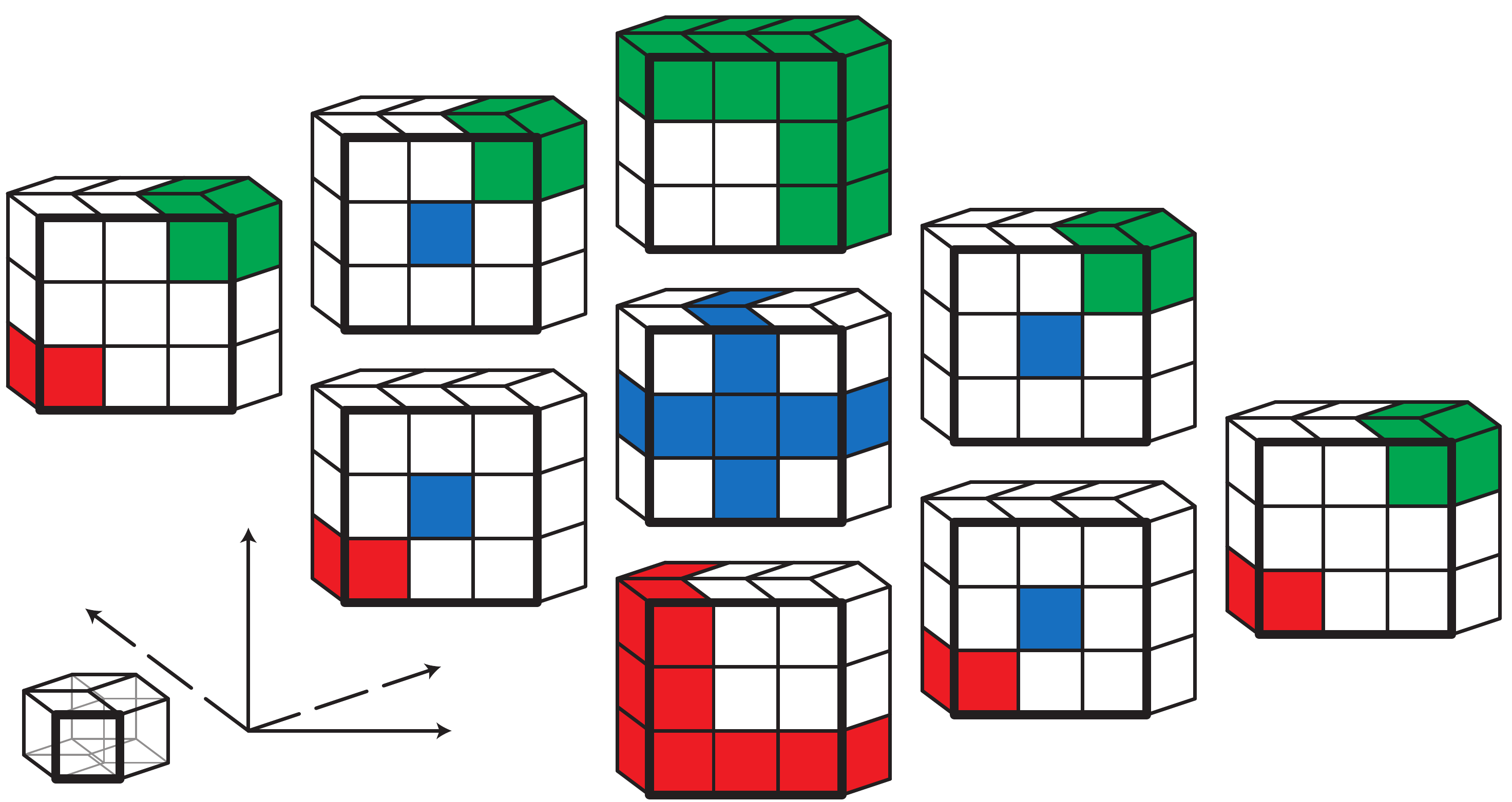}
\caption{Start partitioning the subcubes of $T^4=(\R/3\Z)^4$ like this, giving three 4-dimensional 1-handlebodies.}
\label{Fi:T4Start}
\end{center}
\end{figure}

How might one construct efficient trisections of $T^n$, $n=4,5$, with symmetry properties analogous to Figure \ref{Fi:T3Intro}'s splitting of $T^3$? To begin, one might view these $T^n$ as $(\R/3\Z)^n$---rather than, say, $(\R/2\Z)^n$, because we seek a trisection rather than a splitting---and seek an appropriate partition of the $3^n$ unit cubes with vertices in the lattice $(\Z/3\Z)^n$.  From now on, for brevity, we will refer to these unit cubes as {\it subcubes} of $T^n$.

To start forming this partition, one might assign each subcube $[i,i+1]^n$ to $X_i$ (because of the translation action).  Next, one might assign those subcubes of the forms $[i,i+1]^{n-1}[i+1,i+2]$ and $[i,i+1]^{n-1}[i-1,i]$ to $X_i$ as well, and extend these assignments using the permutation action on the indices.  
At this point, each $X_i$ is indeed an $n$-dimensional 1-handlebody, and so the rest of the partition should be constructed in a way that preserves this fact, while also giving rise to the needed intersection properties.  
Figure \ref{Fi:T4Start} illustrates this intermediate stage in the case of $T^4$.\footnote{All combinatorial data conveyed in Figures \ref{Fi:T4Start}--\ref{Fi:T4Pieces} comes from the arrangements of the nine $3\times 3$ squares outlined in bold; beyond this, the style of the illustration reflects the fact that each pictured subcube is a 4-cube.  A model 4-cube is also drawn, next to coordinate axes.  The solid axes represent directions in which abutting subcubes are shown in contact with each other (understanding that the interval that appears as $[0,3]$ actually represents the circle $\R/3\Z$); the dashed axes represent directions in which abutting subcubes align at a distance in the figure.  
Similarly, Figure \ref{Fi:T3Intro} shows a model 3-cube and coordinate axes, Figure \ref{Fi:T5Intro} a model 5-cube and coordinate axes.} 

\begin{figure}
\begin{center}
\includegraphics[width=\textwidth]{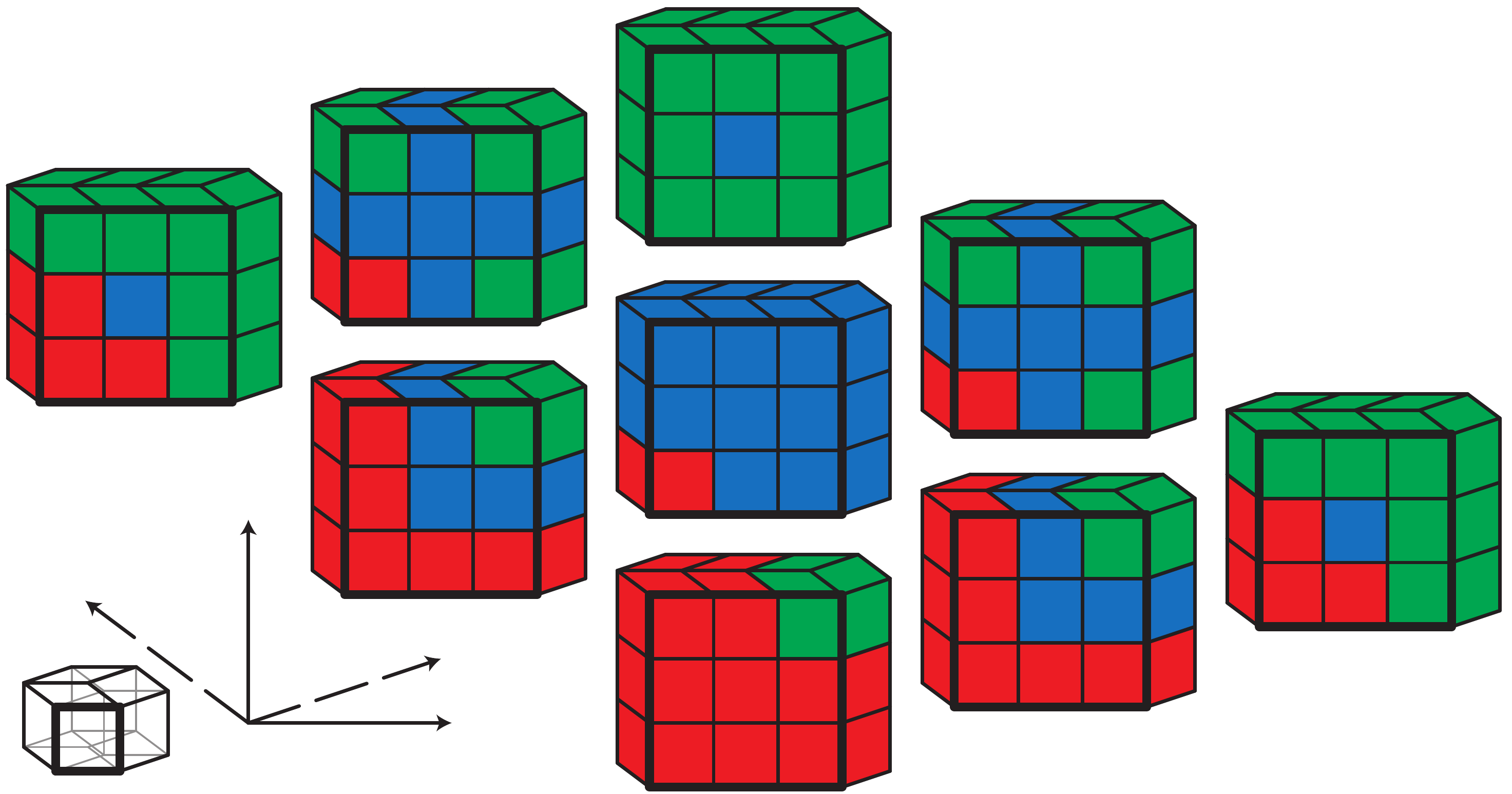}
\caption{Partitioning the $3^4$ subcubes of $T^4=(\R/3\Z)^4$ like this gives a symmetric efficient trisection.}
\label{Fi:T4Pieces}
\end{center}
\end{figure}

\begin{figure}
\begin{center}
\labellist
\small\hair 4pt
\pinlabel {$t\in(0,1)$} [r] at 300 720
\pinlabel {$t\in(1,2)$} [t] at 100 800
\pinlabel {$t\in(2,3)$} [t] at 50 905
\endlabellist
\includegraphics[width=\textwidth]{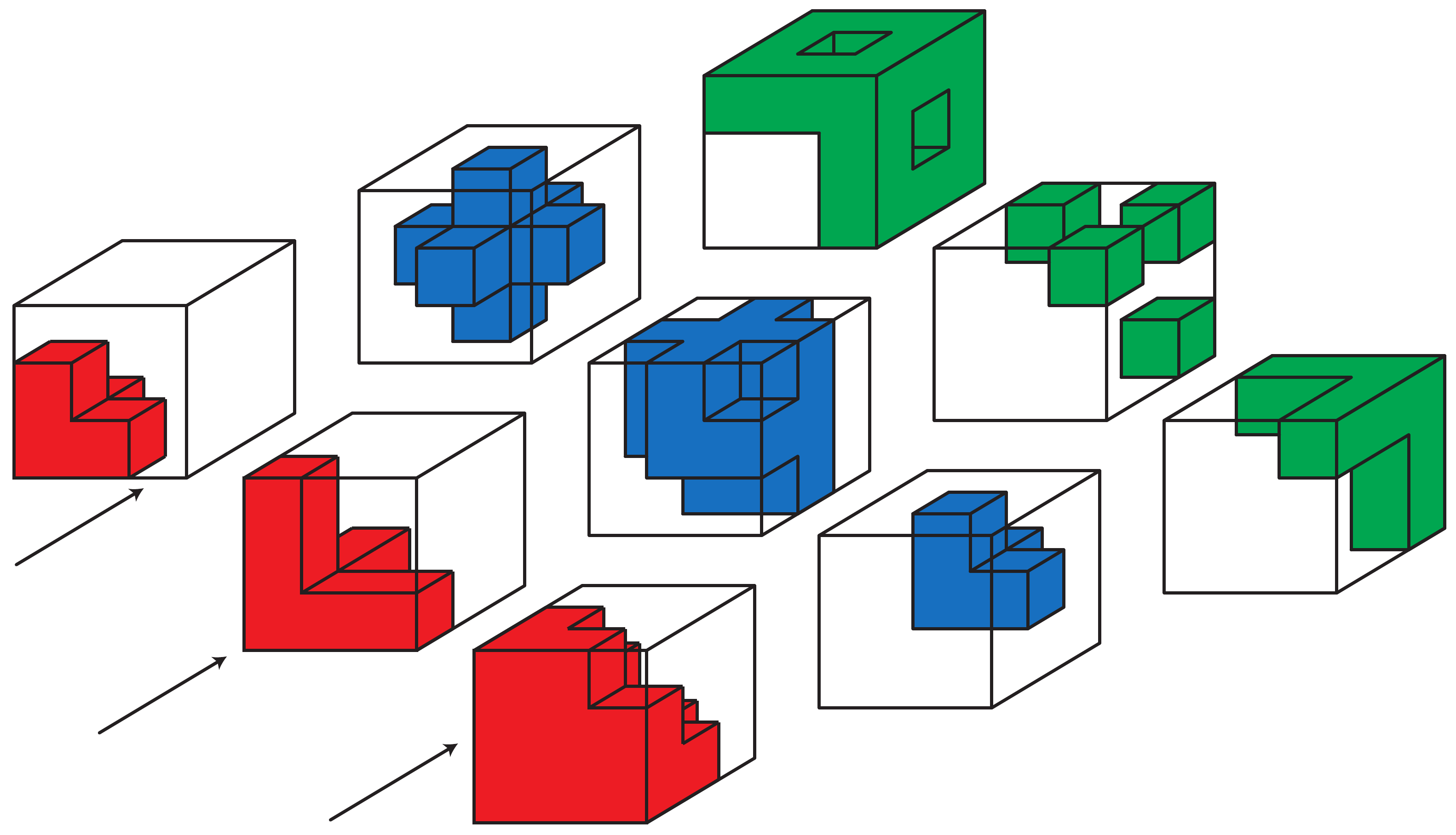}
\caption{In the multisection of $T^4$ from Figure \ref{Fi:T4Pieces}, each slice $T^3\times \{t\}$, $t\in(\R/3\Z)\setminus\Z_3$, intersects $X_0,X_1,X_2$ like this.}
\label{Fi:T4Slices}
\end{center}
\end{figure}

For $T^4$, the symmetry properties imply that the remaining partition is determined by the assignments of the subcubes $[0,1]^2[1,2][2,3]$ and $[0,1]^2[1,2]^2$.  Assigning both subcubes to $X_0$ and extending symmetrically gives the decomposition of $T^4$ illustrated in Figures \ref{Fi:T4Pieces} and \ref{Fi:T4Slices}.  Section \ref{S:T4} will confirm that this decomposition is indeed a trisection.

A similar approach leads to the decomposition of $T^5$ shown in Figure \ref{Fi:T5Intro}.  Section \ref{S:T5} will confirm that this, too, is a trisection.

\begin{figure}
\begin{center}
\includegraphics[width=\textwidth]{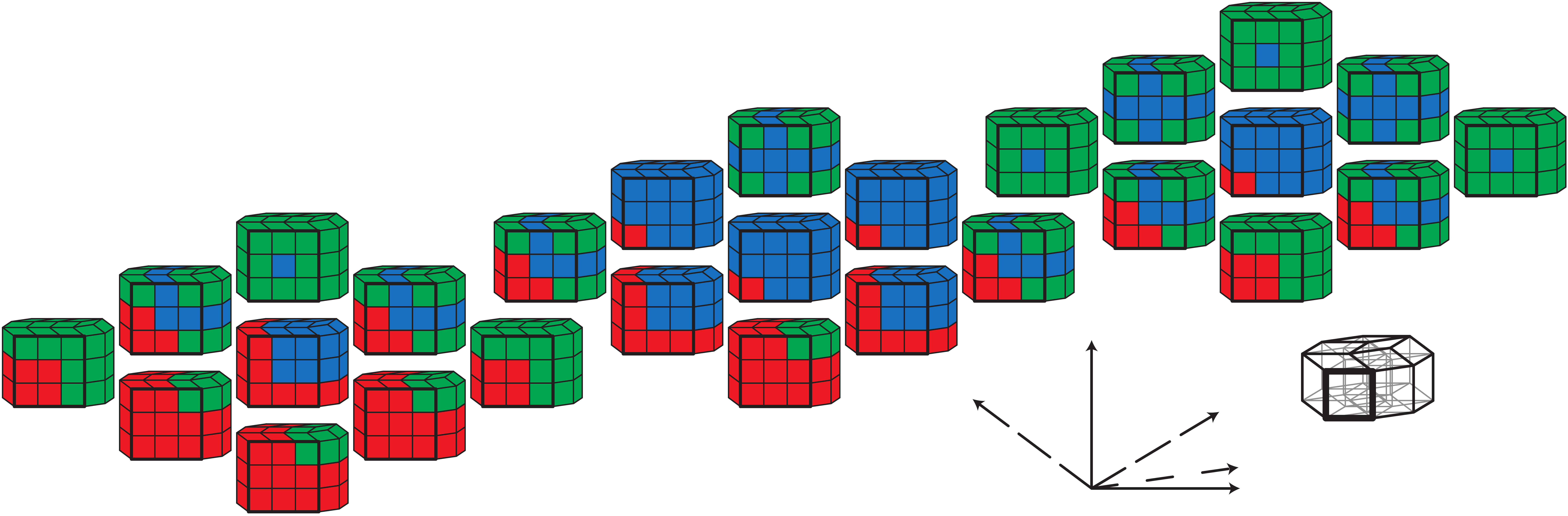}
\caption{Partitioning the $3^5$ subcubes of $T^5=(\R/3\Z)^5$ like this gives a symmetric efficient trisection.}
\label{Fi:T5Intro}
\end{center}
\end{figure}


\subsection{Notation}\label{S:NotationEx}

\begin{notation}\label{N:cut}
Let $X,Y\subset Z$ be compact subspaces of a topological space.  Denote ``$X$ cut along $Y$" by $X\setminus\setminus Y$.  In every example where we use this notation, $X\setminus\setminus Y$ equals the closure of $X\setminus Y$ in $Z$. (The general construction is somewhat more complicated.)
\end{notation}

Given $n=2k-1,2k-2$, view the $n$-torus $T^n$ as $(\R/k\Z)^n$. Let $S_n$ denote the permutation group on $n$ elements.
\begin{notation}\label{N:xsigma}
Given $\x=(x_1,\hdots,x_n)\in T^n$ and $\sigma\in S_n$,  denote 
\[\x_\sigma=(x_{\sigma(1)},\hdots,x_{\sigma(n)}).\]
Also, given $U\subset T^n$ and $\vv\in T^n$, denote
\[U+\vv=\{\uu+\vv:~\uu\in U\}.\]
\end{notation}
The symmetric group $S_n$ acts on $T^n$ by permuting the indices, $\sigma:\x\mapsto\x_\sigma$. Because we are interested in subsets of $T^n$ which are fixed by this action:
\begin{notation}\label{N:lla}
For any subset $U\subset T^n$, denote
\begin{equation*}\label{E:langle}
\lla U \rra=\left\{\x_\sigma:\x\in U,~\sigma\in S_n\right\}\subset T^n.
\end{equation*}
\end{notation}
Note, for any $U\subset T^n$, that $\lla U\rra$ is fixed by the action of $S_n$ on $T^n$.  We can state our main result explicitly:

\begin{T:Main}
For $n=2k-1$, the $n$-torus $T^n=(\R/k\Z)^n=[0,k]^n/\sim$ admits an efficient multisection $T^n=\bigcup_{i\in\Z_k}X_i$ defined by
\begin{equation}\label{E:X0ThmIntro}
\begin{split}
X_0&=\lla[0,1]^2\cdots[0,k-1]^2[0,k]\rra,\\
X_i&=X_0+(i,\hdots,i),~i\in\Z_k.
\end{split}
\end{equation}
\end{T:Main}

By construction, the decomposition is symmetric with respect to the permutation action on the indices and the translation action on the main diagonal.

Anticipating the concrete and (somewhat) low-dimensional nature of the examples in \textsection\textsection\ref{S:Motivation}, \ref{S:Examples2} and Appendix 1, we give the first few intervals $[i,i+1]$, $i\in\Z_k$, special notations:

\begin{notation}\label{N:alpha}
Denote
\begin{equation*}\label{E:alphabeta}
[0,1]=\alpha,~[1,2]=\beta,~[2,3]=\gamma,[3,4]=\delta,
[4,5]=\ep,~[5,6]=\zeta,~[6,7]=\eta.
\end{equation*}
\end{notation}
To further abbreviate our notation, we often omit $\times$ symbols and use exponents to denote repeated factors. For example, we can describe the two pieces of the Heegaard splitting of $T^3$ from Figure \ref{Fi:T3Intro} like this:
\[X_0=\alpha^3\cup\alpha^2\beta\cup\alpha\beta\alpha\cup\beta\alpha^2,\hspace{1in}
X_1=\beta^3\cup{\beta^2}\alpha\cup\beta\alpha\beta\cup\alpha{\beta^2}.\]
Using Notation \ref{N:lla}, we can further abbreviate this notation:
\begin{align*}X_0&=\alpha^3\cup\lla\alpha^2\beta\rra&
X_1&=\beta^3\cup\lla\alpha{\beta^2}\rra\\
&=\lla\alpha^2[0,2]\rra&&
=\lla[0,2]{\beta^2}\rra.
\end{align*}
We often omit the braces around singleton factors. For example, in $T^3$:
\begin{align*}
X_0\cap X_1&=\lla[0,1]\times[1,2]\times\{0\}\rra\cup\lla[0,1]\times[1,2]\times\{1\}\rra\\
&=\lla\alpha\beta0\rra\cup\lla\alpha\beta1\rra.
\end{align*}
We also extend Notation \ref{N:lla} in the way suggested by the following example:
\[\lla0\alpha\rra{\beta^2}=\left(\{0\}\times \alpha\times\beta\times\beta\right)\cup\left(\alpha\times\{0\}\times\beta\times\beta\right).\]
More precisely, if we decompose $T^n$ as a product $T^n=T^{n_1}\times\cdots\times T^{n_p}$ and $U_i\subset T^{n_i}$ for $i=1,\hdots, p$, then
\[\lla U_1\rra\cdots\lla U_p\rra=\left\{(\x^1_{\sigma_1},\x^2_{\sigma_2},\hdots,\x^p_{\sigma_p}):~\x^i\in T^{n_i},~\sigma_i\in S_{n_i},~i=1,\hdots,p\right\}\]
where, extending Notation \ref{N:xsigma} and denoting $\x^i=(x^i_1,\hdots,x^i_{n_i})$, each
\[\x^i_{\sigma_i}=\left(x^i_{\sigma_i(1)},\hdots,x^i_{\sigma_i(n_i)}\right).\]
Starting in dimension 7,  some handle decompositions will require subdividing unit subintervals $\alpha,\beta,\gamma,\delta,\hdots$ into halves or thirds. Anticipating this:
\begin{notation}\label{N:alpha+}
Denote
\[\textstyle
\alpha^-=\left[0,\frac{1}{2}\right],~\alpha^+=\left[\frac{1}{2},1\right],\hdots,\left[\eta^-\right]=\left[6,\frac{13}{2}\right],\eta^+=\left[\frac{13}{2},7\right]\]
and
\[
\textstyle
\alpha^-_3=\left[0,\frac{1}{3}\right],\alpha^\circ_3=\left[\frac{1}{3},\frac{2}{3}\right],~\alpha^+_3=\left[\frac{2}{3},1\right],\hdots,
\eta^\circ_3=\left[\frac{19}{3},\frac{20}{3}\right],~\eta^+_3=\left[\frac{20}{3},7\right].\]
\end{notation}

Because of the symmetry of our main construction under the $\Z_k$ translation action on $T^n$, it will suffice, when considering $X_I$ from that construction, to allow $I$ to be arbitrary {\it only up to cyclic permutation}. In order to utilize this convenience: 

\begin{notation}\label{N:i}
Given $I\subset \Z_k$ with $|I|=\ell>0$, denote $X_I=\bigcap_{i\in I}X_i$, and denote $I=\{i_s\}_{s\in\Z_{{\ell}}}$
such that 
\[0\leq i_0<i_1<\cdots<i_{\ell-1}\leq k-1.\]
\end{notation}

\begin{definition}\label{D:simple}
Let $I=\{i_s\}_{s\in\Z_\ell}$ as in Notation \ref{N:i}. For each $r\in\Z_{{\ell}}$, define $I^r=\{i+r:~i\in I\}\subset \Z_k$. Denote each $I^r=\{i^r_s\}_{s\in\Z_{{\ell}}}$ with $0\leq i^r_{0}<i^r_{1}<\cdots<i^r_{\ell-1}\leq k-1$.  
Say that $I$ is {\bf simple} if
, for each $r\in\Z_\ell$, we have $I\leq I^r$ under the lexicographical ordering of their elements, i.e. if each $I_r\neq I$ has some $s\in\Z_{{\ell}}$ with $i_t=i^r_{t}$ for each $t=0,\hdots,s-1$ and $i_s<i^r_{s}$.
\end{definition}

\begin{notation}\label{N:T}
Given simple $I=\{i_s\}_{s\in\Z_{\ell}}\subsetneqq \Z_k$ as in Notation \ref{N:i}, define 
\begin{equation*}\label{E:T}
T=\{s\in\Z_\ell:~i_s-1\notin I\}.
\end{equation*}
Denote $T=\{t_r\}_{r\in\Z_m}$ with $0=t_0<\cdots< t_m<\ell$ (see Observation \ref{O:Simple}). For each $r\in\Z_m$, denote $I_r=\{i_{t_r},\hdots,i_{t_{r+1}-1}\}$.  Then 
\begin{equation*}\label{E:IBlocks}
I=I_1\sqcup\cdots \sqcup I_m,
\end{equation*} 
and for each $r=0,\hdots, m-1$, we have $|I_r|=\max I_r+1-\min I_r$ (each block $I_r$ is comprised of consecutive indices) and $\min I_{r+1}\geq \max I_r+2$ (the blocks are nonconsecutive).  

Given $i_*\in I$ (denoted specifically as $i_*$), denote the block $I_r$ containing $i_*$ by $I_*$.
\end{notation}

\begin{convention}\label{Conv:ISimple}
Throughout, reserve the notations $n$, $k$, $\alpha,\hdots$, $\eta$, $\alpha^-,\hdots$, $\eta^+$, $\alpha^-_3,\hdots$, $\eta^+_3$, $I$, $X_I$, $\ell$, $T$, and $m$ for the way they are used in Notations \ref{N:alpha}-\ref{N:T}; assume, unless otherwise stated, that $I\subset\Z_k$ is simple; 
and reserve, for any $s\in\Z_\ell$ or $r\in\Z_m$, the notations $i_s$, $t_r$, $I_r$, $i_*$, and $I_*$ for the way they are used in Notations \ref{N:i} and \ref{N:T}.  
\end{convention}


\begin{obs}\label{O:Simple}
Given $I\subsetneqq\Z_k$, we have $i_0= 0$, $i_{\ell-1}\leq k-2$, and $|I_0|\geq |I_r|$ for each $r\in\Z_m$; if $|I_0|= |I_r|$, then $|I_1|\geq |I_{r+1}|$.
\end{obs}

Given $I\subset\Z_k$ and $s\in\Z_\ell$, denote
\[(i_1,\hdots,\widehat{i_s},\hdots,i_\ell)=(i_1,\hdots,i_{s-1},i_{s+1},\hdots,i_\ell)\subset T^{\ell-1}.\]
We now have enough notation to describe a closed formula for the $X_I$ coming from our main construction (\ref{E:X0ThmIntro}):

\begin{L:XI}
Given nonempty $I\subseteqq\Z_k$, $X_I$ is given by:
\begin{equation}\label{E:XI}
\bigcup_{i_*\in I}\lla(i_1,\hdots,\widehat{i_*},\hdots,i_\ell)\prod_{r\in\Z_{\ell}}[i_{r},i_{r}+1]^2\cdots[i_{r},i_{r+1}-1]^2{[i_{r},i_{r+1}]}
\rra.
\end{equation}
In particular,
\begin{equation}\label{E:XZk}
\bigcap_{i\in \Z_k}X_i=\bigcup_{i_*\in \Z_k}\lla(0,\hdots,\widehat{i_*},\hdots,k-1)\prod_{i\in\Z_k}[i,i+1]\rra.
\end{equation}
\end{L:XI}
We will prove Lemma \ref{L:XI} in \textsection \ref{S:EXI}.

\subsection{Trisection of $T^4$}\label{S:T4}

\begin{figure}
\begin{center}
\includegraphics[width=\textwidth]{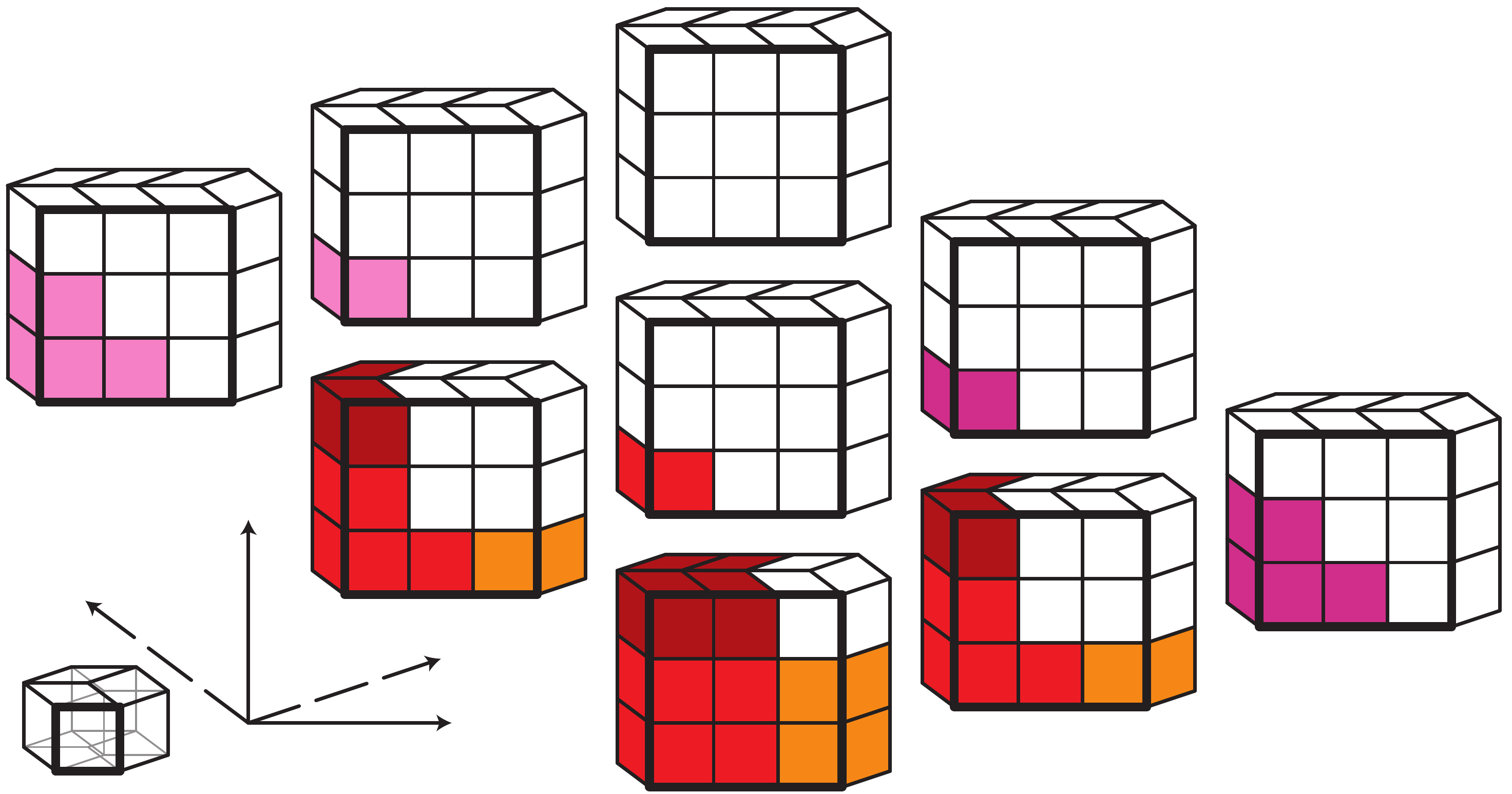}
\caption{A handle decomposition of $X_0$ in Figure \ref{Fi:T4Pieces}'s trisection of $T^4$. Each of the five handles has a different color. 
}
\label{Fi:T4Decomp}
\end{center}
\end{figure}

The decomposition of $T^4$ from Figure \ref{Fi:T4Pieces} is given by
\begin{equation}
\label{E:T4A}
\begin{split}
X_0=\lla \alpha^2[0,2]\violet{[0,3]}\rra=&\lla\alpha^4\rra\cup\lla\alpha^3{\beta}\rra\cup\lla\alpha^3{\gamma}\rra\cup\lla\alpha^2\beta^2\rra\cup\lla\alpha^2{\beta}{\gamma}\rra\\
X_i=&X_0+(i,i,i,i).
\end{split}
\end{equation}
It is evident from Figure \ref{Fi:T4Pieces} that $X_0\cup X_1\cup X_2=T^4$.
Also, $I=\{0\}$ and $I=\{0,1\}$ are the only proper subsets of $\{0,1,2\}$ which are simple.
Therefore, in order to check that (\ref{E:T4A}) determines a trisection of $T^4$, it suffices to prove that $X_0$ is a 4-dimensional 1-handlebody and $X_0\cap X_1$ is a 3-dimensional 1-handlebody with $\partial (X_0\cap X_1)=X_0\cap X_1\cap X_2$.

Indeed, Figure \ref{Fi:T4Decomp} shows a handle decomposition of $X_0$ in which 
$\lla\Navy{\alpha^2[0,2]^2}\rra$ is a 0-handle and  $\lla{\alpha^2[0,2]}{\gamma}\rra$ supplies four 1-handles, each a permutation of $\lla \Navy{\alpha^2[0,2]}\rra
{\gamma}$. More precisely, each 1-handle is given, in terms of some permutation $\sigma\in S_4$ (using Notation \ref{N:xsigma}), by 
\begin{equation}
\label{E:T41Handles}
\left\{\vec{x}_\sigma:~\vec{x}\in \lla \Navy{\alpha^2[0,2]}\rra
{\gamma}\right\}.
\end{equation}
\begin{figure}
\begin{center}
\includegraphics[width=.75\textwidth]{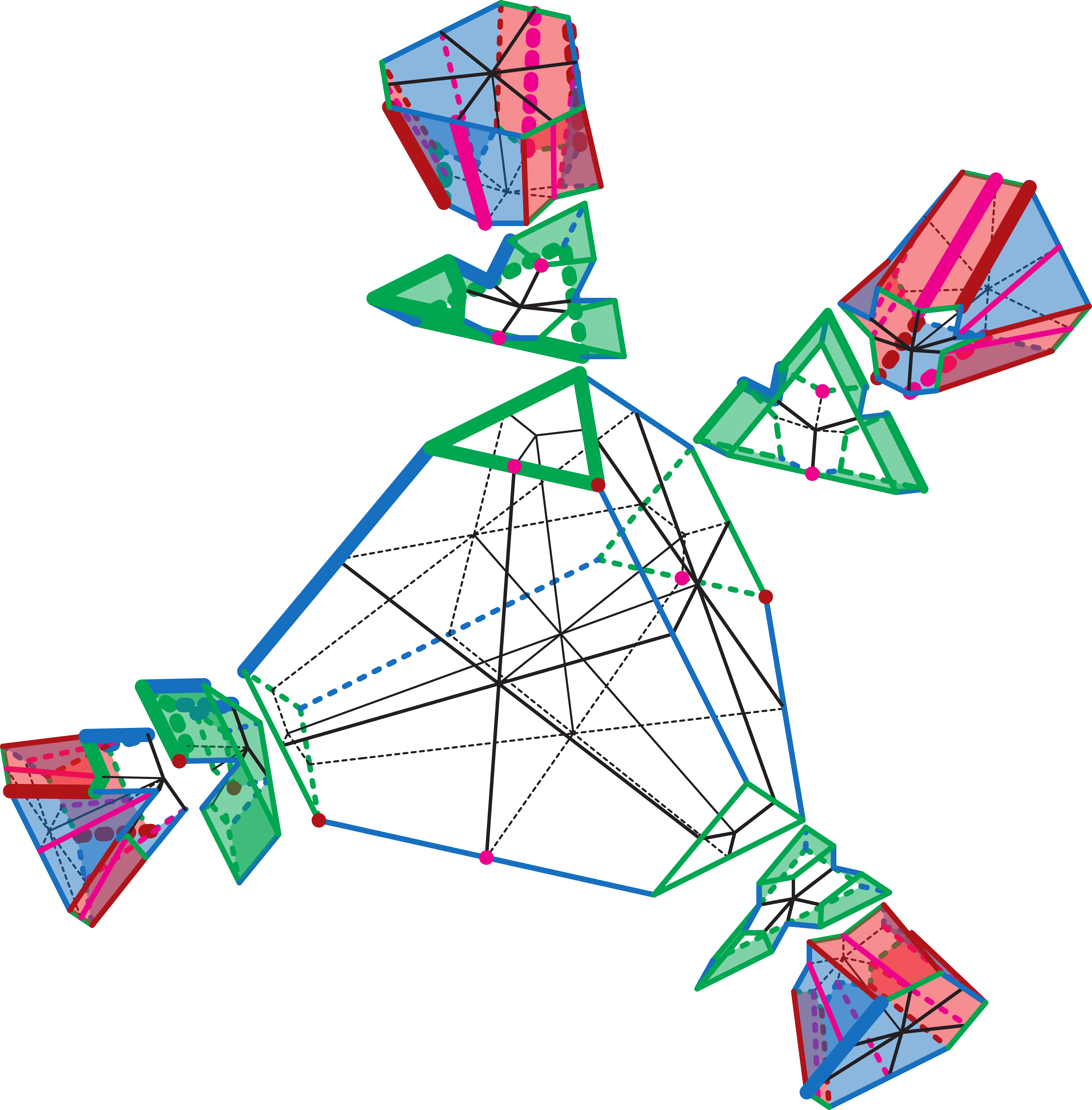}
\caption{A handle decomposition of $X_0\cap X_1$ in our trisection of $T^4$. The trisection diagram on $\partial (X_0\cap X_1)=X_0\cap X_1\cap X_2=\FG{\lla\alpha\beta02\rra}\cup\NavyBlue{\lla\alpha\gamma12\rra}\cup\red
{\lla\beta\gamma01\rra}$ has two types of red curves; one of each is bold. Same with blue and green.
}
\label{Fi:T4Actual}
\end{center}
\end{figure}
Now consider 
\begin{equation}\label{E:X014}
X_0\cap X_1=\lla\alpha1\beta{[1,3]}\rra\cup\lla0\alpha{\beta^2}\rra.
\end{equation}
We claim that this is a 3-dimensional 1-handlebody in which:
\begin{itemize}
\item $Y_1=\lla\Navy{\alpha}1\Navy{\beta^2}\rra$ is the 0-handle;
\item $Y_2=\lla0\alpha{\beta^2}\rra$ gives six 1-handles, all permutations of $Y_2^*=\lla0
{\alpha}\rra\Navy{\beta^2}$; 
\item $Y_3=\lla\alpha1\beta{\gamma}\rra$ gives four 1-handles, all permutations of $Y_3^*=\lla\Navy{\alpha}1\Navy{\beta}\rra
{\gamma}$.
\end{itemize}  
Figure \ref{Fi:T4Actual} shows this decomposition of $X_0\cap X_1$: 
\begin{itemize}
\item The shape in the center (which looks like a truncated tetrahedron) is the 0-handle $\lla\Navy{\alpha}1\Navy{\beta^2}\rra$, comprised of 12 cubes, each a permutation of $\alpha1\beta^{{2}}$ (c.f. (\ref{E:T41Handles}) and the paragraph before it). The interior lattice point is $(1,1,1,1)$, and each triangular-looking face is a permutation of $0\lla1{\beta^2}\rra$ (again, c.f. (\ref{E:T41Handles}) ). Each blue segment on $\partial \lla \alpha1\beta^{{2}}\rra$ is a permutation of $\lla\alpha1\rra2^2$.
\item Each of the four three-pronged pieces is a permutation of $0\lla\alpha{\beta^2}\rra$, glued to the 0-handle along $0\lla1{\beta^2}\rra$. The twelve cubes comprising these pieces are then glued in pairs:  $0\alpha{\beta^2}$ and $\alpha0{\beta^2}$, e.g., meet along the face $00{\beta^2}$, and the other pairs are permutations of this.  The union of each pair of cubes, (a permutation of) $Y_2^*=\lla 0
{\alpha}\rra\Navy{\beta^2}$, is a 1-handle which is glued to the 0-handle along (the corresponding permutation of) $\lla01\rra{\beta^2}$. Note that $Y_2^*$ intersects other permutations of $Y_2^*$, but only within $Y_2^*\cap Y_1$.  Therefore, attaching $Y_2^*$ to $Y_1$ amounts to attaching six 1-handles.
\item Each of the four remaining pieces is a permutation of $Y_3^*=\lla\Navy{\alpha}1\Navy{\beta}\rra
{\gamma}$ and attaches to $Y_1$ and $Y_2$, respectively, along (the corresponding permutations of) $\lla\alpha1\beta\rra2\subset\lla\alpha1\beta^2\rra$ and $\lla\alpha1\beta\rra0\subset\lla\alpha{\beta^2}\rra0$.
\end{itemize}
%
%
For emphasis, here are some key details of this decomposition which will be instructive toward the odd-dimensional case (we will justify some of these details in \textsection\ref{S:StarShaped}):
\begin{equation*}\label{E:4Y1}
Y_1=Y_1^*=\lla\Navy{\alpha}1\Navy{\beta^{{2}}}\rra\cong D^3,
\end{equation*}
so $Y_1$ is a 0-handle;
\begin{equation*}
\begin{split}
Y_2^*&=\lla0
{\alpha}\rra\Navy{\beta^2}\cong D^1\times D^2 \text{ and}\\
Y_2^*\cap (Y_2\setminus\setminus Y_2^*)\subset Y_2^*\cap Y_1&=\left(\partial\lla0
{\alpha}\rra\right)\times \Navy{\beta^2}
=\lla01\rra\Navy{\beta^2}\cong S^0\times D^2,\\
\end{split}
\end{equation*}
so attaching $Y_2$ to $Y_1$ amounts to attaching a collection of 1-handles; and
\begin{equation*}
\begin{split}
Y_3^*&=\lla\Navy{\alpha}1\Navy{\beta}\rra
{\gamma}\cong D^2\times D^1 \text{ and}\\
Y_3^*\cap (Y_3\setminus\setminus Y_3^*)\subset Y_3^*\cap (Y_1\cup Y_2)
&=\lla\Navy{\alpha}1\Navy{\beta}\rra\times\partial
{\gamma}
\cong D^2\times S^0,\\
\end{split}
\end{equation*}
so attaching $Y_3$ to $Y_1\cup Y_2$ amounts to attaching a collection of 1-handles. Thus, $X_0\cap X_1$ is a 4-dimensional 1-handlebody. 
Note in Figure \ref{Fi:T4Actual} that $\partial (X_0\cap X_1)$ is the central surface 
\begin{equation}\label{E:X012T4}
X_0\cap X_1\cap X_2=\FG{\lla\alpha\beta02\rra}\cup\NavyBlue{\lla\alpha\gamma12\rra}\cup\red{\lla\beta\gamma01\rra},
\end{equation}
which is colored in Figure \ref{Fi:T4Actual} according to the color scheme from (\ref{E:X012T4}).  Moreover, the red (resp. blue, green) line segments in Figure \ref{Fi:T4Actual} comprise the ``red (resp. blue, green) curves'' in a trisection diagram for this trisection, and so Figure \ref{Fi:T4Actual} {\it is}, in fact, a trisection diagram (see \cite{gk,msz}). 

Note that what we have actually shown is that Figures \ref{Fi:T4Pieces}, \ref{Fi:T4Slices}, and \ref{Fi:T4Actual} give a combinatorial description of an efficient trisection of $T^4$.  Thus, since the PL and smooth categories coincide in dimension 4, $T^4$ has a smooth structure for which we have described a trisection. Most likely, this is the standard smooth structure on $T^4$, but we have not yet proven this, nor will we in this paper.

One way to prove this would be to describe a (smooth=PL) isotopy (i.e. a sequence of handleslides on the central surface) between our trisection and another trisection of the standard $T^4$, such as either of those due to Koenig or Williams, the former obtained by viewing $T^4$ as $T^3\times S^1$ \cite{k21}, the latter by viewing $T^4$ as $T^2\times T^2$ \cite{w20}.  There may well be isotopies between our constructions are theirs, but attempting to construct such isotopies explicitly is messy, in part because the central surface has genus 10, and so it remains an open question as to whether or not {\it all} efficient trisections of $T^4$ are mutually isotopic. In other words does the following theorem, proven using minimal surface theory, extend to dimension four? 

\begin{theorem}[Frohman \cite{froh}]
Up to isotopy, $T^3$ has a unique minimal genus Heegaard splitting.
\end{theorem}

\begin{question}\label{Q:T4UniqueSmooth}
Up to isotopy, does $T^4$ have a unique efficient trisection?
\end{question}

\begin{question}\label{Q:T4Exotic}
Does $T^4$ admit exotic smooth structures? If it does, then which of these exotic structures are compatible with efficient trisections?
\end{question}

\subsection{Trisection of $T^5$}\label{S:T5}

The decomposition of $T^5$ from Figure \ref{Fi:T5Intro} is given by 
\begin{equation}\label{E:X05}
X_0=\lla \alpha^2[0,2]^2{[0,3]}\rra,~X_i=X_0+(i,i,i,i,i).\end{equation}
The handle decompositions of $X_I$, $I=\{0\},\{0,1\}$, are quite similar to those from $T^4$. Focus first on $I=\{0\}$, i.e. on the handle decomposition of $X_0$. Note the single factor of $[0,3]$ in (\ref{E:X05}).   As we will explain shortly, the handle decomposition of $X_0$ here comes from the decomposition of the interval
\[\violet{[0,3]}=\Navy{[0,2]}\cup\red{\gamma},\]
and likewise for $X_0$ from the trisection of $T^4$.  These handle decompositions appear in Tables \ref{T:4X0} and \ref{T:5X0}.  These and subsequent tables are organized as follows.  In each $\boldsymbol{z}^\text{th}$ row, $\boldsymbol{Y_z}$ is a union of handles of index $\boldsymbol{h}$, $\boldsymbol{Y_z^*}$ is an {\it example} of such an $h$-handle, and the entry in the column {\bf glue to} lists those indices $z'$ for which $Y_z^*$ glues  to $Y_{z'}$ along at least one face of codimension 1. The other handles from $Y_z$ are related to $Y_z^*$ by permutation; for details, see \textsection\ref{S:Yzstar}.

\begin{table}[h!]
\begin{center}
\begin{tabular}{||c|c|ccc||}
\hline
$Y_z$&$Y_z^*$&$h$&$z$&glue to\\ 
\hline
$\lla\Navy{\alpha^2[0,2]^2}\rra$&$\lla\Navy{\alpha^2[0,2]^2}\rra$&0&1&\\
$\lla\Navy{\alpha^2[0,2]}\red{\gamma}\rra$&$\lla\Navy{\alpha^2[0,2]}\rra\red{\gamma}$&{\bf 1}&2&1\\ 
\hline
\end{tabular}
\caption{
$X_0$ from the trisection of $T^4$. }\label{T:4X0}
\end{center}
\end{table}

\begin{table}[h!]
\begin{center}
\begin{tabular}{||c|c|c|ccc||}
\hline
$J$&$Y_z$&$Y_z^*$&$h$&$z$&glue to\\ 
\hline
$\varnothing$&$\lla\Navy{\alpha^2[0,2]^3}\rra$&$\lla\Navy{\alpha^2[0,2]^3}\rra$&0&1&\\
$\{0\}$&$\lla\Navy{\alpha^2[0,2]^2}\red{\gamma}\rra$&$\lla\Navy{\alpha^2[0,2]^2}\rra\red{\gamma}$&{\bf 1}&2&1\\ 
\hline
\end{tabular}
\caption{
$X_0$ from the trisection of $T^5$}\label{T:5X0}
\end{center}
\end{table}

Note in both Tables \ref{T:4X0} and \ref{T:5X0} that $Y_1=Y_1^*$ is star-shaped in a particularly nice way (more detail to come in \textsection\ref{S:StarShaped}), hence is a ball which we may view as a 0-handle. 
Then $Y_2^*$ is the product of the same sort of star-shaped ball with the interval $\red{\gamma}$ and glues to $Y_1$ along the product of that ball with $\partial\red{\gamma}$.  The red $\gamma$ here, and all red henceforth, indicates a positive contribution to the handle index $h$.

Next, consider $X_I$, $I=\{0,1\}$ from $T^4$ and $T^5$. Similarly to the former (recall (\ref{E:X014})), the latter is given by
\begin{equation}\label{E:X015}
X_0\cap X_1=\lla\alpha1\beta^2{[1,3]}\rra\cup\lla0\alpha{\beta^2}[1,3]\rra.
\end{equation}
Handle decompositions are summarized in Tables \ref{T:4X01} and \ref{T:5X01}, which are organized largely the same way as Tables \ref{T:4X0} and \ref{T:5X0}. 


\begin{table}[h!]
\begin{center}
\begin{tabular}{||c|c|ccc||}
\hline
$Y_z$&$Y_z^*$&$h$&$z$&glue to\\ 
\hline
$\lla\Navy{\alpha}1\Navy{\beta}^{\Navy{2}}\rra$
&
$\lla\Navy{\alpha}1\Navy{\beta}^{\Navy{2}}\rra$
&0&1&\\
$\lla0\red{\alpha}\Navy{\beta}^{\Navy{2}}\rra$
&
$\lla0\red{\alpha}\rra\Navy{\beta}^{\Navy{2}}$
&1&2&1\\
$\lla\Navy{\alpha}1\Navy{\beta}\red{\gamma}\rra$
&
$\lla\Navy{\alpha}1\Navy{\beta}\rra\red{\gamma}$
&1&3&1,2\\
\hline
\end{tabular}
\caption{From the trisection of $T^4$: $X_0\cap X_1=\lla\alpha1{\beta}{[1,3]}\rra\cup\lla0{\alpha}{\Navy{\beta^2}}\rra$.
}\label{T:4X01}
\end{center}
\end{table}

\begin{table}[h!]
\begin{center}
\begin{tabular}{||cc|c|c|ccc||}
\hline
$J$&$i_*$&$Y_z$&$Y_z^*$&$h$&$z$&glue to\\ 
\hline
$\varnothing$&$0$&
$\lla\Navy{\alpha}1\Navy{\beta}^{\Navy{3}}\rra$
&
$\lla\Navy{\alpha}1\Navy{\beta}^{\Navy{3}}\rra$
&0&1&\\
&$1$&
$\lla0\red{\alpha}\Navy{\beta}^{\Navy{3}}\rra$
&
$\lla0\red{\alpha}\rra\Navy{\beta}^{\Navy{3}}$
&1&2&1\\
$\{0\}$&$0$&
$\lla\Navy{\alpha}1\Navy{\Navy{\beta^2}}\red{\gamma}\rra$
&
$\lla\Navy{\alpha}1\Navy{\Navy{\beta^2}}\rra\red{\gamma}$
&1&3&1,2\\
&$1$&
$\lla\red{\gamma}0\red{\alpha}\Navy{\Navy{\beta^2}}\rra$
&
$\lla\red{\gamma}0\red{\alpha}\rra\Navy{\Navy{\beta^2}}$
&{\bf 2}&4&2,3\\
\hline
\end{tabular}
\caption{From the trisection of $T^5$: $X_0\cap X_1=\lla\alpha1{\Navy{\beta^2}}{[1,3]}\rra\cup\lla0{\alpha}{\Navy{\beta^2}}{[1,3]}\rra.$
}\label{T:5X01}
\end{center}
\end{table}

Regarding the first columns of Table \ref{T:5X01}, each $Y_z$ there corresponds to a pair $(J,i_*)$, where $J\subset \{\min I_r\}=\{0\}$\footnote{Recall from Notation \ref{N:T} that $\{\min I_r\}=\{i_t:~t\in T\}=\{i_t\in I:~i_t-1\notin I\}$, so e.g. $\{\min I_r\}=\{0\}$ if $I=\{0\}$, $I=\{0,1\}$ or $I=\{0,1,2\}$, and $\{\min I_r\}=\{0,2\}$ if $I=\{0,2\}$.}  and $i_*\in I=\{0,1\}$. For details on this correspondence, see \textsection\ref{S:XIJi}.

\subsection{The difficulty with $T^6$}

Suppose we try to quadrisect $T^6$ in the same way, viewing $T^6$ as $(\R/4\Z)^6=[0,4]^6/\sim$ and partitioning the $4^6$ resulting subcubes into four classes.  The first problem is that no such partition is symmetric with respect to both the permutation action of $\Z_6$ on the indices and the translation action of $\Z_4$ along the main diagonal.  To see this, consider the subcube $\alpha^3\gamma^3$. The problem is that 
\[\alpha^3\gamma^3+(2,2,2,2,2,2)=\gamma^3\alpha^3\subset\lla\alpha^3\gamma^3\rra.\]
Fundamentally, the problem is that $k=4$ and $n=2k-2=6$  are not relatively prime.  (In odd dimensions, this trouble does not arise, since $k$ and $2k-1$ are relatively prime.) Perhaps there is a less symmetric way to partition the subcubes of $[0,4]^6/\sim$ which gives a quadrisection of $T^6$, but trial and error suggests to the author that this is unlikely. 

\begin{conjecture}\label{Conj:T6}
No partition of the subcubes of $[0,4]^6/\sim$ gives a quadrisection of $T^6$.
\end{conjecture}

\begin{question}\label{Q:T6}
Does the 6-dimensional torus admit an {\it efficient} quadrisection?
\end{question}

\section{Star-shaped building blocks}\label{S:StarShaped}

This section introduces three types of building blocks, each of which is PL homeomorphic to a ball.\footnote{Note that, in the PL category, an $n$-ball $D^n$ is any manifold PL homeomorphic to the standard $n$-simplex, and an $n$-sphere $S^n$ is any manifold PL homeomorphic to $\partial D^n$.} In \textsection\ref{S:MainGen}, when we describe and then justify the handle decomposition of arbitrary $X_I$  in arbitrary odd dimension, this will be particularly helpful.  The main idea is that we will decompose arbitrary $X_I$ into many pieces. Each piece will be a product of such building blocks, hence PL homeomorphic to a ball (see Lemma \ref{L:Yzr}).  Of course, we will still need to describe how all these balls are glued together and explain why this gives a handle decomposition.

In fact, we saw all three types of building blocks in \textsection\ref{S:Motivation}.  For example, denoting PL homeomorphism by $\cong$, the factors $\alpha^2[0,2]\cong D^3$, $\alpha^2[0,2]^2\cong D^4$, and $\alpha^2[0,2]^3\cong D^5$ from Tables \ref{T:4X0} and \ref{T:5X0} are examples of the first type of building block; see (\ref{E:C1}). The factor $\lla0\red{\alpha}\rra\cong D^1$ of $Y_2^*$ in Tables \ref{T:4X01} and \ref{T:5X01} is an example of the second type of building block; see (\ref{E:C2}).  The factor $\lla \red{\gamma} 0\red{\alpha}\rra\cong D^2$ from $Y_4^*$ in Table \ref{T:5X01} is an example of the third type, as are those factors $\lla \Navy{\alpha} 1\Navy{\beta^r}\rra\cong D^{r+1}$, which appear four places in Tables \ref{T:4X01} and \ref{T:5X01}. 

Given $\vec{p},\vec{q}\in\R^n$, denote the convex hull of $\{\vec{p},\vec{q}\}$ by
\[\left[\vec{p},\vec{q}\right]=\left\{t\vec{p}+(1-t)\vec{q}:~0\leq t\leq 1\right\}.\]
Let $\vec{p}\in Y\subset \R^n$. Define the {\it scope} of $\vec{p}$ in $Y$ to be the largest star of $\vec{p}$ in $Y$:
\[
\text{scope}(Y;\vec{p})=\{\vec{q}\in Y:~\left[\vec{p},\vec{q}\right]\subset Y\}.\]
Say that $Y$ is {\it star-shaped} about $\vec{p}$ if $Y=\text{scope}(Y;\vec{p})$.  The {\it link} of $\vec{p}$ in $Y$ is 
\[\text{lk}_Y(\vec{p})=\left\{\vv\in\R^n:~|\vv|=1,~\left[\vec{p},\vec{p}+\ep\vv\right]\subset Y\text{ for some }\ep>0\right\}.\]
Thus, $Y$ is a $d$-dimensional PL submanifold of $\R^n$ near $\vec{p}$ if and only if either
\begin{itemize}
\item $\text{lk}_Y(\vec{p})\cong S^{d-1}$, in which case $\vec{p}$ is in the {\it interior} of $Y$; or
\item $\text{lk}_Y(\vec{p})\cong D^{d-1}$, in which case $\vec{p}\in \partial Y$.
\end{itemize}
Suppose $Y=\text{scope}(Y;\vec{p})$ and $\text{lk}_Y(\vec{p})\cong S^{d-1}$, so $Y$ is star-shaped about $\vec{p}$ and is a PL $d$-submanifold of $\R^n$ near $\vec{p}$. In this situation, we say $Y$ is {\it strongly star-shaped} about $\vec{p}$ if moreover, for every point $\vec{q}\in  Y$, every point $\vec{x}\in\left[\vec{p},\vec{q}\right]\setminus\left\{\vec{q}\right\}$ satisfies $\text{lk}_Y(\vec{x})\cong S^{d-1}$. This extra requirement implies that, 
for each $\vec{q}\in\text{link}_Y(\vec{p})$, the ray from $\vec{p}$ through $\vec{q}$ contains at most one point of $\partial Y$. Moreover:

\begin{prop}\label{P:StarShaped}
If $Y\subset \R^n$ is compact and strongly star-shaped about $\vec{p}\in Y$, then $Y$ is PL homeomorphic to a compact ball.
\end{prop}

\begin{proof}
By definition, there is a PL homeomorphism $\phi: S^{d-1}\to\text{lk}_Y(\vec{p})$. There is also a map $\psi:Y\setminus\{\vec{p}\}\to\text{lk}_Y(\vec{p})$ given by $\psi:\vec{q}\mapsto \frac{\vec{q}-\vec{p}}{|\vec{q}-\vec{p}|}$.\footnote{We use the product metric on $\R^n$: if $\vec{p}=(p_1,\hdots,p_n)$ and $\vec{q}=(q_1,\hdots,q_n)$, then $|\vec{q}-\vec{p}|=\max_i|q_i-p_i|$.} Denote the restriction $\psi|_{\partial Y}$ by $\Psi$. The assumptions that $Y$ is compact and {\it strongly} star-shaped about $\vec{p}$ imply that $\Psi$ has a well-defined, continuous inverse map, hence is a PL homeomorphism.  Define a polar coordinate system $\Phi:Y\to D^d$ by $\Phi:\vec{p}\mapsto \vec{0}$ and, for $\vec{q}\neq \vec{p}$,
\[\Phi:\vec{q}\mapsto\frac{|\vec{q}-\vec{p}|}{|\Psi^{-1}\circ \psi(\vec{q})-\vec{p}|}\cdot\phi^{-1}\circ \psi(\vec{q}).\]
This map $\Phi$ is a PL homeomorphism, because the inverse map $D^d\to Y$ is 
\[\pushQED{\qed}
\Phi^{-1}:r\vec{\theta}\mapsto \vec{p}+r|\Psi^{-1}\circ \phi(\vec{\theta})-\vec{p}|\cdot\phi(\vec{\theta}).\qedhere\]
\end{proof}
\begin{figure}
\begin{center}
\includegraphics[width=\textwidth]{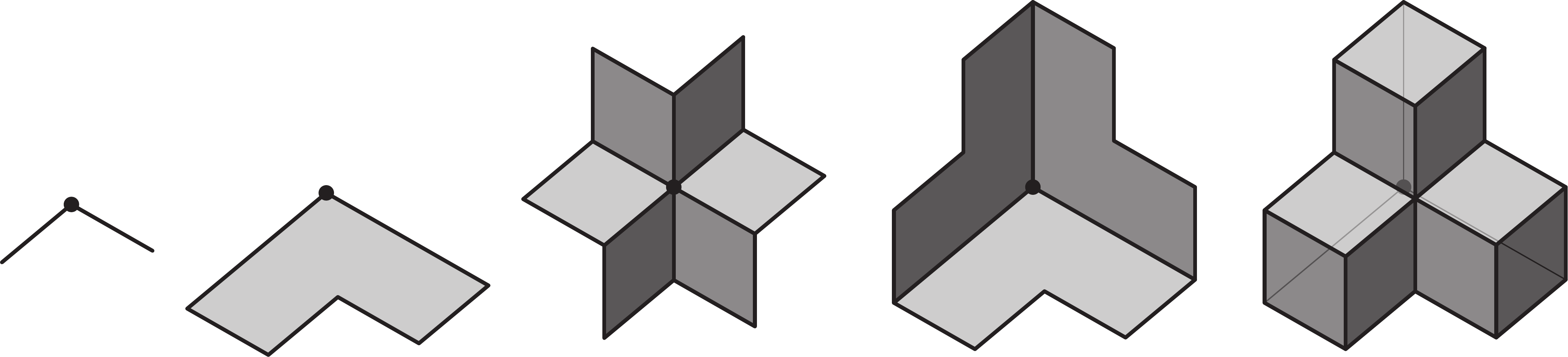}
\caption{Left to right: $\lla 0\alpha\rra$, $\lla\alpha[0,2]\rra$, $\lla \alpha1\beta\rra$, $\lla0\alpha[0,2]\rra$, $\lla \alpha^2[0,2]\rra$.}
\label{Fi:BuildingBlocks}
\end{center}
\end{figure}

In $T^n=(\R/k\Z)^n$, for $d\leq n-1$, identify $T^d=(\R/k\Z)^d$ with $(\R/k\Z)^d\times\{\vec{0}\}\subset T^n$, and likewise for $T^{d+1}$. 
For any $0<a_1\leq\cdots\leq a_d<k$ (not necessarily integers), define
\begin{align}\label{E:C1}
C_1&=\lla\prod_{r=1}^d[0,a_r]\rra\subset T^d,\\
\label{E:C2}
C_2&=\lla\{0\}\times \prod_{r=1}^d[0,a_r]\rra\subset T^{d+1},\text{ and }
\end{align}
\begin{equation}\label{E:C3}
C_3=\lla[0,a_1]\times\{a_1\}\times \prod_{r=2}^d[a_1,a_r]\rra\subset T^{d+1}.
\end{equation}

Figures \ref{Fi:BuildingBlocks} and \ref{Fi:BuildingBlocks2} show low-dimensional examples of these {\it building blocks}. In Figure \ref{Fi:BuildingBlocks}, $\lla\alpha[0,2]\rra$ and $\lla\alpha^2[0,2]\rra$ are examples of $C_1$, $\lla 0\alpha\rra$ and $\lla0\alpha[0,2]\rra$ are examples of $C_2$, and $\lla \alpha1\beta\rra$ is an example of $C_3$. In Figure \ref{Fi:BuildingBlocks2},  $\lla 0\alpha^3\rra$ and $\lla0\alpha^2[0,2]\rra$ are examples of $C_2$, and $\lla \alpha1\beta^2\rra$ is an example of $C_3$.

\begin{figure}
\begin{center}
\includegraphics[width=\textwidth]{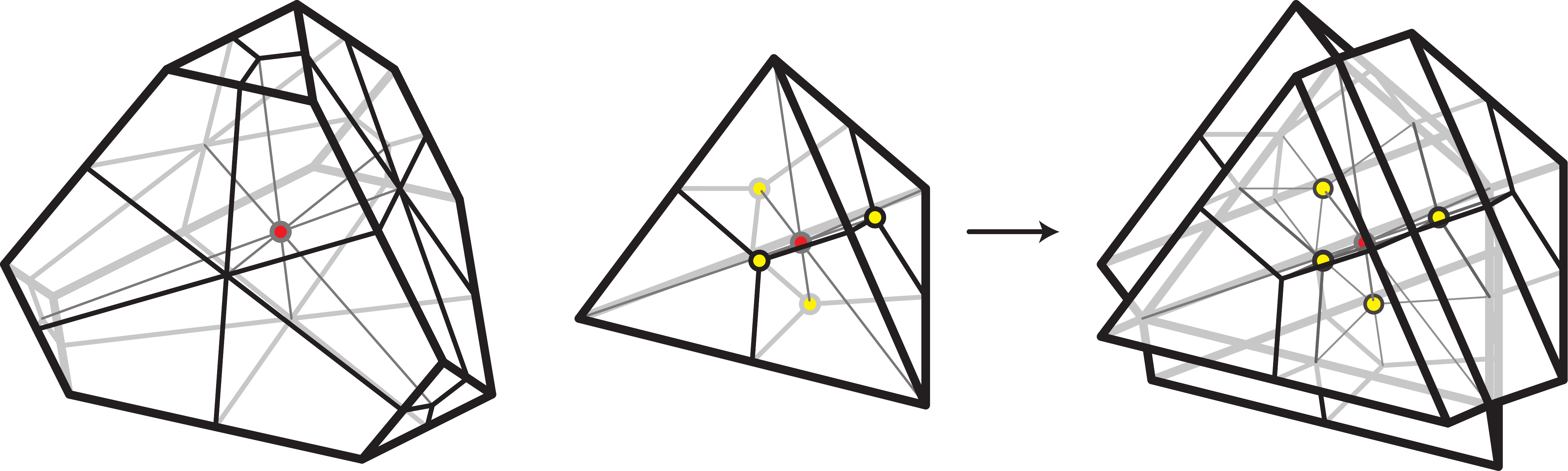}
\caption{Left to right: $\lla \alpha1{\beta^2}\rra$ and $\lla0\alpha^3\rra\to\lla0\alpha^2[0,2]\rra$.}
\label{Fi:BuildingBlocks2}
\end{center}
\end{figure}


\begin{lemma}\label{L:StarShaped}
$C_1$, $C_2$, and $C_3$ from (\ref{E:C1})-(\ref{E:C3}) are PL homeomorphic to $D^d$. 
\end{lemma}

\begin{proof}
Let $b=\frac{1}{2}(k+a_d)$. 
Then $C_1\subset [0,b]^d$ and $C_2,C_3\subset [0,b]^{d+1}$, where $b<k$, so we may view $C_1$ as a subset of $\R^d$ and $C_2,C_3$ as subsets of $\R^{d+1}$.  Let $a=\frac{a_1}{2}$, $\vec{p}_1=(a,\hdots,a)\in \R^d$, $\vec{p}_2=\vec{0}\in \R^{d+1}$, and $\vec{p}_3=(a_1,\hdots,a_1)\in\R^{d+1}$. Then, for $i=1,2,3$, $C_i$ is compact and strongly star-shaped about $\vec{p}_i$, with $\text{link}_{C_i}(\vec{p}_i)\cong S^{d-1}$, hence PL homeomorphic to $D^d$ by Proposition \ref{P:StarShaped}.
\end{proof}

\section{Further examples}\label{S:Examples2}

As noted in the introduction, the hardest part of verifying our multisection of $T^n$, in arbitrary odd dimension $n$, is describing the handle decomposition of $X_I$ for arbitrary $I\subset\Z_k$.  That task will follow three main steps. First, Lemma \ref{L:XI} will establish a closed formula (\ref{E:XI}) for arbitrary $X_I$. Second, \textsection\ref{S:Handles} will describe how (in several steps) to decompose $X_I$ into pieces, each of which is a product of the building blocks from \textsection\ref{S:StarShaped}, and will describe an order on these pieces. Third, \textsection\ref{S:HandleProp} will establish several properties of the resulting decomposition, eventually proving that it is an appropriate handle decomposition of $X_I$ and thus verifying Theorem \ref{T:Main}.  

To prepare, this section describes a few more examples, 
each of which confronts and resolves an additional complication in the handle decomposition of some $X_I$ in some dimension.
This section contains no proofs and little narration.  Instead, the reader is encouraged to peruse the tables that follow in order to build intuition for the denser sections that follow. Indeed, assuming only the correctness of the formula (\ref{E:XI}), the reader should now be able to use their understanding of the building blocks from \textsection\ref{S:StarShaped} to check the correctness of the handle decompositions, as detailed in the last five columns of the tables (starting with $Y_z$).  

The harder part will be understanding how each handle decomposition has been constructed.  This is the purpose of the columns in each table which precede $Y_z$, which we do not attempt to describe in detail until \textsection\ref{S:MainGen}.

\subsection{Quadrisection of $T^7$}\label{S:T7}

The next several examples come from the decomposition of $T^7$ given by $X_0=\lla\alpha^2[0,2]^2[0,3]^2[0,4]\rra$ and $X_i=X_0+(i,i,i,i,i,i,i)$.
The handle decompositions of $X_I$, $I=\{0\},\{0,1\}$,  summarized in Tables \ref{T:7X0} and \ref{T:7X01}, respectively, follow the same pattern in dimension seven (and all higher odd dimensions) as in dimension five (recall Tables \ref{T:5X0} and \ref{T:5X01} and the attending discussions).
More instructive examples follow.

 \begin{table}[h!]
\begin{center}
\begin{tabular}{||c|c|c|ccc||}
\hline
$J$&$Y_z$&$Y_z^*$&$h$&$z$&glue to\\ 
\hline
$\varnothing$&$ \lla\Navy{\alpha^2[0,2]^2[0,3]^2[0,4]^3}\rra$&$ \lla\Navy{\alpha^2[0,2]^2[0,3]^2[0,4]^3}\rra$&0&1&\\
$\{0\}$&$\lla\Navy{\alpha^2[0,2]^2[0,3]^2[0,4]^2}\red{\ep}\rra$&$ \lla\Navy{\alpha^2[0,2]^2[0,3]^2[0,4]^2}\rra\red{\ep}$&{\bf 1}&2&1\\ 
\hline
\end{tabular}
\caption{
$X_0$ from the quadrisection of $T^7$}\label{T:7X0}
\end{center}
\end{table}

\begin{table}[h!]
\begin{center}
\begin{tabular}{||cc|c|c|ccc||}
\hline
$J$&$i_*$&$Y_z$&$Y_z^*$&$h$&$z$&glue to\\ 
\hline
$\varnothing$&$0$&
$\lla\Navy{\alpha}1\Navy{\Navy{\beta^2}\Navy{[1,3]^3}}\rra$
&
$\lla\Navy{\alpha}1\Navy{\Navy{\beta^2}\Navy{[1,3]^3}}\rra$
&0&1&\\
&$1$&
$\lla0\red{\alpha}\Navy{\Navy{\beta^2}\Navy{[1,3]^3}}\rra$
&
$\lla0\red{\alpha}\rra\lla\Navy{\Navy{\beta^2}\Navy{[1,3]^3}}\rra$
&1&2&1\\
$\{0\}$&$0$&
$\lla\Navy{\alpha}1\Navy{\Navy{\beta^2}\Navy{[1,3]^2}}\red{\delta}\rra$
&
$\lla\Navy{\alpha}1\Navy{\Navy{\beta^2}\Navy{[1,3]^2}}\rra\red{\delta}$
&1&3&1,2\\
&$1$&
$\lla\red{\delta}0\red{\alpha}\Navy{\Navy{\beta^2}\Navy{[1,3]^2}}\rra$
&
$\lla\red{\delta}0\red{\alpha}\rra\lla\Navy{\Navy{\beta^2}\Navy{[1,3]^2}}\rra$
&{\bf 2}&4&2,3\\
\hline
\end{tabular}
\caption{
$X_I$, $I=\{0,1\}$ from the quadrisection of $T^7$}\label{T:7X01}
\end{center}
\end{table}

\subsubsection{$X_I$ when $I=\{0,2\}$}\label{S:X02}

From the quadrisection of $T^7$, consider 
\[X_0\cap X_2=\lla\alpha^2{[0,2]}{\gamma^2}[2,4]\rra
\cup
\lla\alpha^2{[0,2]}2{\gamma^2}[2,4]\rra.\]

\begin{table}[h!]
\begin{center}
\scalebox{1}{\begin{tabular}{||cc|cc|c|c|ccc||}
\hline
$J$&$i_*$&$V$&$V^-$&$Y_z$&$Y_z^*$&$h$&$z$&glue to\\ 
\hline
$\varnothing$&0&$\varnothing$&$\varnothing$&$\lla\Navy{\alpha^3}2\Navy{\gamma^3}\rra$&$\Navy{\alpha^3}\lla2\Navy{\gamma^3}\rra$&0&1&\\
&2&$\varnothing$&&$\lla0\Navy{\alpha^3}\Navy{\gamma^3}\rra$&$\lla0\Navy{\alpha^3}\rra\Navy{\gamma^3}$&0&2&\\
\hline%
\{0\}&0&$\varnothing$&$\varnothing$&$\lla\Navy{\alpha^3}2\Navy{\Navy{\gamma^2}}\red{\delta}\rra$&$\Navy{\alpha^3}\lla2\Navy{\Navy{\gamma^2}}\rra\red{\delta}$&1&3&1,2\\ 
&2&$\{0\}$&$\varnothing$&$\lla\Navy{\delta^+}0\Navy{\alpha^3}\Navy{\Navy{\gamma^2}}\rra$&$\lla\Navy{\delta^+}0\Navy{\alpha^3}\rra\Navy{\Navy{\gamma^2}}$&0&4&\\ 
&&&$\{0\}$&$\lla\red{\delta^-}0\Navy{\alpha^3}\Navy{\Navy{\gamma^2}}\rra$&$\red{\delta^-}\lla0\Navy{\alpha^3}\rra\Navy{\Navy{\gamma^2}}$&1&5&2,4\\ 
\hline%
$\{2\}$&0&$\{2\}$&$\varnothing$&$\lla\Navy{\alpha^2}\Navy{\beta^+}2\Navy{\gamma^3}\rra$&$\Navy{\alpha^2}\lla\Navy{\beta^+}2\Navy{\gamma^3}\rra$&0&6&\\
&&&$\{2\}$&$\lla\Navy{\alpha^2}\red{\beta^-}2\Navy{\gamma^3}\rra$&$\Navy{\alpha^2}\red{\beta^-}\lla2\Navy{\gamma^3}\rra$&1&7&1,6\\
&2&$\varnothing$&$\varnothing$&$\lla0\Navy{\alpha^2}\red{\beta}\Navy{\gamma^3}\rra$&$\lla0\Navy{\alpha^2}\rra\red{\beta}\Navy{\gamma^3}$&1&8&1,2\\
\hline%
$\{0,2\}$&0&$\{2\}$&$\varnothing$&$\lla\Navy{\alpha^2}\Navy{\beta^+}2\Navy{\Navy{\gamma^2}}\red{\delta}\rra$&$\Navy{\alpha^2}\lla\Navy{\beta^+}2\Navy{\Navy{\gamma^2}}\rra\red{\delta}$&1&9&6,8\\ 
&&&$\{2\}$&$\lla\Navy{\alpha^2}\red{\beta^-}2\Navy{\Navy{\gamma^2}}\red{\delta}\rra$&$\Navy{\alpha^2}\red{\beta^-}\lla2\Navy{\Navy{\gamma^2}}\rra\red{\delta}$&{\bf 2}&10&3,7,8,9\\ 
&2&$\{0\}$&$\varnothing$&$\lla\Navy{\delta^+}0\Navy{\alpha^2}\red{\beta}\Navy{\Navy{\gamma^2}}\rra$&$\lla\Navy{\delta^+}0\Navy{\alpha^2}\rra\red{\beta}\Navy{\Navy{\gamma^2}}$&1&11&3,4\\ 
&&&$\{0\}$&$\lla0\Navy{\alpha^2}\red{\beta}\Navy{\Navy{\gamma^2}}\red{\delta^-}\rra$&$\lla0\Navy{\alpha^2}\rra\red{\beta}\Navy{\Navy{\gamma^2}}\red{\delta^-}$&{\bf 2}&12&3,5,8,11\\ 
\hline \hline
\end{tabular}}
\caption{$X_I$, $I=\{0,2\}$ from the quadrisection of $T^7$%
}\label{T:7X02}
\end{center}
\end{table}

Table \ref{T:7X02} summarizes a handle decomposition $X_I=Y_1\cup \cdots\cup Y_{12}$.  As with $X_I$, $I=\{0,1\}$, the decomposition of $X_I$, $I=\{0,2\}$ is organized largely according to $\{(J,i_*):J\subset\{\min I_r\},~i_*\in I\}$. With $Y_4,Y_5,Y_7,Y_8,Y_{10}$, and $Y_{11}$ here, we have $J\setminus\{\min I_*\}\neq \varnothing$, requiring us to split a unit interval into subintervals, in this case halves.  Details on how this is done, including the definitions and purposes of the sets $V^-\subset V\subset I$, appear in \textsection\ref{S:Handles}, especially Table \ref{T:UV}, and in Tables \ref{T:UV1} and \ref{T:UV2} in Appendix 1.

\subsubsection{$X_I$ when $I=\{0,1,2\}$}\label{S:X012}

Still in dimension seven, consider 
\[X_0\cap X_1\cap X_2=\lla\alpha^2{[0,2]}{\gamma^2}[2,4]^2{[2,5]}\rra
\cup
\lla\alpha^2{[0,2]}2{\gamma^2}[2,4]^2{[2,5]}\rra.\]

\begin{table}[h!]
\begin{center}
\scalebox{1}{\begin{tabular}{||cc|ccc|c|ccc||}
\hline
$J$&$i_*$&$U$&$V$&$V^-$&$Y_z^*$&$h$&$z$&glue to\\ 
\hline
$\varnothing$&0&$\varnothing$&\{1,2\}&$\{1\}$&$\Navy{\alpha^-}1\lla\Navy{\beta^+}2\Navy{\gamma^3}\rra$&0&1&\\
&&&&$\varnothing$&$\lla\red{\alpha^+}1\rra\lla\Navy{\beta^+}2\Navy{\gamma^3}\rra$&1&2&1\\
&&&&\{1,2\}&$\Navy{\alpha^-}\lla1\red{\beta^-}\rra\lla2\Navy{\gamma^3}\rra$&1&3&1\\
&&&&\{2\}&$\lla\red{\alpha^+}1\red{\beta^-}\rra\lla2\Navy{\gamma^3}\rra$&2&4&2,3\\

&1&$\varnothing$&$\varnothing$&$\varnothing$&$\lla0\red{\alpha}\rra\lla\Navy{\beta}2\Navy{\gamma^3}\rra$&1&5&1,3\\

&2&$\{1\}$&$\varnothing$&$\varnothing$&$0\Navy{\alpha^\circ_3}\lla1\red{\beta}\rra\Navy{\gamma^3}$&1&6&5\\
&&&&&$\lla0\red{\alpha^-_3}\rra\lla1\red{\beta}\rra\Navy{\gamma^3}$&2&7&5,6\\
&&&&&$0\lla\red{\alpha^+_3}1\red{\beta}\rra\Navy{\gamma^3}$&2&8&5,6\\
$\{0\}$&0&$\varnothing$&\{1,2\}&$\{1\}$&$\red{\delta}\Navy{\alpha^-}1\lla\Navy{\beta^+}2\Navy{\Navy{\gamma^2}}\rra$&1&9&1,6,7\\
&&&&$\varnothing$&$\red{\delta}\lla\red{\alpha^+}1\rra\lla\Navy{\beta^+}2\Navy{\Navy{\gamma^2}}\rra$&2&10&2,6,8\\
&&&&\{1,2\}&$\red{\delta}\Navy{\alpha^-}\lla1\red{\beta^-}\rra\lla2\Navy{\Navy{\gamma^2}}\rra$&2&11&3,6,7\\
&&&&\{2\}&$\red{\delta}\lla\red{\alpha^+}1\red{\beta^-}\rra\lla2\Navy{\Navy{\gamma^2}}\rra$&{\bf 3}&12&4,6,8\\
&1&$\varnothing$&$\varnothing$&$\varnothing$&$\lla\red{\delta}0\red{\alpha}\rra\lla\Navy{\beta}2\Navy{\Navy{\gamma^2}}\rra$&2&13&5,9,11\\
&2&$\{1\}$&$\varnothing$&$\varnothing$&$\lla\red{\delta}0\rra\Navy{\alpha^\circ_3}\lla1\red{\beta}\rra\Navy{\Navy{\gamma^2}}$&2&14&6,13\\
&&&&&$\lla\red{\delta}0\red{\alpha^-_3}\rra\lla1\red{\beta}\rra\Navy{\Navy{\gamma^2}}$&{\bf 3}&15&7,13,14\\
&&&&&$\lla\red{\delta}0\rra\lla\red{\alpha^+_3}1\red{\beta}\rra\Navy{\Navy{\gamma^2}}$&{\bf 3}&16&8,13,14\\
\hline
\end{tabular}}
\caption{$X_I$, $I=\{0,1,2\}$ from the quadrisection of $T^7$
}\label{T:7X012}
\end{center}
\end{table}

Table \ref{T:7X012} summarizes a handle decomposition $X_I=Y_1\cup \cdots\cup Y_{12}$.  Again, the decomposition of $X_I$, $I=\{0,1,2\}$ is organized largely according to $\{(J,i_*):J\subset\{\min I_r\},~i_*\in I\}$. Here, we have a block $I_r$ (in this case $I_r=I$) with $|I_r|\geq 3$, requiring us at times to split a unit interval into thirds, as seen here in $Y_6-Y_{8}$ and $Y_{14}-Y_{16}$. Details on this and the set $U$ appear in \textsection\ref{S:Handles}, especially Table \ref{T:UV}, and in Tables \ref{T:UV1} and \ref{T:UV2} in Appendix 1.

Another new complication arises here in $Y_1-Y_4$ and $Y_9-Y_{12}$, where $i_*+2\in I_*$, requiring us to split certain unit intervals into halves according to a different rule than in \textsection\ref{S:X02}.  Again, all the rules for splitting unit intervals into halves and thirds are detailed in \textsection\ref{S:Handles}, especially Table \ref{T:UV}, and in Tables \ref{T:UV1} and \ref{T:UV2} in Appendix 1.

\subsection{$X_I$, $I=\{0,1,2,3,5\}$ from $T^{13}$}\label{S:0124}

There is one more complication, which arises, first in dimension 11, whenever 
$X_I$, $I=I_1\sqcup \cdots \sqcup I_m$, has some $I_r\not\ni i_*$ with $|I_r|\geq 3$. In fact, though, the {\it difficulty} of this complication only becomes apparent in dimension 13. From the septisection of $T^{13}$, consider $X_I$, $I=\{0,1,2,3,5\}$, which is given by 
\begin{equation*}\label{E:13X01235}
\begin{split}
\lla\alpha1\beta2{\gamma}3\delta^2{[3,5]}5{\ep^2}{[5,7]}\rra
&\cup\lla0\alpha\beta2{\gamma}3\delta^2{[3,5]}5{\ep^2}{[5,7]}\rra
\cup\lla0\alpha1\beta{\gamma}3\delta^2{[3,5]}5{\ep^2}{[5,7]}\rra\\
&\cup\lla0\alpha1\beta2{\gamma}\delta^2{[3,5]}5{\ep^2}{[5,7]}\rra
\cup\lla0\alpha1\beta2{\gamma}3\delta^2{[3,5]}{\ep^2}{[5,7]}\rra.
\end{split}
\end{equation*}
In this example, the new complication arises when $i_*=5$ and $0\notin J$, i.e in the part of $X_I$ given by 
 \[\lla0\alpha1\beta2{\gamma}3\delta^2{[3,5]}{\ep^3}\rra
,\]
part of which appears in the first several $Y_z$ in the handle decomposition of this $X_I$. See  Table \ref{T:13}. The tricky part here is how to order the pieces $Y_z$. See \textsection\ref{S:Order}, especially (\ref{E:order}).
\begin{table}[h!]
\begin{center}
\scalebox{1}{\begin{tabular}{||c|c|ccc||}
\hline
$V^-$&$Y_z^*$&$h$&$z$&glue to\\ 
\hline
$\varnothing$&$0\lla\Navy{\alpha^+}1\rra\lla\Navy{\beta^+}2\rra\lla\Navy{\gamma^+}3\Navy{\delta^3}\rra\Navy{\zeta^3}$&0&1&\\
$\{1\}$&$\lla0\red{\alpha^-}\rra1\lla\Navy{\beta^+}2\rra\lla\Navy{\gamma^+}3\Navy{\delta^3}\rra\Navy{\zeta^3}$&1&2&1\\
$\{1,2\}$&$\lla0\Navy{\alpha^-}\rra\lla1\red{\beta^-}\rra2\lla\Navy{\gamma^+}3\Navy{\delta^3}\rra\Navy{\zeta^3}$&1&3&2\\
$\{2\}$&$0\lla\red{\alpha^+}1\red{\beta^-}\rra2\lla\Navy{\gamma^+}3\Navy{\delta^3}\rra\Navy{\zeta^3}$&2&4&1,3\\
$\{2,3\}$&$0\lla\Navy{\alpha^+}1\Navy{\beta^-}\rra\lla2\red{\gamma^-}\rra\lla3\Navy{\delta^3}\rra\Navy{\zeta^3}$&1&5&4\\
$\{1,2,3\}$&$\lla0\red{\alpha^-}\rra\lla1\Navy{\beta^-}\rra\lla2\red{\gamma^-}\rra\lla3\Navy{\delta^3}\rra\Navy{\zeta^3}$&2&6&3,5\\
$\{1,3\}$&$\lla0\Navy{\alpha^-}\rra1\lla\red{\beta^+}2\red{\gamma^-}\rra\lla3\Navy{\delta^3}\rra\Navy{\zeta^3}$&2&7&2,6\\
$\{3\}$&$0\lla\red{\alpha^+}1\rra\lla\red{\beta^+}2\red{\gamma^-}\rra\lla3\Navy{\delta^3}\rra\Navy{\zeta^3}$&3&8&1,5,7\\
\hline
\end{tabular}}
\caption{From the septisection of $T^{13}$: the start of the handle decomposition of $X_I$ when $I=\{0,1,2,3,5\}$. Here, $J=\varnothing$, $i_*=5$, $U=\varnothing$, and $V=\{1,2,3\}$.}\label{T:13}
\end{center}
\end{table}

Also see Table \ref{T:15} in Appendix 1, which summarizes the start of the handle decomposition of $X_I$, $I=\{0,1,2,3,4,6\}$, from $T^{15}$

\section{Combinatorics}\label{S:Cutoff}

This section proves several combinatorial facts about the decompositions of $T^n$. In particular, \textsection\ref{S:DisjointCover} proves that $T^n=\bigcup_{i\in\Z_k}X_i$, and \textsection\ref{S:EXI} establishes a closed expression (\ref{E:XI}) for arbitrary $X_I$. Also, \textsection\ref{S:CombCor} establishes two combinatorial corollaries, which may be of independent interest but otherwise are not needed in this paper.
%
%


\subsection{Notation}\label{S:NotationCutoff}


Because each $X_i$ from our construction (\ref{E:X0ThmIntro}) is symmetric under the permutation action of $S_n$ on the indices in $T^n$, it will often suffice, when considering an arbitrary point $\x=(x_1,\hdots, x_n)\in (\R/k\Z)^n=T^n$, to assume that $\x$ is {\bf monotonic} in the sense that $x_1\leq x_2\leq\cdots\leq x_n\leq k+x_1$. 

Denoting the main diagonal of $T^n$ by $\Delta$, note that each monotonic point $\x=(x_1,\hdots,x_n)\in T^n\setminus\Delta$ corresponds to a unique point $(x_1,\hdots,x_n)\in \R^n$ with $0\leq x_1\leq x_2\leq\cdots\leq x_n\leq x_1+k<2k$.
For such $\x$, extend the point $(x_1,\hdots,x_n)\in\R^n$ to a point $\x_\infty=\left(x_r\right)_{r\in\Z}\in\R^\Z$ by defining for each $r\in\Z_k$ and $m\in\Z$:
\begin{equation*}\label{E:vecy}
x_{r+mn}=x_r+mk.
\end{equation*}
We will mainly be interested in $0\leq x_1\leq \cdots\leq x_{2n}$, where 
\[x_{2n}=x_n+k\leq x_1+2k< 3k.\]
With this setup for any monotonic $\x\in T^n\setminus\Delta$, define the following {\bf cutoff indices} $a_r(\x),b_r(\x)\in\Z$ for each $r\in\Z$:
\begin{equation*}\label{E:ar}
a_r(\x)=\min\{a:x_{a+1}\geq r\}\text{ and}
\end{equation*}
\begin{equation*}\label{E:br}
b_r(\x)=\min\{b:x_{b+1}> r\}.
\end{equation*}
Note that, in all cases, we have $a_0(\x)\leq 0$,
with equality if and only if $x_n\neq k\equiv 0\in\R/k\Z$. The main point is:
\begin{obs}\label{O:a0block}
Let $\x\in T^n\setminus\Delta$ be monotonic. Then $\x\in[0,1]^2\cdots[0,k-1]^2[0,k]$ if and only if $b_s(\x)\geq 2s$ for every $s=0,\hdots,k-1$.
\end{obs}
Note that $b_0(\x)\geq 0$ in all cases. In order to apply the principle of Observation \ref{O:a0block} more broadly, denote for each $r\in\Z$:
\begin{equation*}\label{E:xr}
\x_r=(x_{1+a_r(\x)},x_{2+a_r(\x)},\hdots,x_{n+a_r(\x)})
\end{equation*}
The point regarding monotonic points off the main diagonal is: 
\begin{obs}\label{O:brs}
If $\x\in T^n\setminus\Delta$ is monotonic and $r\in\Z$, then 
\[r\leq x_{1+a_r(\x)}\leq \cdots\leq x_{n+a_r(\x)}<r+k,\]
and the following conditions are equivalent:
\begin{itemize}
\item $\x_r\in [r,r+1]^2\cdots[r,r+k-1]^2[r,r+k]$;
\item $b_{r+s}(\x_r)\geq 2s$ for every $s=r+1,\hdots,r+k$;
\item $b_{r+s}(\x)\geq a_r(\x)+ 2s$ for every $s=r+1,\hdots,r+k$
\end{itemize}
\end{obs}
Observations \ref{O:a0block} and \ref{O:brs} apply more generally using:
\begin{obs}\label{O:XrCr}
If $\x\in X_r\subset T^n\setminus\Delta$, then there is a permutation $\sigma\in S_n$ such that $\x_\sigma\in[r,r+1]^2\cdots[r,r+k-1]^2[r,r+k]$ is monotonic.
\end{obs}
Note also that 
either class of cutoff indices provides two-sided bounds for the other class:
\begin{obs}\label{O:arbrar}
If $\x\in T^n$ is nonzero and monotonic and $r\in\Z$, then 
\begin{equation*}\label{E:Ineqs}
\cdots\leq a_r(\x)\leq b_r(\x)\leq a_{r+1}(\x)\leq b_{r+1}(\x)\leq \cdots
\end{equation*}
with $a_r(\x)=b_r(\x)$ if and only if $x_{a_r(\x)+1}\notin \Z_k$, and $b_r(\x)=a_{r+1}(\x)$ if and only if $x_{b_r(\x)+1}\geq r+1$.
\end{obs}
Note that $x_{b_r(\x)+1}$ is the first coordinate in $\x$ that exceeds $r$. Here is another convenient property:
\begin{obs}\label{O:ark}
Any nonzero monotonic $\x\in T^n$, $r\in\Z_\geq 0$ satisfy 
\begin{equation}\label{E:arar+k}
\begin{split}
a_{r+k}(\x)&=n+a_r(\x),\\
b_{r+k}(\x)&=n+b_r(\x).
\end{split}
\end{equation}
\end{obs}

Noting that $X_r\cap \Delta=\{(x,\hdots,x):~x\in [r,r+1]\}$, we can express each $X_r$ in terms of cutoff indices as follows.

\begin{prop}\label{P:Xr}
Let $\x\in T^n\setminus\Delta$ be monotonic, and let $r\in\Z_k$. Then $\x\in X_r$ if and only if $\x_r\in [r,r+1]^2\cdots[r,r+k-1]^2[r,r+k]$. In particular, 
\begin{equation}\label{E:Xr}
X_r\setminus\Delta=\lla \text{monotonic }\x :~b_{r+s}(\x)\geq a_r(\x)+2s\text{ for }s=0,\hdots, k-1\rra.
\end{equation}
\end{prop}

\begin{proof}
Write $\x_r=(x_1,\hdots, x_n)$. Note that $r\leq x_1\leq \cdots\leq x_n<r+k$. To show that $\x_r\in [r,r+1]^2\cdots[r,r+k-1]^2[r,r+k]$ if and only if $\x_r\in X_r$, we will prove both containments. One is trivial.  For the other, suppose that $\x_r\notin [r,r+1]^2\cdots[r,r+k-1]^2[r,r+k]$. Then Observation \ref{O:brs} implies that $b_{r+s}(\x_r)<2s$ for some $s=0,\hdots, k-1$, so 
\[r+s<x_{2s},\hdots,x_n<r+k.\]
Thus, at least $n+1-2s$ of the coordinates of $\x$ lie in the open interval $(r+s,r+k)$.  Yet, $2s$ of the $n$ factors of $[r,r+1]^2\cdots[r,r+k-1]^2[r,r+k]$ are disjoint from that open interval. Contradiction. Observation \ref{O:XrCr} now implies that $\x\in X_r\setminus\Delta$ if and only if $\x$ is an element of the \textsc{rhs} of (\ref{E:Xr}). 
 \end{proof}

\subsection{The $X_r$ have disjoint interiors and cover $T^n$.}\label{S:DisjointCover}


\begin{prop}\label{P:Int}
With the setup from Theorem \ref{T:Main},  $X_r$ and $X_s$ have disjoint interiors whenever $0\leq r<s\leq k-1$.
\end{prop}

This will follow from Lemma \ref{L:XI}, but the following proof is much easier than that of the lemma; we include it for expository reasons.

\begin{proof}
By the symmetry of the construction, 
we may assume that $r=0$. Assume for contradiction that the interiors of $X_r$ and $X_s$ intersect.  Then $X_r\cap X_s$ has positive measure, so there is a monotonic point $\x=(x_1,\hdots,x_n)\in X_0\cap X_j$ such that for every $i=1,\hdots, n$ we have $x_i\notin \Z_k$. 

This implies that $a_i(\x)=b_i(\x)$ for each $i=1,\hdots, n$, by Observation \ref{O:arbrar}.  In particular, since $\x\in X_0$, we have $a_0=b_0=0$, and $a_s=b_s\geq 2s$ by Proposition \ref{P:Xr}.  But then, since $\x\in X_s$ and $a_s\geq 2s$, Observation \ref{O:ark} and Proposition \ref{P:Xr} give the following contradiction:
\begin{align*}
\pushQED{\qed}
n&=n+b_0=b_k=b_{s+(k-s)}\\
n
&\geq a_s+2(k-s)\\
n &\geq 2k.\qedhere
\end{align*}
\end{proof}


\begin{lemma}\label{L:cover}
We have $X_0\cup\cdots\cup X_{k-1}=T^n$.
\end{lemma}


\begin{proof}
Let $\x\in T^n$. We will prove that $\x\in X_s$ for some $s$.
If $\x=(x,\hdots,x)\in\Delta$, then $\x\in X_{\lfloor x\rfloor}$.  Assume instead that $\x\in T^n\setminus\Delta$. Also assume without loss of generality that $\x$ is monotonic with $a_0(\x)=0$. Throughout this proof, denote each $a_s(\x)$ by $a_s$ and each $b_s(\x)$ by $b_s$.

Let $s_0=0$, so that $a_{s_0}=a_0=0$.  If $b_s\geq 2s=2s-a_{s_0}$ for all $s=1,\hdots, k-1$, then $\x\in X_0=X_{s_0}$.  Otherwise, choose the smallest $s_1$ such that $b_{s_1}<2s_1$. Thus, $b_s\geq 2s$ whenever $s<s_1$, so by Observation \ref{O:arbrar}:
\[2s_1-2\leq b_{s_1-1}\leq a_{s_1}\leq b_{s_1}\leq 2s_1-1.\]
Continue in this way: for each $s_t$, choose the minimum $s_{t+1}=s_t+1,\hdots,k-1$ such that $b_{s_{t+1}}<a_{s_t}+2(s_{t+1}-s_t)$, if such $s_{t+1}$ exists.  Eventually this process terminates with some $s_{{u}}$, so that:
\begin{itemize}
\item $b_s\geq a_{s_t}+2(s-s_t)$ whenever $s_t\leq s\leq s_{t+1}$ for $t=0,\hdots,u-1$, 
\item $b_s\geq a_{s_{{t}}}+2(s-s_{{u}})$ whenever $s_{{u}}\leq s\leq k-1$, and
\item $b_{s_{t+1}}<a_{s_t}+2(s_{t+1}-s_t)$ for each $t=0,\hdots, u-1$.  
\end{itemize}
Hence, for each $t=0,\hdots,u-1$, Observation \ref{O:arbrar} gives:
\[a_{s_t}+2(s_{t+1}-1-s_t)\leq b_{s_{t+1}-1}\leq a_{s_{t+1}}\leq b_{s_{t+1}}\leq a_{s_{t}}+2(s_{t+1}-s_t)-1.\]
Subtracting $a_{s_t}$ from the first, middle, and last expressions gives:
\[2(s_{t+1}-s_t)-2
\leq a_{s_{t+1}}-a_{s_t}\leq
2(s_{t+1}-s_t)-1.
\]
Therefore, for any $t=0,\hdots, u-1$:
\begin{align*}
\pushQED{\qed}
a_{s_{{u}}}-a_{s_t}&=\sum_{r=t}^{u-1}(a_{s_{r+1}}-a_{s_r})\\
&\leq\sum_{r=t}^{u-1}\left(2({s_{r+1}}-{s_r})-1\right)\\
&=2(s_{{u}}-s_t)-(u-t)\\
a_{s_{{u}}}-a_{s_t}&\leq 2(s_{{u}}-s_t)-1.
\end{align*}
Rearranging gives
\begin{equation}\label{E:cover}
a_{s_{{u}}}-2s_{{u}}\leq a_{s_t}-2s_t-1.
\end{equation}
We claim that $\x\in X_{s_{{u}}}$.  This is true if (and only if) $b_s\geq a_{s_{{u}}}+2(s-s_{{u}})$ for each $s=s_{{u}},\hdots,s_{{u}}+k-1$.  Fix some $s=k,\hdots, s_{{u}}+k-1$.  Then $s_t\leq s-k\leq s_{t+1}-1$ for some $t=0,\hdots,u-1$. By construction, we have $b_{s-k}\geq a_{s_t}+2(s-k-s_t)$.  Together with (\ref{E:arar+k}) and (\ref{E:cover}), this gives:
\begin{align*}
b_s&=b_{s-k}+n\\
&\geq a_{s_t}+2(s-k-s_t)+2k-1\\
&=(a_{s_t}-2s_{t}-1)+2s\\
&\geq(a_{s_{{u}}}-2s_{u})+2s\\
&=a_{s_{{u}}}+2(s-s_{u}). \qedhere
\end{align*}
\end{proof}

\subsection{Combinatorial corollaries}\label{S:CombCor}

This subsection establishes two combinatorial corollaries, which may be of independent interest but otherwise are not needed in this paper.

We have proven that the pieces $X_r$ of the multisection of $T^n$ have disjoint interiors and cover $T^n$.  Also, each $X_r=X_0+(r,\hdots,r)$, so all $X_r$ have the same number of unit cubes. Since there are $k^n$ unit cubes in $T^n=(\R/k\Z)^n$, each $X_r$ contains $k^{n-1}$ unit cubes. By counting these unit cubes a different way, we obtain the following.\footnote{Note that by definition, if $a,b\in\Z$ with $b<0$, then $\binom{a}{b}=0$.}

\begin{cor}\label{Cor:Combo1}
For any $n=2k-1$, we have:
\begin{equation}\label{E:Combo1}
\begin{split}
k^{n-1}=&\sum_{i_0=2}^n\binom{n}{i_0}
\sum_{i_2=4-i_0}^{n-i_0}\binom{n-i_0}{i_1}
\sum_{i_3=6-i_0-i_1}^{n-i_0-i_1}\binom{n-i_0-i_1}{i_2}\cdots\\
&\cdots
\sum_{i_{k-2}=2k-2-\sum_{j=0}^{k-3}i_j}^{n-\sum_{j=0}^{k-3}i_j}\binom{n-\sum_{j=0}^{k-3}i_j}{i_{k-1}}.
\end{split}
\end{equation}
\end{cor}

Note that (\ref{E:Combo1}) is also the \href{https://oeis.org/A068087}{number of spanning trees} of the complete bipartite graph $K_{j,j}$ where $j=k$ \cite{oeis}. 

\begin{proof}
$X_0$ consists of $k^{n-1}$ subcubes, each of the form $\prod_{r=1}^n[w_r,w_r+1]$ for some $w_1,\hdots,w_n\in\Z_k$. For each $s=0,\hdots, k-2$, there are at least $2s+2$ indices among $r=1,\hdots, n$ with $w_r\in \{0,\hdots, s\}$, and conversely any subcube of that form with this property will be in $X_0$. (This characterization follows from the expression (\ref{E:X0ThmIntro}) for $X_0$.) In (\ref{E:Combo1}), each $i_s=\#\{r:w_r=s\}$, so $i_0\geq 2$, $i_0+i_1\geq 4$, and so on.  
\end{proof}

As noted above, each subcube of $T^n$ has the form $\prod_{r=1}^n[w_r,w_r+1]$ for some $w_1,\hdots,w_n\in\Z_k$. Say that two subcubes $\prod_{r=1}^n[w_r,w_r+1]$ and $\prod_{r=1}^n[w'_r,w_r'+1]$ have the same {\it combinatorial type} if $(w'_1,\hdots,w'_n)$ is a permutation of $(w_1,\hdots,w_n)$. Counting combinatorial {\it cube types} in three different ways yields:

\begin{cor}\label{Cor:Combo2}
For any $n=2k-1$, we have:
\begin{equation}\label{E:Combo2}
\begin{split}
k\sum_{i_0=2}^n~
\sum_{i_1=\max\{0,4-i_0\}}^{n-i_0}~&
\sum_{i_2=\max\{0,6-i_0-i_1\}}^{n-i_0-i_1}\cdots
\sum_{i_{k-2}=\max\left\{0,2k-2-\sum_{j=0}^{k-3}i_j\right\}}^{n-\sum_{j=0}^{k-3}i_j}1\\&=\sum_{i_0=0}^n~\sum_{i_1=0}^{n-i_0}~\sum_{i_2=0}^{n-i_0-i_1}\cdots\sum_{i_{k-2}=0}^{n-\sum_{j=0}^{k-3}i_j}1\\
&=\binom{3k-2}{k-1}.
\end{split}
\end{equation}
\end{cor}

\begin{proof}
The first expression is $k$ times the number of {\it cube types} in $X_0$, counted using the same principle and notation as in Corollary \ref{Cor:Combo1}. The second counts the number of cube types in $T^n$, each of which we may write in the form $\prod_{r=0}^{k-1}[r,r+1]^{i_r}$ and is thus characterized by a tuple $(i_0,\hdots,i_{k-1})$ with $\sum_{r=0}^{k-1}i_r=n$. The third counts the number of cube types in $T^n$ by denoting $a_0=0,a_k=3k-1$ and associating to each $A=\{a_1,\hdots,a_{k-1}\}\subset\{1,\hdots,3k-2\}$ satisfying $a_1<\cdots< a_{k-1}$ with the cube type 
\[\pushQED{\qed}
\prod_{i=1}^{k}~\prod_{j=a_{i-1}+1}^{a_i-1}[i-1,i].\qedhere\]
\end{proof}

See \cite{oeis} for \href{https://oeis.org/A045721}{other interpretations} of (\ref{E:Combo2}).

\subsection{Verification of the formula $X_I=(\ref{E:XI})$}\label{S:EXI}

Next, we will use the cutoff indices $a_r(\x),b_r(\x)$ 
 to verify (\ref{E:XI}).  To prepare this, we define subsets ${C}_{I,s}\subset T^n$ as follows.  Let $I\subset\Z_k$ following Convention \ref{Conv:ISimple}, with $s\in\Z_\ell$, and denote $i_s=i_*$.  Then define:\footnote{Note that the first line in (\ref{E:Dagger}) contributes no factors to $C_{I,s}$ if $s=0$, and likewise for the third line if $s=\ell-1$. In particular, if $I=\{0\}$, then $s=0$ and $C_{I,s}=[0,1]^2\times\cdots\times[0,k-1]^2[0,k]$, so $\lla C_{I,s}\rra=X_0$.}
\begin{equation}\label{E:Dagger}
\begin{split}
{C}_{I,s}=&\left(\prod_{t=0}^{s-1}\{i_t\}\times [i_t,i_t+1]^2\times\cdots\times[i_t,i_{t+1}-1]^2\times{[i_t,i_{t+1}]}\right)\\
&\times
[i_*,i_*+1]^2\times\cdots\times[i_*,i_{s+1}-1]^2\times{[i_*,i_{s+1}]}\\
&\times
\left(\prod_{t=s+1}^{\ell-1}\{i_t\}\times[i_t,i_t+1]^2\times\cdots\times[i_t,i_{t+1}-1]^2\times[i_t,i_{t+1}]\right).
\end{split}
\end{equation}
Note the ``missing'' $\{i_*\}$ at the start of the second line; this corresponds to the $\widehat{i_*}$ in (\ref{E:XI}).  Observe that the expression on the \textsc{rhs} of (\ref{E:XI}) equals 
\[\bigcup_{s\in\Z_\ell}\lla {C}_{I,s}\rra.\]

\begin{prop}\label{P:XIConditions}
Let $I\subset \Z_k$ follow Convention \ref{Conv:ISimple}, $s\in\Z_{\ell}$, and ${C}_{I,s}$ as in (\ref{E:Dagger}). Suppose $\x\in T^n\setminus \Delta$ is monotonic.  Then $\x\in{C}_{I,s}$ if and only if all of the following conditions hold:
\begin{itemize}
\item $b_t(\x)\geq 2t+1$ for $0\leq t<i_*$,
\item $b_t(\x)\geq 2t$ for $i_*\leq t\leq k-1$,
\item $a_t(\x)\leq 2t$ for $t=i_0,\hdots,i_*$, and
\item $a_t(\x)\leq 2t-1$ for $t=i_{s+1},\hdots,i_{\ell-1}$.
\end{itemize}
\end{prop}

\begin{proof} 
This follows immediately from the definitions, upon consideration of each entry in $\x$.
\end{proof}

Also note the following generalization of Observation \ref{O:XrCr}: 

\begin{obs}\label{O:XIPerm}
Let $I\subset \Z_k$ follow Convention \ref{Conv:ISimple}, $s\in\Z_{\ell}$, and ${C}_{I,s}$ as in (\ref{E:Dagger}). Suppose $\x\in\lla{C}_{I,s}\rra$.  Then there is a permutation $\sigma\in S_n$ such that $\x_\sigma\in{C}_{I,s}$ is monotonic. 
\end{obs}



\begin{lemma}\label{L:XI}
Given nonempty $I\subset\Z_k$ (following Convention \ref{Conv:ISimple}),
\begin{equation}\label{E:XIshort}
X_I=\bigcup_{s\in\Z_\ell} \lla{C}_{I,s}\rra.
\end{equation}
In particular,
\begin{equation}
\bigcap_{i_*\in \Z_k}X_i=\bigcup_{i_*\in I}\lla(i_1,\hdots,\widehat{i_*},\hdots,i_\ell)\prod_{i\in\Z_k}[i,i+1]\rra.
\tag{\ref{E:XZk}}
\end{equation}
\end{lemma}

Note that the formula (\ref{E:XIshort}) is equivalent to (\ref{E:XI}).

\begin{proof}
We argue by induction on $\ell$. When $\ell=1$, $X_I=X_0=\lla C_{I,0}\rra=(\ref{E:XI})$.

Assume now that $\ell>1$. 
First, we will show that 
\begin{equation}\label{E:ClaimXI1}
X_I\subset \bigcup_{s\in\Z_\ell} \lla{C}_{I,s}\rra.
\end{equation}

%
Let $\x\in X_I$, and define
$I'=I\setminus\{i_{\ell-1}\}$. Note that $I'$ is simple 
and $X_I=X_{I'}\cap X_{i_{\ell-1}}$.  Since $\x\in X_{I'}$, the induction hypothesis implies that 
$\x\in \lla{C}_{I',s_0}\rra$
for some $s_0\in 
\Z_{\ell-1}$. 
By Observation \ref{O:XIPerm}, there exists $\sigma\in S_n$ such that $\x_\sigma$ is monotonic and $\x_\sigma\in {C}_{I',s_0}$. 
Proposition \ref{P:XIConditions} implies that: 
\begin{itemize}
\item $b_t(\x_\sigma)\geq 2t+1$ for $0\leq t\leq i_{s_0}-1$,
\item $b_t(\x_\sigma)\geq 2t$ for $i_{s_0}\leq t\leq k-1$,
\item $a_t(\x_\sigma)\leq 2t$ for $t=i_0,\hdots,i_{s_0}$, and
\item $a_t(\x_\sigma)\leq 2t-1$ for $t=i_{s_0+1},\hdots,i_{\ell-2}$.
\end{itemize}

If also $a_{i_{\ell-1}}(\x_\sigma)\leq 2i_{\ell-1}-1$, then Proposition \ref{P:XIConditions} implies that 
$\x_\sigma\in {C}_{I,s_0}$.
In that case, we are done proving the forward containment. Assume instead that $a_{i_{\ell-1}}(\x_\sigma)\geq 2i_{\ell-1}$. We now split into two cases:

\underline{Case 1:} Assume that $a_{i_{\ell-1}}(\x_\sigma)=2i_{\ell-1}$. 
We claim that $\x_\sigma\in {C}_{I,\ell-1}$. By Proposition \ref{P:XIConditions}, since $\x_\sigma$ is monotonic, it will suffice to show:
\begin{enumerate}[label=(\alph*)]
\item $b_t(\x_\sigma)\geq 2t+1$ for $0\leq t\leq i_{\ell-1}-1$,
\item $b_t(\x_\sigma)\geq 2t$ for $i_{\ell-1}\leq t\leq k-1$, and
\item $a_t(\x_\sigma)\leq 2t$ for $t=i_0,\hdots,i_{\ell-1}$.
\end{enumerate}
Observation \ref{O:ark}, Proposition \ref{P:Xr}, and the facts that $\x_\sigma\in X_{i_{\ell-1}}$ and $a_{i_{\ell-1}}(\x_\sigma)=2i_{\ell-1}$ imply for each $t=0,\hdots, i_{\ell-1}-1$ that:
\begin{equation*}
\begin{split}
b_t(\x_\sigma)&=b_{t+k}(\x_\sigma)-n\\
&\geq 2(t+k)+a_{i_{\ell-1}}(\x_\sigma)-2i_{\ell-1}-n\\
&\geq2t+1.
\end{split}
\end{equation*}
This verifies (a).
Taking $t=i_{\ell-1},\hdots, k-1$, similar reasoning confirms (b): 
\begin{equation*}
b_t(\x_\sigma)\geq a_{i_{\ell-1}}(\x_\sigma)+2(t-i_{\ell-1})
\geq 2t.
\end{equation*}
Finally, we have $a_t(\x_\sigma)\leq 2t$ for each $t=i_0,\hdots,i_{\ell-1}$. For $t=i_0,\hdots,i_{\ell-2}$, this is because $\x_\sigma\in X_{I'}$; for $t=i_{\ell-1}$, it is our assumption in Case 1.  Thus, in Case 1, (a), (b), and (c) hold, and so $\x_\sigma\in{C}_{I,\ell-1}$.

\underline{Case 2:} Assume instead that $a_{i_{\ell-1}}(\x_\sigma)\geq2i_{\ell-1}+1$. Denoting $\x_\sigma=(x_1,\hdots,x_n)$, we claim in this case that $x_1=x_2=0\equiv k$ and that $\y=(x_2,\hdots,x_n,x_1)\in C_{I,\ell-1}$. By similar reasoning to Case 1, we have:
\begin{equation*}\label{E:bt}
\begin{split}
b_0(\x_\sigma)&=b_k(\x_\sigma)-n\\
&\geq a_{i_{\ell-1}}(\x_\sigma)+2(k-i_{\ell-1})-n\\
&\geq 2.
\end{split}
\end{equation*}
Thus, $x_1=x_2=0\equiv k$. Define $\y$ as above. Note that, since $\x_\sigma$ is monotonic, $\y$ is also monotonic.  It remains to show that $\y\in C_{I,\ell-1}$. The arguments are almost identical to those in Case 1, except that we need to check that $a_{i_{\ell-1}}(\y)\leq 2{i_{\ell-1}}$. Using Observations \ref{O:arbrar} and \ref{O:ark} and the fact that $a_{i_{\ell-1}}(\y)=a_{i_{\ell-1}}(\x_\sigma)-1$, we compute:
\begin{equation*}
\begin{split}
a_{i_{\ell-1}}(\y)&=a_{i_{\ell-1}}(\x_\sigma)-1\\
&\leq b_{k-1}(\x_\sigma)-2(k-1-i_{\ell-1})-1\\
&\leq a_k(\x_\sigma)-2k+1+2i_{\ell-1}\\
&=a_0(\x_\sigma)+(n+1-2k)+2i_{\ell-1}\\
&=a_0(\x_\sigma)+2i_{\ell-1}\\
&\leq2i_{\ell-1}.
\end{split}
\end{equation*}
This completes the proof of the forward containment (\ref{E:ClaimXI1}).
For the reverse containment, keep the same subset $I\subset \Z_k$ from the start of the induction step of the proof, fix some $s\in \Z_\ell$,  let
$\x\in {C}_{I,s}$ be monotonic, and let $t\in I=\{i_0,\hdots, i_{\ell-1}\}$. We will show for each $r=0,\hdots,k-1$ that $b_{t+r}(\x)\geq a_t(\x)+2r$. Proposition \ref{P:Xr} will then imply that $\x\in X_t$. Since $t$ is arbitrary, this will imply that $\x\in X_I$, completing the proof. We will split into cases, but first note, since $\x$ is monotonic, that Proposition \ref{P:XIConditions} implies:

\begin{itemize}
\item $b_t(\x)\geq 2t+1$ for $0\leq t\leq i_s-1$,
\item $b_t(\x)\geq 2t$ for $i_s\leq t\leq k-1$,
\item $a_t(\x)\leq 2t$ for $t=i_0,\hdots,i_s$, and
\item $a_t(\x)\leq 2t-1$ for $t=i_{s+1},\hdots,i_{\ell-2}$.
\end{itemize}

\underline{Case 1:} If $t+r\leq k-1$, then $b_{t+r}(\x)\geq2(t+r) \geq a_{t}(\x)+2r.$

\underline{Case 2:} If instead $ t+r\geq k $ and $t+r\leq k+i_s-1$, then
\begin{align*}
b_{t+r}(\x)&= n+b_{t+r-k}(\x)\geq n+2(t+r-k)+1=2t+2r+(n+1-2k)\\
b_{t+r}(\x)&\geq a_{t}(\x)+2r
\end{align*}
\underline{Case 3:} Similarly, if $ t+r\geq k $ and $t\geq i_{s+1}$, then
\begin{align*}
b_{t+r}(\x)&= n+b_{t+r-k}(\x)\geq n+2(t+r-k)=(2t-1)+2r+(n+1-2k)\\
b_{t+r}(\x)&\geq a_{t}(\x)+2r
\end{align*}
Are there other cases? If there were, they would satisfy $t+r\geq k+i_s$ and $t\leq i_s$, giving
\begin{align*}
k+i_s&\leq t+r\leq i_s+r\\
k&\leq r.
\end{align*}
Yet $r\leq k-1$ by assumption. Therefore, in every case, $b_{t+r}(\x)\geq a_{t}(\x)+2r$, and so $\x\in X_{i_t}$ for arbitrary $t\in I$. Thus, $\x\in X_I$. This completes the proof of the reverse containment, and thus of the equality in (\ref{E:XI})=(\ref{E:XIshort}).
\end{proof}

\section{General construction}\label{S:MainGen}

This section confirms the remaining details of our main construction and completes the proof of our main result, Theorem \ref{T:Main}.  Namely, \textsection\ref{S:Handles} describes how to decompose arbitrary $X_I$, and \textsection\ref{S:HandleProp} shows that this decomposition does in fact give an appropriate handle structure for $X_I$.

Section \ref{S:MainGen} uses Notations \ref{N:lla},
\ref{N:i},  
\ref{N:T},
and Convention \ref{Conv:ISimple}.

\subsection{Handle decompositions: the general case}
\label{S:Handles}
Throughout \textsection\ref{S:Handles}, fix arbitrary $I=\{i_s\}_{s\in\Z_\ell}=\bigsqcup_{r\in \Z_m}I_r\subsetneqq\Z_k$
, following Convention \ref{Conv:ISimple}. Recall in particular that $T=\{t\in\Z_\ell:~i_t-1\notin I\}=\{t_r\}_{r\in\Z_m}$, so that $\{\min I_r\}_{r\in\Z_m}=\{i_{t}\}_{t\in T}$.

\subsubsection{Overview}

In \textsection\ref{S:Handles}, we will decompose $X_I$ into handles in several steps as follows. First, we will decompose $X_I$ into pieces $X_{I,J,i_*}$ determined by all pairs $(J,i_*)$ where $J\subset\{\min I_r\}$ and $i_*\in I$.  Second, for fixed $(J,i_*)$, we will define disjoint subsets $U,V\subset I$ for the purpose of dividing each interval $[i-1,i]$, $i\in I$, into thirds if $i\in U$, into halves if $i\in V$, or neither if $i\notin U,V$.  Third, still fixing $(J,i_*)$, after dividing certain intervals into halves and thirds as just described, we will decompose each piece $X_{I,J,i_*}$ into pieces $X_{I,J,i_*,V^-,U^\circ,U^-}$; these pieces are determined by all triples $(V^-,U^\circ,U^-)$ where $V^-\subset V$, $U^\circ\subset U$, and $U^-\subset U\setminus U^\circ$.  For each of the first three steps, we will describe what to do within each block $I_r$; then we will take a product across all blocks and extend by permutations of the indices.

Fourth, we will order the possibilities of the tuple $(J,i_*,V^-,U^\circ,U^-)$, thus determining an order on the pieces $X_{I,J,i_*,V^-,U^\circ,U^-}$. The order will be lexicographical, and will thus require defining orders on $\{J\subset \{\min I_r\}\}$, $\{i_*\in I\}$, $\{V^-\subset V\}$, $\{U^\circ\subset U\}$, and $\{U^-\subset U\setminus U^\circ\}$. Of these five orders, only the third will be somewhat complicated.  Once we define this order, we will use it to relabel the various pieces  $X_{I,J,i_*,V^-,U^\circ,U^-}$ as $Y_z$, with $z=1,2,3,\hdots.$  Fifth and finally, we will decompose each $Y_z$ into handles, one of which we denote $Y_z^*$ (each handle $H$ from $Y_z$ is related to $Y_z^*$ by $H=\{\x_\tau:~\x\in Y_z^*\}$ for some fixed permutation $\tau\in S_n$).

\subsubsection{Decomposing $X_I$ according to $(J,i^*)$}\label{S:XIJi}

Fix arbitrary $J\subset\{\min I_r\}$ 
and $i_*\in I$ for all of \textsection\ref{S:XIJi}.
Momentarily fixing arbitrary $r\in \Z_m$, denote 
\begin{equation}
\label{E:CrHat}
a=\min I_r,~b=\max I_r,~c=\min I_{r+1},~\text{and}~\widehat{C}_r=\prod_{j=b+1}^{c-1}[b,j]^2,
\end{equation}
and define
\begin{align}
\label{E:Cr}
\begin{split}
C_r=&\left.\begin{cases}[a-1,a]&i_*=a\in J\\
[a-1,a]\times\{a\}&i_*\neq a\in J\\
\{a\}&i_*\neq a\notin J\\
\text{(no factor)}&i_*=a\notin J\\
\end{cases}\right\}\\
&\times\prod_{i= a+1}^{b}\left.\begin{cases}[i-1,i]\times\{i\}&i\neq i_*\\
[i-1,i]&i=i_*\end{cases}\right\}
\times\left.\begin{cases}\widehat{C}_r\times[b,c-1]&c\notin J\\
\widehat{C}_r&c\in J\end{cases}\right\}.
\end{split}
\end{align}
Now the piece of $X_I$ corresponding to the pair $(J,i_*)$ is given by
\begin{equation*}
X_{I,J,i_*}=\lla\prod_{r\in\Z_m} C_r\rra.
\end{equation*}

\subsubsection{The index subsets $U,V\subset I$}\label{S:UV}

Fix arbitrary $J\subset\{\min I_r\}$ 
and $i_*\in I$ for all of \textsection\ref{S:UV}. For each $r\in \Z_m$, define subsets $U_r,V_r\subset I_r$ following Table \ref{T:UV} (or equivalently according to Tables \ref{T:UV1} and \ref{T:UV2} in Appendix 1, which present $U_r$ and $V_r$ more explicitly). Note that $\min I_r\notin(U_r\cup V_r)$ unless $I_r\neq I_*$ and $\min I_r=\max I_r\in J$. See  Table \ref{T:7X02} for an example of this exceptional case: $X_I$, $I=\{0,2\}$, from $T^7$.

\begin{table}[H]
\begin{center}
\scalebox{1}{\begin{tabular}{||c|c|c|c|c||}
\hline
& $\begin{matrix}i_*\notin I_r,\\a\notin J\end{matrix}$&$\begin{matrix}i_*\notin I_r,\\a\in J\end{matrix}$& $\begin{matrix}i_*\in I_r,\\ i_*\leq b-2\end{matrix}$ & $\begin{matrix}i_*\in I_r,\\ i_*\geq b-1\end{matrix}$\\
\hline%
$U_r$&$\varnothing$&$I_r\setminus\{a,b\}$&$I_r\setminus \{a,i_*,i_*+1,b\}$&$I_r\setminus \{a,i_*,b\}$\\
$V_r$&$I_r\setminus\{a\}$&$\{b\}$&$
\{i_*+1,b\}$&$\varnothing$\\
$I_r\setminus(U_r\cup V_r)$&$\{a\}$&$\{a\}\setminus\{b\}$&$\{a,i_*\}$&$
\{a,i_*,b\}$\\
\hline
\end{tabular}}
\caption{The index subsets $U_r,V_r\subset I_r$ when $I_r=\{a,\hdots,b\}$.}\label{T:UV}
\end{center}
\end{table}
Define
\begin{equation*}
U=\bigcup_{r\in\Z_m}U_r\text{ and }
V=\bigcup_{r\in \Z_m}V_r.
\end{equation*}
%

Next, decompose each $X_{I,J,i_*}$ into pieces $X_{I,J,i_*,V^-,U^\circ,U^-}$ as follows.  Denote
\begin{equation*}
\begin{split}
2^V&=\{V^-\subset V\},\\
2^U&=\{U^\circ \subset U\},\\
\end{split}
\end{equation*}
and given $U^\circ\subset U$, denote
\begin{equation*}
2^{U\setminus U^\circ}=\{U^-\subset U\setminus U^\circ\}.
\end{equation*}
Given $V^-\subset V$, denote $V^+=V\setminus V^-$, and given $U^\circ\subset U$ and $U^-\subset {U\setminus U^\circ}$, denote $U^+=U\setminus (U^\circ \cup U^-)$. Then $V=V^-\sqcup V^+$ and $U=U^-\sqcup U^\circ\sqcup U^+$. Momentarily fixing $r\in\Z_m$, denote $a$, $b$, $c$, and $\widehat{C}_r$ as in (\ref{E:CrHat}), and for each $i\in I_r$ define
\begin{equation*}\label{E:rho}
\rho_i=
\begin{cases}
{[i-1,i-\frac{2}{3}]}&i\in U^-\\
{[i-\frac{2}{3},i-\frac{1}{3}]}&i\in U^\circ\\
{[i-\frac{1}{3},i]}&i\in U^+\\
{[i-1,i-\frac{1}{2}]}&i\in V^-\\
{[i-\frac{1}{2},i]}&i\in V^+\\
{[\max I_{r-1},i-1]}&i=a\notin J\cup V\\
{[i-1,i]}&\text{else}.\\
\end{cases}
\end{equation*}
Note that $\rho_i\subset [i-1,i]$ for each $i=a+1,\hdots, b$, that $\rho_a\subset[a-1,a]$ if $a\in J$, and that $\rho_{c}=[b,c-1]$ if $c\notin J$.  Still fixing $r\in\Z_m$, define
\begin{equation*}\label{E:XIJiVUr1}
\begin{split}
X_{I,J,i_*,V^-,U^\circ,U^-,r}=&\left.\begin{cases}\rho_a&i_*=a\in J\\
\rho_a\times\{a\}&i_*\neq a\in J\\
\{a\}&i_*\neq a\notin J\\
\text{(no factor)}&i_*=a\notin J\\
\end{cases}\right\}\\
&\times\prod_{i= a+1}^{b}\left.\begin{cases}\rho_i\times\{i\}&i\neq i_*\\
\rho_i&i=i_*\end{cases}\right\}
\times\left.\begin{cases}\widehat{C}_r\times\rho_{c}&c\notin J\\
\widehat{C}_r&c\in J\end{cases}\right\}.
\end{split}
\end{equation*}
The piece of $X_I$ corresponding to the tuple $(J,i_*,V^-,U^\circ,U^-$ is:
\begin{equation*}\label{E:XIJiVU}
\begin{split}
X_{I,J,i_*,V^-,U^\circ,U^-}=&\lla \prod_{r\in \Z_m}X_{I,J,i_*,V^-,U^\circ,U^-,r}\rra
\end{split}
\end{equation*}
Note that $X_{I,J,i_*}=\bigcup_{V^-,U^\circ,U^-}X_{I,J,i_*,V^-,U^\circ,U^-}$.  

\subsubsection{Ordering the pieces $X_{I,J,i_*,V^-,U^\circ,U^-}$}\label{S:Order}

Next, we define orders $\prec$ on $\{J\subset \{\min I_r\}\}$, $I$, $2^{V}$, $2^{U}$, and $2^{U\setminus U^\circ}$ and use these to order the pieces $X_{I,J,i_*,V^-,U^\circ,U^-}$ lexicographically and then relabel them as $Y_1,Y_2,Y_3,\hdots$.

Order $\{J\subset \{\min I_r\}\}$ and $2^U$ partially by inclusion, so that $J'\prec J$ if $J'\subsetneqq J$ and $U'^\circ\prec U^\circ$ if $U'^\circ\subsetneqq U^\circ$; extend these partial orders arbitrarily to total orders.  Define an arbitrary total order $\prec$ on $2^{U\setminus U^\circ}$.  
Partially order $I$ such that $i<i'$ if  $i\in I_r$, $i'\in I_s$, and $i-\min I_{r}<i_s-\min I_{s}$; extend arbitrarily to a total order on $I$. 

It remains to order $2^V$. This will be slightly more complicated. To do this, we first define a total order $\prec_r$ on $2^{V_r}$ for each $r\in \Z_m$. First consider the case $I_r\ni i_*$, i.e. $I_r=I_*$. If $i_*\geq \max I_*-1$, we have $V_r=\varnothing$, so there is nothing to do. Otherwise, we have $i_*\leq \max I_*-2$ and $V_r=\{i_*+1,\max I_*\}$; in this case, order $2^{V_r}$ as follows:
\begin{equation*}
\{i_*+1\}\prec_r\varnothing \prec_r\{i_*+1,\max I_*\}\prec_r\{\max I_*\}.
\end{equation*}
Now consider the case $I_r\not\ni i_*$.  
Define $\prec_r$ on $2^{V_r}$ recursively by ${V^-_r}\prec_rV'^-_r$ if:
\begin{itemize}
\item $\max V^-_r<\max V'^-_r$, or
\item $\max V^-_r=\max V'^-_r$ and $V'^-_r\setminus\{\max V'^-_r\}\prec_rV^-_r\setminus\{\max V^-_r\}$.
\end{itemize}
Note the reversal of order on the line above.  If we assume without loss of generality that ${V}_r=\{0,\hdots,b\}$, we can write the order explicitly:
\begin{align}\label{E:order}
\begin{split}
\varnothing&\prec_r\{0\}\prec_r\{0,1\}\prec_r\{1\}\prec_r\{1,2\}
\prec_r\{0,1,2\}\prec_r\{0,2\}\prec_r\{2\}\\
&\prec_r\{2,3\}\prec_r\{0,2,3\}\prec_r\{0,1,2,3\}\prec_r\{1,2,3\}\prec_r\{1,3\}\prec_r
\cdots\\
\cdots&\prec_r\{0,1,2,b\}\prec_r\{1,2,b\}\prec_r\{1,b\}\prec_r\{0,1,b\}\prec_r\{0,b\}\prec_r\{b\}.
\end{split}
\end{align}
See Tables \ref{T:13} and \ref{T:15}, and the part of Table \ref{T:111} where $i_*=4$.

Use the orderings $\prec_r$ on $2^{V_r}$ to define a partial order on $2^V$ by declaring $V^-\prec V'^-$ if 
\begin{itemize}
\item $V^-\cap I_r\prec_r V'^-\cap I_r$ for some $r$, and
\item there is no $r$ for which $V'^-\cap I_r\prec_r V^-\cap I_r$.
\end{itemize}
Extend $\prec$ arbitrarily to a total order on $2^V$. %
This determines a total order on 
\begin{equation}\label{E:ToOrder}
\left\{(J,i_*,V^-,U^\circ,U^-)\right\}_{J\subset T,~i_*\in I,~V^-\subset {V},~U^\circ \subset {U},~U^-\subset {U\setminus U^\circ}},
\end{equation}
and thus on the pieces $X_{I,J,i_*,V^-,U^\circ,U^-}$. Relabel these pieces as $Y_z$, $z=1,\hdots,\#(\ref{E:ToOrder})$, according to this order.  

\subsubsection{Decomposing each $Y_z$ into handles}\label{S:Yzstar}

Each $Y_z$ is now given by an expression of the form
\begin{equation}\label{E:xi}
\lla\prod_{r=1}^{n}\chi_r\rra,
\end{equation}
where each $\chi_r$ is either a closed interval or a singleton.  Fixing arbitrary $z$, use the expression (\ref{E:xi}) to define the coarsest equivalence relation $\sim$ on $\{1,\hdots,n\}$ that obeys the following property: whenever $\chi_r\subset\chi_s$, we have $r\sim s$.  Denote the set of equivalence classes under $\sim$ by $P=\{R_1,\hdots,R_p\}$, and for each $r=1,\hdots,p$, denote $\lla \prod_{s\in R_r}\chi_s\rra=\xi_r$.  Define
\begin{equation}\label{E:Yzstar}
Y_z^*=\prod_{r=1}^p\xi_r.
\end{equation}
In \textsection\ref{S:HandleProp}, we will see that each $Y_z^*$ is a handle, and that attaching $Y_z$ to $\bigcup_{s=1}^{z-1}Y_{s}$ amounts to attaching a collection of handles, each of which is related to $Y_z^*$ as follows. Let 
\[G=\{\sigma\in S_n:~\vec{x}_\sigma\in Y_z^*\text{ whenever }\vec{x}\in Y_z^*\}=S_{n_{|R_1|}}\times\cdots\times S_{n_{|R_p|}}\] consist of the permutations on the indices of $T^n$ which fix $Y_z^*$ setwise.  Then there is a one-to-one correspondence between the left cosets of $G$ and the handles comprising $Y_z$:
\[\tau G\longleftrightarrow \{\x_\tau:~\x\in Y_z^*\}.\]

\begin{example}\label{Ex:T9}
Consider $X_I\subset T^9$ where $I=\{0,1,2,3\}$, which is detailed in Tables \ref{T:92} and \ref{T:93}. Note that $T=\{0\}$. 
In particular, consider the first and twelfth rows of Table \ref{T:92} (after the headings), where $J=\varnothing$, $i_*=0$, $U=\{2\}$, and $V=\{1,3\}$. The first row of Table \ref{T:92} corresponds to 
\begin{equation}\label{E:Y1}
Y_1=X_{I,J,s,V^-,U^\circ,U^-}=\lla{\alpha^-}1{\beta^\circ_3}2{\gamma^+}3\delta^{{3}}\rra,
\end{equation}
where $V^-=\{1\}$, $U^\circ=\{2\}$, and $U^-=\varnothing$ with
\begin{equation*}
\begin{split}\chi_1&={\alpha^-}=\left[0,\frac{1}{2}\right],~\chi_2=\{1\},~\chi_3={\beta^\circ_3}=\left[\frac{4}{3},\frac{5}{3}\right],~\chi_4=\{2\},\\
\chi_5&={\gamma^+}=\left[\frac{5}{2},3\right],~\chi_6=\{3\},~\text{and }\chi_7=\chi_8=\chi_9=\delta=[3,4].
\end{split}
\end{equation*}
The ensuing partition of $\{1,\hdots,9\}$ gives
\[P=\{\{1\},\{2\},\{3\},\{4\},\{5,6,7,8,9\}\},\]
and so 
\begin{align*}Y_1^*&=\chi_1\times\chi_2\times\chi_3\times\chi_4\times\lla\chi_5\times\chi_6\times\chi_7\times\chi_8\times\chi_9\rra\\
&={\alpha^-}1{\beta^\circ_3}2\lla{\gamma^+}3\delta^{{3}}\rra,
\end{align*}
where
\[\xi_1={\alpha^-},~\xi_2=\{1\},~\xi_3={\beta^\circ_3},\xi_4=\{2\},~\text{and}~\xi_5=\lla{\gamma^+}3\delta^{{3}}\rra.\]
The twelfth row of Table \ref{T:92} corresponds to 
\begin{equation*}\label{E:Y12}
Y_{12}=X_{I,J,s,V^-,U^\circ,U^-}=\lla{\alpha^+}1{\beta^+_3}2{\gamma^-}3\delta^{{3}}\rra,
\end{equation*}
where $V^-=\{3\}$, $U^\circ=\varnothing=U^-$.  The ensuing partition of $\{1,\hdots,9\}$ gives
\[P=\{\{1,2\},\{3,4,5\},\{6,7,8,9\}\},\]
and so 
\begin{align*}Y_{12}^*&=\lla\chi_1\times\chi_2\rra\times\lla\chi_3\times\chi_4\times\chi_5\rra\times\lla\chi_6\times\chi_7\times\chi_8\times\chi_9\rra\\
&=\underbrace{\lla{\alpha^+}1\rra}_{\xi_1}\underbrace{\lla{\beta^+_3}2{\gamma^-}\rra}_{\xi_2}\underbrace{\lla3\delta^{{3}}\rra}_{\xi_3}.
\end{align*}
\end{example}

\subsection{Properties of handle decompositions}\label{S:HandleProp}

\subsubsection{Combinatorics}

\begin{prop}\label{P:VOrder}
Let $i\in I_s\cap V^-$ for some  $s\in \Z_m$, where $i_*\notin I_s$.  Denote $b=\max I_s$, $c=\max (I_s\cap V^-)$. Let $V'^-=V^-\setminus\{i\}$.  Then $V'^-\prec V^-$ if and only if $|V^-\cap\{i+1,\hdots,b\}|$ is even.
\end{prop}

\begin{proof}
We argue by induction on $c-i$.  When $c-i=0$, we have $c=i>\max (I_s\cap V^-\setminus\{i\})$ and $I_r\cap V^-=I_r\cap V^-\setminus\{i\}$ for all $r\neq s$, so $V'^-\prec V^-$.

Now assume that $c-i=t>0$, and assume that the claim is true whenever $\max (I_s\cap V^-)-i<t$. Let $W^-=V^-\setminus \{c\}$ and $W'^-=V'^-\setminus\{c\}$.  
Then $|V^-\cap\{i+1,\hdots,b\}|$ and $|W^-\cap\{i+1,\hdots,b\}|$ have opposite parities. Also, by construction, $V^-\prec V'^-$ if and only if $W'^-\prec W^-$. The result now follows by induction.
\end{proof}

\begin{notation}\label{N:ominus}
Denote the symmetric difference of sets $R$ and $S$ by 
\[R\ominus S=(R\setminus S)\cup (S\setminus R).\]
\end{notation}

\begin{prop}\label{P:VOrderSet}
Let $A\subset V$ such that $V^-\prec V^-\ominus\{a\}$ for each $a\in A$. Then $V^-\prec V^-\setminus A$.
\end{prop}

\begin{proof}
Suppose first that $A\subset I_s$ for some $s\in \Z_m$.  Denote $A=\{a_1,\hdots,a_q\}$ with $\min I_s\leq a_1\leq \cdots\leq a_q\leq \max I_s=b$.
Assume that $i_*\notin I_s$ and $|I_s|\geq 3$ (the other cases are trivial). Proposition \ref{P:VOrder} implies, for each $a\in A$, that $|V^-\cap\{a+1,\hdots,b\}|$ is odd if and only if $a\in V^-$.  For each $r=1,\hdots, q$, denote the symmetric difference $V^-_r=V^-\ominus\{a_1,\hdots,a_r\}$.  Then, $|V^-_a\cap\{a+1,\hdots,b\}|=|V^-\cap\{a+1,\hdots,b\}|$ for each $a=0,\hdots,q-1$. Since this quantity is odd if and only if $a\in V^-$, Proposition \ref{P:VOrder} implies:
\[V^-\prec V^-_1\prec\cdots\prec V^-_q=V^-\setminus A.\]
For the general case, apply this argument repeatedly for each $s\in \Z_m$.
\end{proof}

\subsubsection{Topology}

\begin{obs}\label{O:U-Disjoint}
In $X_I$, if $Y_z$ comes from $(J,i_*,V^-,U^\circ,U^-)$ and $Y_w$ comes from $(J,i_*,V^-,U^\circ,U'^-)$, then $Y_z\cap Y_w=\varnothing$ unless $U^-=U'^-$. That is, if $U^-\neq U'^-$, then 
\[X_{I,J,i_*,V^-,U^\circ,U^-}\cap X_{I,J,i_*,V^-,U^\circ,U'^-}=\varnothing.\]
\end{obs}

\begin{lemma}\label{L:Yzr}
Each factor $\xi_r$  in the expression  (\ref{E:Yzstar}) for $Y_z^*$   has one of the forms described in Lemma  \ref{L:StarShaped}, and thus is PL homeomorphic to  $D^{d(r)}$ for some $d(r)\geq 0$.

Moreover, $\sum_{r=1}^{p}d(r)=n+1-|I|$, so $Y_z^*\cong D^{n+1-|I|}$.
\end{lemma}
\begin{proof}
Regarding the first claim, we examine the equivalence relation $\sim$ that led to (\ref{E:Yzstar}). Suppose $\chi_r\subset\chi_{r'}$.  Then, by construction, either $\chi_r$ is a singleton (in $I\setminus \{i_*\}$) and $\chi_{r'}$ is an interval with this singleton as an endpoint, or else $\chi_r\supset [\max I_s,\max I_s+1]$ for some $s\in\Z_\ell$. Moreover, by construction, if $\chi_{r'}$ contains a point of $I\setminus\{i_*\}$, then it contains only one such point and it contains no interval of the form $[\max I_s,\max I_s+1]$, and no $\chi_{r'}$ contains more than one interval of the form $[\max I_s,\max I_s+1]$. The first claim now follows. (For an explicit accounting of the types of factors $\xi_r=\lla\prod_{s\in R_r}\chi_s\rra$, see Tables \ref{T:YzrNoSing}, \ref{T:YzrRe}, and \ref{T:YzrSing}.)

Regarding the second claim, note for each $r=1,\hdots,p$, that $d(r)$ equals the number of intervals among $\{\chi_s\}_{s\in R_r}$, which equals the order of $R_r$ minus the number of singletons among $\{\chi_s\}_{s\in R_r}$. Since $\sum_{r=1}^p|R_r|=|P|=n$ and $\{\chi_s:~s=1,\hdots n\}$ contains a total of $|I|-1$ singletons, it follows that $\sum_{r=1}^{p}d(r)=n+1-|I|$. Thus, $Y_z^*\cong D^{n+1-|I|}$.
\end{proof}

We wish to show, in arbitrary $X_I$, that attaching any $Y_z$ to $\bigcup_{w<z}Y_w$ amounts to attaching a collection of $(n+1-|I|)$-dimensional $h(z)$-handles for some $h(z)$.  Indeed, Lemma \ref{L:Yzr} confirms that each $Y_z^*$ from $X_I$ is a compact $(n+1-|I|)$-ball, so it remains to consider how everything is glued together.  Our goal is to show that 
\begin{equation}\label{E:hHandle}
Y_z^*\cap \bigcup_{w<z}Y_w\cong S^{h(z)-1}\times D^{n+1-|I|-h(z)}
\end{equation}
and
\begin{equation}\label{E:Within}
Y_z^*\cap (Y_z\setminus\setminus Y_z^*)\subset Y_z^*\cap \bigcup_{w<z}Y_w.
\end{equation}
The former will imply that attaching $Y_z^*$ to $\bigcup_{w<z}Y_w$ amounts to attaching an $(n+1-|I|)$-dimensional $h(z)$-handle, and the latter will further imply that if we attach all the copies of $Y_z^*$ one at a time to $\bigcup_{w<z}Y_w$, then attaching each copy amounts to attaching another $(n+1-|I|)$-dimensional $h(z)$-handle. 

Recall that each $Y_z^*$ has the form $\prod_{r=1}^p\xi_r(z)$. Hence, 
\[\partial Y_z^*=\bigcup_{a=1}^p\left(\prod_{r=1}^{a-1}\xi_r(z)\times\partial\xi_a(z)\times\prod_{r=a+1}^p\xi_r(z)\right).\]
We will show, given arbitrary $Y_z^*$ in $X_I$, that there is a subset $S(z)\subset \{1,\hdots, p\}$ such that 
\begin{equation}\label{E:Sz}
Y_z^*\cap \bigcup_{w<z}Y_w=\bigcup_{a\in S(z)}\left(\prod_{r=1}^{a-1}\xi_r(z)\times\xi_a(z)\times\prod_{r=a+1}^p\xi_r(z)\right).
\end{equation}
Then, denoting $h(z)=\sum_{r\in S(z)}\text{dim}(\xi_r)$, we will obtain (\ref{E:hHandle}):
\begin{align*}
Y_z^*\cap \bigcup_{w<z}Y_w&=\bigcup_{a\in S(z)}\left(\prod_{r=1}^{a-1}\xi_r(z)\times\xi_a(z)\times\prod_{r=a+1}^p\xi_r(z)\right)\\
&\cong\left(\partial \prod_{r\in S(z)}\xi_r(z)\right)\times\prod_{r\notin S(z)}\xi_r(z)\\
&\cong \partial D^{h(z)}\times D^{n+1-|I|-h(z)}\\
&=S^{h(z)-1}\times D^{n+1-|I|-h(z)},
\end{align*}

Our next step is to describe the subset $S(z)\subset \{1,\hdots,p\}$.  To do so, we characterize each $\xi_r(z)$ as type $\red{\text{(A)}}$ or type $\Navy{\text{(B)}}$; then $S(z)$ will consist of those $r=1,\hdots,p$ for which $\xi_r(z)$ has type $\red{\text{(A)}}$. After that, Lemmas \ref{L:Prec} and \ref{L:PrecConverse} will establish (\ref{E:Sz}) by double containment, implying (\ref{E:hHandle}), and Lemma \ref{L:PrecWithin} will establish (\ref{E:Within}).  

Consider an arbitrary $Y_z^*=\prod_{r=1}^p\xi_r(z)$ from an arbitrary $X_I$.  In the following way, classify each factor $\xi_r(z)$ into one of two classes, $\red{\text{(A)}}$ or $\Navy{\text{(B)}}$.  Say that $\xi_r(z)$ is in class $\Navy{\text{(B)}}$ if \begin{itemize}
\item $\xi_r(z)=\left[i-\frac{2}{3},i-\frac{1}{3}\right]$ for some $i\in I$;
\item $\xi_r(z)=\left[i_*,i_*+\frac{1}{2}\right]$; 
\item $[\max I_s,j]$ is a factor in the expression for $\xi_r(z)$ for some $s,j$; or
\item Some $\{i\}$ is a factor in the expression for $\xi_r(z)$ and:
\begin{itemize}
\item $i\in V^+$ and $i+1\in U^\circ\cup U^+\cup V^+$, or $i\in U^-\cup U^\circ\cup V^-$ and $i+1\in V^-$; and
\item $|V^-\cap\{i+1,\hdots,\max I_s\}|$ is even, where $i\in I_s$.
\end{itemize}
\end{itemize}
All other types of $\xi_r(z)$ are of class $\red{\text{(A)}}$.  Tables \ref{T:YzrNoSing}, \ref{T:YzrRe}, and \ref{T:YzrSing} in Appendix 1 list the possibilities explicitly.

\begin{lemma}\label{L:Prec}
Suppose $Y_z^*=\prod_{r=1}^p\xi_r(z)$ comes from $(J,i_*,V^-,U^\circ,U^-)$. If, for some $a=1,\hdots,p$, $\xi_a(z)$ is of class $\red{\text{(A)}}$ and 
\[\x=(x_1,\hdots,x_n)\in\prod_{r=1}^{a-1} \xi_r(z)\times \partial\xi_a(z)\times \prod_{r=a+1}^{p}\xi_r(z),\] then $\x\in Y_w$ for some $w<z$.
\end{lemma}

\begin{proof}
Suppose first that some $\{i\}$ appears in the expression for $\xi_a(z)$, with $i\in I_s$; $i\in V^+$ and $i+1\in U^\circ\cup U^+\cup V^+$, or $i\in U^-\cup U^\circ\cup V^-$ and $i+1\in V^-$; and $|V^-\cap\{i+1,\hdots,\max I_s\}|$ is odd. Then $\x$ is in the $Y_w$ coming from $(J,i_*,V'^-,U^\circ,U^-)$ where $V'^-$ is either $V^-\cup\{i\}$ or $V^-\setminus\{i+1\}$.  In either case, Proposition \ref{P:VOrder} implies that $V'^-\prec V$ and thus $w<z$. 

Next, suppose that $\xi_a(z)$ has no singleton factors.  There are two possibilities. If  $\xi_a(z)=[i_*-1,i_*]$ with $i_*\in J$, then $\x$ is in some $Y_w$ coming from $J\setminus\{i_*\}\prec J$. Otherwise, $\xi_a(z)=\left[i-1,i-\frac{1}{2}\right]$ for some $i\in J\cap V^-$; in this case, $i+1\notin I$, and so $\x$ is in some $Y_w$ coming either from $J\setminus\{i\}\prec J$ or the from same $J$ and $i_*$ and $V'^-=V^-\setminus\{i\}$, where Proposition \ref{P:VOrder} implies that $V'^-\prec V^-$ because $i+1\notin I$.  

The remaining cases follow by similar reasoning. The interested reader may find Table \ref{T:YzrSing} useful for this.
\end{proof}

\begin{lemma}\label{L:PrecConverse}
Let $Y_z^*=\prod_{r=1}^p\xi_r(z)$ come from some $(J,i_*,V^-,U^\circ,U^-)$. If
\[\x=(x_1,\hdots,x_n)\in Y_z^*\cap\bigcup_{w<z}Y_w,\]
then 
\[\x\in\prod_{r=1}^{a-1} \xi_r(z)\times \partial\xi_a(z)\times \prod_{r=a+1}^{p}\xi_r(z)\] 
for some $a=1,\hdots,p$, such that $\xi_a(z)$ is of class $\red{\text{(A)}}$.
\end{lemma}

\begin{proof}
Let $\x=(x_1,\hdots,x_n)\in Y_z^*\cap Y_{w'}$ for some $w'<z$. Choose the smallest $w<z$ such that $\x\in Y_w$, and assume that $Y_w$ comes from some $(J',i'_*,V'^-,U'^\circ,U'^-)$ with $V'^-\subset V'$ and $U'^\circ\subset U'$, whereas $Y_z$ comes from some $(J,i_*,V^-,U^\circ,U^-)$ with $V^-\subset V$ and $U^\circ \subset U$. Denote 
\[S=\left\{a=1,\hdots,p:~\x\in \prod_{r=0}^{a-1} \xi_r(z)\times \partial\xi_a(z)\times \prod_{r=a+1}^{p}\xi_r(z)\right\}.\]
Assume for contradiction that $\xi_a(z)$ is of class $\Navy{\text{(B)}}$ for every $a\in S$.  If $S=\varnothing$, then no coordinate of $\x$ equals $i_*$, so $i'_*=i_*$.  Also, in that case, no coordinate of $\x$ equals $\min I_s-1$ for any $s\in\Z_m$, and so $J$ and $J'$ completely determine the number of coordinates that $\x$ has in each open interval $(\min I_s-1,\min I_{s+1}-1)$.  It follows that either $J'=J$ or $J'=T\setminus J$.  
If $J'=T\setminus J$, then considering the coordinates of $\x$ in $[\min I_*,\max I_*]$ yields a contradiction. 
 If $J'=J$, then the fact that $S=\varnothing$ implies that $V^-=V'^-$, $U^\circ=U'^\circ$, and $U^-=U'^-$, contradicting the fact that $w<z$.

Therefore, $S\neq \varnothing$. If no coordinate of $\x$ equals $i_*$, then $i'_*=i_*$, so again either $J'=J$ or $J'=T\setminus J$.  The latter case gives the same contradiction as before.  Therefore $J'=J$, and so $V'=V$. 

For each $i\in V^-\ominus V'^-$, $\x$ has a coordinate $x_t=i-\frac{1}{2}$ (using the fact that $i'_*=i_*$ and $J'=J$).  The corresponding $\xi_r(z)$ has $r\in S$, and so by assumption $\xi_r(z)$ is of class $\Navy{\text{(B)}}$. Therefore, $V^-\prec V^-\ominus\{i\}$ for each $i\in V^-\ominus V'^-$. Proposition  \ref{P:VOrderSet} implies that $V^-\prec V'^-$ unless $V^-=V'^-$.  Since $w<z$, we must have $V^-=V'^-$.  

Each $i\in U'^\circ$ must also be in $U^\circ$, or else the corresponding coordinate of $\x$ would equal $i-\frac{1}{3}$ or $i-\frac{2}{3}$, and the corresponding $\xi_a(z)$ would be of class $\red{\text{(A)}}$ with $a\in S$, contrary to assumption.  Thus, $U^\circ\subset U'^\circ$. Similarly, each $i\in U^\circ$ must also be in $U'^\circ$, or else the $Y_{w'}$ coming from $J,i_*,V,U'^\circ\cup\{i\},U^-\setminus\{i\}$ would still contain $\x$ but with $w'<w$, contrary to assumption.  Thus, $U'^\circ=U^\circ$.  

Finally, we must have $U'^-=U^-$, by Observation \ref{O:U-Disjoint}.  This implies, contrary to assumption, that $Y_w=Y_z$.
\end{proof}  

\begin{lemma}\label{L:PrecWithin}
Let $Y_z^*=\prod_{r=1}^p\xi_r(z)$ come from some $(J,i_*,V^-,U^\circ,U^-)$. If
\[\x=(x_1,\hdots,x_n)\in Y_z^*\cap(Y_z\setminus\setminus Y_z^*),\]
then 
\[\x\in\prod_{r=1}^{a-1} \xi_r(z)\times \partial\xi_a(z)\times \prod_{r=a+1}^{p}\xi_r(z)\] 
for some $a=1,\hdots,p$, such that $\xi_a(z)$ is of class $\red{\text{(A)}}$.
\end{lemma}

\begin{proof}
This follows from a case analysis, for which the interested reader may find Tables \ref{T:YzrNoSing}--\ref{T:YzrSing} useful. It comes down to this. Consider two pieces $\xi_a(z)$ and $\xi_b(z)$ of $Y_z^*$ for which the infimum $\min \xi_b(z)$ of all coordinates in $(0,k)$ among all points in $\xi_b(z)$ equals the supremum $\max \xi_a(z)$ of all coordinates in $(0,k)$ among all points in $\xi_a(z)$. Denote  $\max \xi_a(z)=\min \xi_b(z)=c$. Then $c\in\Z_k$. If $c$ equals $i-1$ for some $i\in T$, then $i\in J$ and $\xi_b(z)$ is of class $\red{\text{(A)}}$.  Otherwise, $c=i_*$ and $\xi_a(z)$ is of class $\red{\text{(A)}}$. 
\end{proof}

\subsection{Proof of the main result}

The results of \textsection\textsection\ref{S:Cutoff}, \ref{S:HandleProp} provide all the details we need to prove:

\begin{theorem}\label{T:Main}
For $n=2k-1\in\Z_+$, the $n$-torus admits a multisection $T^n=\bigcup_{r\in\Z_k}X_r$ defined by
\begin{equation}
\begin{split}
X_0&=\left\{\x_\sigma:~\x\in [0,1]^2\cdots[0,k-1]^2[0,k]/\sim,~\sigma\in S_n\right\},\\
X_i&=\{\x+(i,\hdots,i):~\x\in X_0\}.
\end{split}\tag{\ref{E:X0ThmIntro}}
\end{equation}
\end{theorem}

\begin{proof}
Lemma \ref{L:cover} implies that $X=\bigcup_{i\in\Z_k}X_i$, so it remains only to prove for each nonempty proper subset $I\subset \Z_k$, that $X_I=\bigcap_{i\in I}X_i$
is an $(n+1-|I|)$-dimensional submanifold of $X$ with a spine of dimension $|I|$.

Fix some such $I$.  Assume \textsc{wlog} that $I$ is simple.  Then $X_I=(\ref{E:XI})$, by Lemma \ref{L:XI}.  Decompose $X_I=\bigcup_zY_z$ as described in \textsection\ref{S:Handles}.  Lemmas \ref{L:Yzr} and \ref{L:Prec} imply that $Y_1^*$ is an $(n+1-|I|)$-dimensional 0-handle with no pieces $\xi_r(1)$ of class $\red{\text{(A)}}$; Lemma \ref{L:PrecWithin} and the symmetry of the construction imply further that $Y_1$ is a union of $(n+1-|I|)$-dimensional  0-handles. 

For each $z$, denote $S(z)=\{r:~\xi_r(z) \text{ is of class}$ $\red{\text{(A)}}\}$. Lemmas \ref{L:Yzr}, \ref{L:Prec}, and \ref{L:PrecConverse} imply that attaching $Y_z^*$ to $\bigcup_{w<z}Y_z$ amounts to attaching an $(n+1-|I|)$-dimensional $h$-handle, where $h(z)$ is the sum of the dimensions of those $\xi_r(z)$ of class $\red{\text{(A)}}$:
\[h(z)=\sum_{r\in S(z)}\text{dim}(\xi_r(z))\leq |I|.\]

Lemma \ref{L:PrecWithin} and the symmetry of the construction imply further that attaching all of $Y_z$ to $\bigcup_{w<z}Y_w$ amounts to attaching several such handles. Thus, $X_I$ is an $(n+1-|I|)$-dimensional $|I|$-handlebody in $T^n$.

It remains to check that $X_{\Z_k}=\bigcap_{i\in\Z_k}X_i$ is a closed $k$-manifold. We know from Lemma \ref{L:XI} that $X_{\Z_k}$ is given by (\ref{E:XZk}). 

Since $X_{\Z_k\setminus\{k-1\}}$ is $(k+1)$-manifold, it suffices to check that $X_{\Z_k}$ equals $\partial X_{\Z_k\setminus\{k-1\}}$, which is the union of those $k$-faces of the $Y_z$ from the handle decomposition  of $X_{\Z_k\setminus\{k-1\}}$ that are not glued to any other $Y_w$.  Case analysis confirms that this union equals the expression from (\ref{E:XZk}). (The reader may find Tables \ref{T:YzrNoSing}-\ref{T:YzrSing} useful).
\end{proof}

Alternatively, one can construct a handle decomposition of $X_{\Z_k}$ as follows. Cut each unit interval $[i,i+1]$ into thirds and, for each $i_*\in\Z_k$, further cut $\left[i_*-\frac{1}{3},i_*\right]$ and $\left[i_*,i_*+\frac{1}{3}\right]$ into halves.  Then, for each $i_*\in \Z_k$, $U^\circ\subset \Z_k$, $U^- \subset \Z_k\setminus U^\circ$, and $U^*\subset(\{i_*+1\}\cap U^-)\cup(\{i_*\}\setminus (U^\circ\cup U^-)$, define
\begin{align*}
\rho_i&=\begin{cases}
{[i-\frac{2}{3},i-\frac{1}{3}]}&i\in U^\circ\\
{[i-1,i-\frac{2}{3}]}&i_*+1\neq i\in U^-\\
{[i-\frac{1}{3},i]}&i_*\neq i\in \Z_k\setminus (U^\circ\cup U^-)\\
{[i_*,i_*+\frac{1}{6}]}&i_*+1= i\in U^*\\
{[i_*+\frac{1}{6},i_*+\frac{1}{3}]}&i_*+1= i\in U^-\setminus U^*\\
{[i_*-\frac{1}{6},i_*]}&i_*= i\in U^*\\
{[i_*-\frac{1}{3},i_*-\frac{1}{6}]}&i_*= i\in U^+\setminus U^*,\\
\end{cases}\\
X_{\Z_k,i_*,U^\circ,U^-,U^*}&=\prod_{i\in\Z_k}\left.\begin{cases}
\rho_i\times\{i\}&i\neq i_*\\
\rho_i&= i_*\\
\end{cases}
\right\}.
\end{align*}
Order the pieces $X_{\Z_k,i_*,U^\circ,U^-,U^*}$ as $Y_z$, $z=1,2,3,\hdots,$ lexicographically according to the following orders on the possibilities for $(i_*,U^\circ,U^-,U^*)$. Order $\{i_*\in I\}$ and $U^-\subset U^\circ$ arbitrarily. Partially order $\{U^\circ\subset\Z_k\}$ by inclusion, with $U^\circ\prec U'^\circ$ if $U^\circ \subset U'^\circ$, and extend arbitrarily to a total order. Order the possibilities for $U^*$ the same way. Then
\begin{equation*}
\bigcup_{i=1,\hdots,k}Y_z=\bigcup_{i_*\in \Z_k}X_{\Z_k,i_*,\Z_k,\varnothing,\varnothing}
\end{equation*}
is a union of 0-handles, and to attach each $Y_z=X_{\Z_k,i_*,U^\circ,U^-,U^*}$ to $\bigcup_{w<z}Y_w$ is to attach a collection of $h(z)$-handles for $h(z)=k-|U^\circ|-|U^*|$.  

We leave the following question open:

\begin{question}
Are the multisections in Theorem \ref{T:Main} smoothable?
\end{question}

That is, for odd $n$, does $T^n$ (under its standard smooth structure) admit a smooth multisection such that, when one passes to the unique PL structure on $T^n$, there is a PL homeomorphism $f:T^n\to T^n$ sending each piece of this smooth multisection to a piece of the multisection from Theorem \ref{T:Main}? 


\section{Cubulated manifolds of odd dimension}\label{S:Cube}

This section extends Theorem \ref{T:Main} to certain cubulated manifolds.  Consider a covering space $p:M\to T^n$, where $n=2k-1$.  Multisect $T^n=\bigcup_{i\in\Z_k}X_i$ as in Theorem \ref{T:Main}.  Then, by Corollary 17 of \cite{rt}, $M=\bigcup_{i\in\Z_k}p^{-1}(X_i)$ determines a PL multisection of $M$.  In general, one expects such multisections to be less efficient than those from Theorem \ref{T:Main}. Also, there seems to be no reason to expect that one can extend the main construction to cubulated odd-dimensional manifolds in general.  There is, however, an intermediate case to which our construction does extend. 

First, we propose a modest generalization of the usual notion of a cubulation. The generalization is similar to Hatcher's $\Delta$-complexes vis a vis simplicial complexes \cite{hatcher}.  A {\it cube} is a homeomorphic copy of $I^n$ for some $n\geq0$, with the usual cell structure; its {\it faces} are defined in the traditional way. 

Consider an arbitrary edge of $I^n$, joining $\vec{a}=(a_1,\hdots,a_{i-1},0,a_{i+1},\hdots,a_n)$ and $\vec{b}=(a_1,\hdots,a_{i-1},1,a_{i+1},\hdots,a_n)$. Orient this edge so that it runs from $\vec{a}$ to $\vec{b}$.  Do the same with every edge of the $n$-cube.  Call these the {\it standard orientations} on the edges of the $n$-cube. Call a face of $I^n$ {\it positive} if it contains $\vec{0}$; otherwise it is {\it negative}, containing $\vec{1}=(1,\hdots,1)$.

\begin{definition}\label{Def:CubeComplex}
A {\bf \cube-complex} $K$ is a quotient space of a collection of disjoint cubes obtained by identifying certain faces of theirs via PL homeomorphisms.\footnote{Unlike the traditional notion of cubulation, we do not require that these identifications are between faces of {\it distinct} cubes.} If all of these face identifications glue a positive face of one cube to a negative face of another (not necessarily distinct) cube and respect the standard orientations on all edges, then $K$ is a {\bf directed} \cube-complex.
\end{definition}
Note that, by definition, a \cube-complex comes equipped with a cell structure.  
\begin{definition}\label{Def:Cubulation}
A {\bf generalized cubulation} of a manifold $M$ is a PL homeomorphism to a \cube-complex. A {\bf directed cubulation} of $M$ is a PL homeomorphism to a \cube-complex.
\end{definition}
In other words, a generalized cubulation of an $n$-manifold $M$ imposes a cell structure on $M$ in which every $n$-cell ``looks like" an $n$-cube, and in a directed cubulation, the $n$-cells are glued in a particularly nice way. 

\begin{example}\label{Ex:Tn}
The usual cell structure on $T^n$ determines a generalized cubulation, and in fact a directed cubulation, but not a cubulation in the traditional sense.
\end{example}

Let $f:M\to K$ be a directed cubulation of an $n$-manifold, $n=2k-1$, let $g:I^n=[0,k]^n\to T^n=(\R/k\Z)^n=[0,k]^n/\sim$ be the quotient map, and multisect $T^n=\bigcup_{i\in\Z_k}X_i$ as in Theorem \ref{T:Main}.  Multisect $M$ as follows. For each $n$-cell $C$ in $K$, let $h_C:I^n\to C$ be the identification from $K$.  For each $i\in\Z_k$,  
define 
\[X'_i=\bigcup_{n\text{-cubes }C\text{ in }K}f^{-1}(h_C(g^{-1}(X_i))).\]

\begin{prop}\label{P:Cube}
With the setup above, $M=\bigcup_{i\in\Z_k}X'_i$ determines a multisection of $M$.
\end{prop}

\begin{proof}
First consider the case where $p:M\to T^n$ is a covering space. Let $I\subset\Z_k$ be arbitrary.  Construct a handle structure on $X_I\subset T^n$, as in \textsection\ref{S:Handles}.  By construction, each handle is a subset of some open cube $(a,a+k)^n\subset T^n$.  Hence, the handle structure on $X_I\subset T^n$ pulls back to a handle structure on $X'_I\subset M$.  
The general case follows for the same reason, due to the fact that the multisection of $T^n$ is fixed by the permutation action on the indices. 
\end{proof}

\begin{rem}
In any multisection $M=\bigcup_{i\in \Z_k}X'_i$ from Proposition \ref{P:Cube}, all $X'_i$ have genus $n\#(n\text{-cubes in }K)$.  In particular, if $p:M\to T^n$ is an $r:1$ covering space, then $M$ has a multisection $M=\bigcup_{i\in \Z_k}X'_i$ in which each $X'_i$ has genus $nr$.
\end{rem}

\begin{example}\label{Ex:T3Twist}
Consider the quotient space $M$ obtained from $I^3$ by identifying the front and right faces, the left and top faces, and the bottom and back faces, all in the way that respects the standard orientations on the edges of $I^n$. See Figure \ref{Fi:T3Twist}, left.  The natural cell structure on $M$ consists of one vertex, three edges, three faces, and one 3-cell. It is easy to check that the link of the vertex is a 2-sphere, and so $M$ is a 3-manifold.  Geometrically, $M$ is geometrically flat, since there is a 27:1 covering space $T^3\to M$ (see Figure \ref{Fi:T3Twist}, right).  But $M$ is not $T^3$, since $H_1(M)\cong\Z\oplus\Z_3$. 
Proposition \ref{P:Cube} gives a genus 3 Heegaard splitting of $M$. Does $M$ have an efficient (genus 2) splitting? We leave this as a puzzle for the reader.
\end{example}

\begin{figure}
\begin{center}
\includegraphics[width=.7\textwidth]{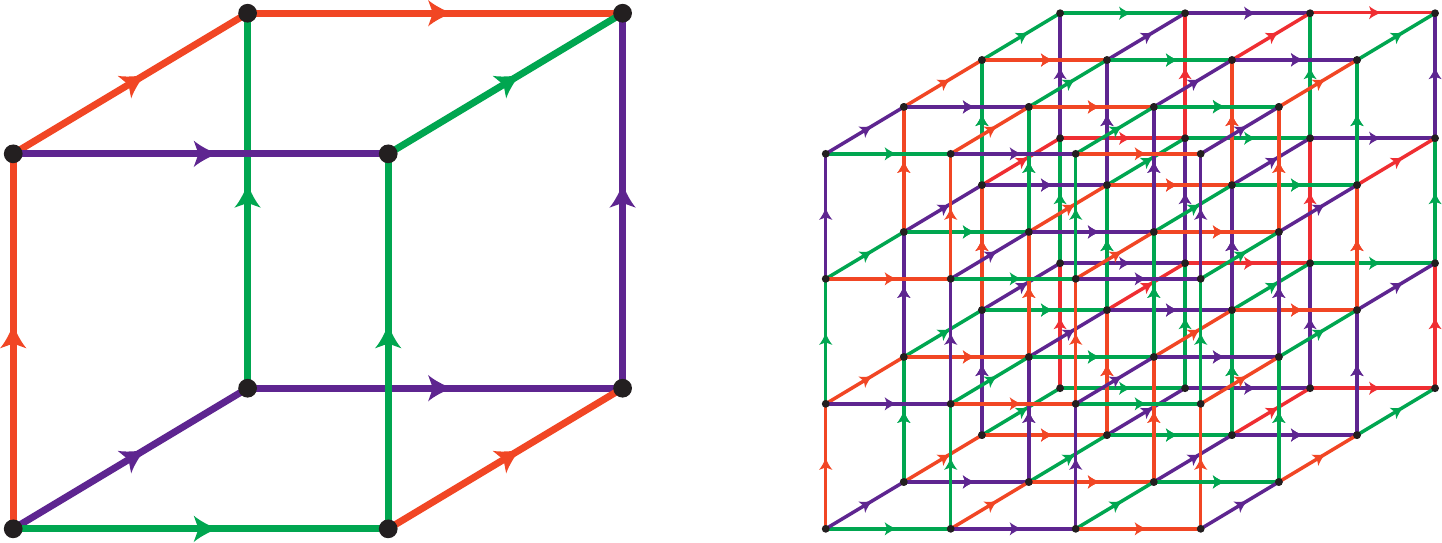}
\caption{Face identifications (left) for the 3-manifold $M$ from Example \ref{Ex:T3Twist}, and a 27:1 covering space (right) $T^3\to M$.}
\label{Fi:T3Twist}
\end{center}
\end{figure}

\begin{example}\label{Ex:TnTwist}
Generalizing Example \ref{Ex:T3Twist}, let $n=2k-1$, and let $\sigma\in S_n$ be an even permutation.  Denote the faces of $I^n$ by $F^\pm_i$, where $F^+_i=\{(x_1,\hdots,x_n):x_i=1\}$ and $F^-_i=\{(x_1,\hdots,x_n):x_i=0\}$. Identify each $F^+_{i}$ with $F^-_{\sigma(i)}$ by identifying each point
$(x_1,\hdots,1,\hdots,x_n)\in F^+_i$ (where the 1 is in the $i^{\text{th}}$ spot) with $(x_{\sigma^{-1}(1)},\hdots,0,\hdots,x_{\sigma^{-1}(n)})\in F^-_{\sigma(i)}$ (where the 0 is in the $\sigma^{-1}(i)^\text{th}$ spot).
\end{example}

\begin{question}
For what $n$ and $\sigma\in S_n$ does the construction in Example \ref{Ex:TnTwist} produce a manifold $M$? When it is a manifold, is $M$ always distinct from $T^n$? Is the multisection of $M$ from Proposition \ref{P:Cube} ever efficient? 
\end{question}

\section*{Appendix 1: Additional tables detailing handle decompositions}
Tables \ref{T:UV1} and \ref{T:UV2} explicitly detail $U_r,V_r\subset I_r$ for arbitrary $I_r$ (following Notation \ref{N:T}).  For simplicity, these tables have $I_r=I_0=\{0,\hdots,w\}$, listing $U_0,V_0$; this is not necessarily consistent with Convention \ref{Conv:ISimple}. To adapt $U_0,V_0\subset I_0$ to the general case $U_r,V_r\subset I_r$, add $\min I_r$ in each coordinate.

\begin{table}[H]
\begin{center}
\scalebox{1}{\begin{tabular}{||c|ccc|ccc||}
\hline
$I_0$& &$0\notin J$&&& $0\in J$&\\
\hline%
&$U_0$&&$V_0$&$U_0$&&$V_0$\\
\hline

\{0\}&%
$\varnothing$&&$\varnothing$
&
$\varnothing$&&\{0\}\\

\{0,1\}&%
$\varnothing$&&$\{1\}$
&%
$\varnothing$&&\{1\}\\

\{0,1,2\}&%
$\varnothing$&&\{1,2\}
&%
\{1\}&&\{2\}\\

\{0,1,2,3\}&%
$\varnothing$&&\{1,2,3\}
&%
\{1,2\}&&\{3\}\\

\{0,1,2,3,4\}&%
$\varnothing$&&\{1,2,3,4\}
&%
\{1,2,3\}&&\{4\}\\



$\{0,\hdots,w\}$&%
$\varnothing$&&$\{1,\hdots,w\}$
&%
$\{1,\hdots,w-1\}$&&$\{w\}$\\
\hline
\end{tabular}}
\caption{The index subsets $U_0,V_0\subset I_0$ when $i_*\notin I_0$.}\label{T:UV1}
\end{center}
\end{table}

\begin{table}[H]
\begin{center}
\scalebox{1}{\begin{tabular}{||c|cc||}
\hline
$I_0$& 
$U_0$&$V_0$\\
\hline

\{0\}&%
$\varnothing$&$\varnothing$\\

\{0,1\}&%
$\varnothing$&$\varnothing$\\

\{0,1,2\}&%
$\left.\begin{cases}
\{1\}&i_*=2\\
\varnothing&i_*\neq 2\\
\end{cases}\right\}$
&$\left.\begin{cases}
\{1,2\}&i_*=0\\
\varnothing&i_*\neq 0\\
\end{cases}\right\}$\\

\{0,1,2,3\}&%
$\left.\begin{cases}
\{2\}&i_*=0\\
\varnothing&i_*=1\\
\{1\}&i_*=2\\
\{1,2\}&i_*=3\\
\end{cases}\right\}$
&$\left.\begin{cases}
\{i_*+1,3\}&i_*\leq1\\
\varnothing&i_*\geq 2\\
\end{cases}\right\}$\\

$\{0,\hdots,w\}$&%
$I_0\setminus\{0,i_*,i_*+1,w\}$&
$\left.\begin{cases}
\{i_*+1,w\}&i_*\leq w-2\\
\varnothing&i_*\geq w-1\\
\end{cases}\right\}$\\
\hline
\end{tabular}}
\caption{The index subsets $U_0,V_0\subset I_0$ when $i_*\in I_0$.}\label{T:UV2}
\end{center}
\end{table}

Table \ref{T:91} details the handle decomposition of $X_I$ from $T^{9}$ with $I=\{0,1,3\}=I_1\sqcup I_2$, $I_1=\{0,1\}$, $I_2=\{3\}$.  The interesting feature of this example is how the two blocks of indices $I_1,I_2$ interact.

\begin{table}[H]
\begin{center}
\scalebox{1}{\begin{tabular}{||cc|ccc|c|ccc||}
\hline
$J$&$i_*$&$U$&$V$&$V^-$&$Y_z^*$&$h$&$z$&glue to\\ 
\hline

$\varnothing$&0&$\varnothing$&$\varnothing$&&${\lla\Navy{\alpha}1\Navy{\beta^3}\rra\lla3\Navy{\delta^3}\rra}$&$0$&1&\\ 

&1&$\varnothing$&$\varnothing$&$\varnothing$&${\lla0\red{\alpha}\rra\Navy{\beta^3}\lla3\Navy{\delta^3}\rra}$&$1$&2&1\\ 

&3&$\varnothing$&\{1\}&$\varnothing$&${0\lla\Navy{\alpha^+}1\Navy{\beta^3}\rra\Navy{\delta^3}}$&$0$&3&\\ 
&&&&\{1\}&${\lla0\red{\alpha^-}\rra\lla1\Navy{\beta^3}\rra\Navy{\delta^3}}$&$1$&4&3\\ 
\hline%

\{0\}&0&$\varnothing$&$\varnothing$&$\varnothing$&${\lla\Navy{\alpha}1\Navy{\beta^3}\rra\lla3\Navy{\delta^2}\rra\red{\ep}}$&$1$&5&1,3,4\\ 

&1&$\varnothing$&$\varnothing$&$\varnothing$&${\lla\red{\ep}0\red{\alpha}\rra\Navy{\beta^3}\lla3\Navy{\delta^2}\rra}$&$2$&6&2,5\\ 

&3&$\varnothing$&\{1\}&$\varnothing$&${\lla\red{\ep}0\rra\lla\Navy{\alpha^+}1\Navy{\beta^3}\rra\Navy{\delta^2}}$&$1$&7&3\\ 
&&&&\{1\}&${\lla\red{\ep}0\red{\alpha^-}\rra\lla1\Navy{\beta^3}\rra\Navy{\delta^2}}$&$2$&8&4,7\\ 

\hline%

\{3\}&0&$\varnothing$&\{3\}&$\varnothing$&${\lla\Navy{\alpha}1\Navy{\beta^2}\rra\lla\Navy{\gamma^+}3\Navy{\delta^3}\rra}$&0&9&\\
&&&&\{3\}&${\lla\Navy{\alpha}1\Navy{\beta^2}\rra\red{\gamma^-}\lla3\Navy{\delta^3}\rra}$&$1$&10&1,9\\ 

&1&$\varnothing$&\{3\}&$\varnothing$&${\lla0\red{\alpha}\rra\Navy{\beta^2}\lla\Navy{\gamma^+}3\Navy{\delta^3}\rra}$&1&11&9\\ 
&&&&\{3\}&${\lla0\red{\alpha}\rra\Navy{\beta^2}\red{\gamma^-}\lla3\Navy{\delta^3}\rra}$&2&12&2,10,11\\ 

&3&$\varnothing$&\{1\}&$\varnothing$&${0\lla\Navy{\alpha^+}1\Navy{\beta^2}\rra\red{\gamma}\Navy{\delta^3}}$&$1$&13&2,3\\ 
&&&&\{1\}&${\lla0\red{\alpha^-}\rra\lla1\Navy{\beta^2}\rra\red{\gamma}\Navy{\delta^3}}$&$2$&14&2,4,13\\ 

\hline%

\{0,3\}&0&$\varnothing$&\{3\}&$\varnothing$&${\lla\Navy{\alpha}1\Navy{\beta^2}\rra\lla\Navy{\gamma^+}3\Navy{\delta^2}\rra\red{\ep}}$&$1$&15&9,13,14\\ 
&&&&$\{3\}$&${\lla\Navy{\alpha}1\Navy{\beta^2}\rra\red{\gamma^-}\lla3\Navy{\delta^2}\rra\red{\ep}}$&$2$&16&10,13,14,15\\ 

&1&$\varnothing$&\{3\}&$\varnothing$&${\lla\red{\ep}0\red{\alpha}\rra\Navy{\beta^2}\lla\Navy{\gamma^+}3\Navy{\delta^2}\rra}$&$2$&17&11,15\\ 
&&&&\{3\}&${\lla\red{\ep}0\red{\alpha}\rra\Navy{\beta^2}\red{\gamma^-}\lla3\Navy{\delta^2}\rra}$&${\bf 3}$&18&6,12,16,17\\ 

&3&$\varnothing$&\{1\}&$\varnothing$&${\lla\red{\ep}0\rra\lla\Navy{\alpha^+}1\Navy{\beta^2}\rra\red{\gamma}\Navy{\delta^2}}$&$2$&19&6,7,13\\ 
&&&&\{1\}&${\lla\red{\ep}0\red{\alpha^-}\rra\lla1\Navy{\beta^2}\rra\red{\gamma}\Navy{\delta^2}}$&${\bf 3}$&20&6,8,14,19\\ 
 
\hline
\end{tabular}}
\caption{A genus 9 quintisection of $T^9$: $X_I$ when $I=\{0,1,3\}$}
\label{T:91}
\end{center}
\end{table}

Tables \ref{T:92}-\ref{T:93} detail the handle decomposition of $X_I$, $I=\{0,1,2,3\}$, from the quintisection of $T^9$.  Note that, since $I=I_1$ consists of a single block in this example, we always have $I_1=I_*$.

\begin{table}[H]
\begin{center}
{\begin{tabular}{||c|ccc|c|ccc||}
\hline
$i_*$&$U$&$V$&$V^-$&$Y_z^*$&$h$&$z$&glue to\\ 
\hline

0&\{2\}&\{1,3\}&$\varnothing$&$\Navy{\alpha^-}1\Navy{\beta^\circ_3}2\lla\Navy{\gamma^+}3\Navy{\delta^3}\rra$&0&1&\\ 
&&&&$\Navy{\alpha^-}\lla1\red{\beta^-_3}\rra2\lla\Navy{\gamma^+}3\Navy{\delta^3}\rra$&1&2&1\\ 
&&&&$\Navy{\alpha^-}1\lla\red{\beta^+_3}2\rra\lla\Navy{\gamma^+}3\Navy{\delta^3}\rra$&1&3&1\\ 
&&&\{1\}&$\lla\red{\alpha^+}1\rra\Navy{\beta^\circ_3}2\lla\Navy{\gamma^+}3\Navy{\delta^3}\rra$&1&4&1\\ 
&&&&$\lla\red{\alpha^+}1\red{\beta^-_3}\rra2\lla\Navy{\gamma^+}3\Navy{\delta^3}\rra$&2&5&2,4\\ 
&&&&$\lla\red{\alpha^+}1\rra\lla\red{\beta^+_3}2\rra\lla\Navy{\gamma^+}3\Navy{\delta^3}\rra$&2&6&3,4\\ 
&&&\{3\}&$\Navy{\alpha^-}1\Navy{\beta^\circ_3}\lla2\red{\gamma^-}\rra\lla3\Navy{\delta^3}\rra$&1&7&1\\ 
&&&&$\Navy{\alpha^-}\lla1\red{\beta^-_3}\rra\lla2\red{\gamma^-}\rra\lla3\Navy{\delta^3}\rra$&2&8&2,7\\ 
&&&&$\Navy{\alpha^-}1\lla\red{\beta^+_3}2\red{\gamma^-}\rra\lla3\Navy{\delta^3}\rra$&2&9&3,7\\ 
&&&\{1,3\}&$\lla\red{\alpha^+}1\rra\Navy{\beta^\circ_3}\lla2\red{\gamma^-}\rra\lla3\Navy{\delta^3}\rra$&2&10&4,7\\ 
&&&&$\lla\red{\alpha^+}1\red{\beta^-_3}\rra\lla2\red{\gamma^-}\rra\lla3\Navy{\delta^3}\rra$&3&11&5,8,10\\ 
&&&&$\lla\red{\alpha^+}1\rra\lla\red{\beta^+_3}2\red{\gamma^-}\rra\lla3\Navy{\delta^3}\rra$&3&12&6,9,10\\

\hline

1&$\varnothing$&\{2,3\}&$\varnothing$&$\lla0\red{\alpha}\rra\Navy{\beta^-}2\lla\Navy{\gamma^+}3\Navy{\delta^3}\rra$&1&13&1,2\\ 
&&&\{2\}&$\lla0\red{\alpha}\rra\lla\red{\beta^+}2\rra\lla\Navy{\gamma^+}3\Navy{\delta^3}\rra$&2&14&1,3,13\\ 
&&&\{3\}&$\lla0\red{\alpha}\rra\Navy{\beta^-}\lla2\red{\gamma^-}\rra\lla3\Navy{\delta^3}\rra$&2&15&7,8,13\\ 
&&&\{2,3\}&$\lla0\red{\alpha}\rra\lla\red{\beta^+}2\red{\gamma^-}\rra\lla3\Navy{\delta^3}\rra$&3&16&7,9,14,15\\ 
\hline%

2&\{1\}&$\varnothing$&$\varnothing$&$0\Navy{\alpha^\circ_3}\lla1\red{\beta}\rra\lla\Navy{\gamma}3\Navy{\delta^3}\rra$&1&17&13\\ 
&&&&$\lla0\red{\alpha^-_3}\rra\lla1\red{\beta}\rra\lla\Navy{\gamma}3\Navy{\delta^3}\rra$&2&18&13,17\\ 
&&&&$0\lla\red{\alpha^+_3}1\red{\beta}\rra\lla\Navy{\gamma}3\Navy{\delta^3}\rra$&2&19&13,17\\ 
\hline%

3&\{1,2\}&$\varnothing$&$\varnothing$&$0\Navy{\alpha^\circ_3}1\Navy{\beta^\circ_3}\lla2\red{\gamma}\rra\Navy{\delta^3}$&1&20&17\\ 
&&&&$\lla0\red{\alpha^-_3}\rra1\Navy{\beta^\circ_3}\lla2\red{\gamma}\rra\Navy{\delta^3}$&2&21&18,20\\ 
&&&&$0\lla\red{\alpha^+_3}1\rra\Navy{\beta^\circ_3}\lla2\red{\gamma}\rra\Navy{\delta^3}$&2&22&19,20\\ 
&&&&$0\Navy{\alpha^\circ_3}\lla1\red{\beta^-_3}\rra\lla2\red{\gamma}\rra\Navy{\delta^3}$&2&23&17,20\\ 
&&&&$\lla0\red{\alpha^-_3}\rra\lla1\red{\beta^-_3}\rra\lla2\red{\gamma}\rra\Navy{\delta^3}$&3&24&18,20,23\\ 
&&&&$0\lla\red{\alpha^+_3}1\red{\beta^-_3}\rra\lla2\red{\gamma}\rra\Navy{\delta^3}$&3&25&19,22,23\\ 
&&&&$0\Navy{\alpha^\circ_3}1\lla\red{\beta^+_3}2\red{\gamma}\rra\Navy{\delta^3}$&2&26&17,20\\ 
&&&&$\lla0\red{\alpha^-_3}\rra1\lla\red{\beta^+_3}2\red{\gamma}\rra\Navy{\delta^3}$&3&27&18,21,26\\ 
&&&&$0\lla\red{\alpha^+_3}1\rra\lla\red{\beta^+_3}2\red{\gamma}\rra\Navy{\delta^3}$&3&28&19,22,26\\ 

\hline
\hline
\end{tabular}}
\caption{$X_I$, $I=\{0,1,2,3\}$, from $T^{9}$. Part 1: $J=\varnothing$.}\label{T:92}
\end{center}
\end{table}

\begin{table}[H]
\begin{center}
{\begin{tabular}{||c|ccc|c|ccc||}
\hline
$i_*$&$U$&$V$&$V^-$&$Y_z^*$&$h$&$z$&glue to\\ 
\hline

0&\{2\}&\{1,3\}&$\varnothing$&$\Navy{\alpha^-}1\Navy{\beta^\circ_3}2\lla\Navy{\gamma^+}3\Navy{\delta^2}\rra\red{\ep}$&1&29&1,19,20\\ 
&&&&$\Navy{\alpha^-}\lla1\red{\beta^-_3}\rra2\lla\Navy{\gamma^+}3\Navy{\delta^2}\rra\red{\ep}$&2&30&2,22,23,29\\ 
&&&&$\Navy{\alpha^-}1\red{\beta^+_3}2\lla\Navy{\gamma^+}3\Navy{\delta^2}\rra\red{\ep}$&2&31&3,25,26,29\\ 
&&&\{1\}&$\lla\red{\alpha^+}1\rra\Navy{\beta^\circ_3}2\lla\Navy{\gamma^+}3\Navy{\delta^2}\rra\red{\ep}$&2&32&4,19,21\\ 
&&&&$\lla\red{\alpha^+}1\red{\beta^-_3}\rra2\lla\Navy{\gamma^+}3\Navy{\delta^2}\rra\red{\ep}$&3&33&5,22,24,30,32\\ 
&&&&$\lla\red{\alpha^+}1\rra\red{\beta^+_3}2\lla\Navy{\gamma^+}3\Navy{\delta^2}\rra\red{\ep}$&3&34&6,25,27,31,32\\ 
&&&\{3\}&$\Navy{\alpha^-}1\Navy{\beta^\circ_3}\lla2\red{\gamma^-}\rra\lla3\Navy{\delta^2}\rra\red{\ep}$&2&35&7,19,20,29\\ 
&&&&$\Navy{\alpha^-}\lla1\red{\beta^-_3}\rra\lla2\red{\gamma^-}\rra\lla3\Navy{\delta^2}\rra\red{\ep}$&3&36&8,22,23,30,35\\ 
&&&&$\Navy{\alpha^-}1\lla\red{\beta^+_3}2\red{\gamma^-}\rra\lla3\Navy{\delta^2}\rra\red{\ep}$&3&37&9,25,26,31,35\\ 
&&&\{1,3\}&$\lla\red{\alpha^+}1\rra\Navy{\beta^\circ_3}\lla2\red{\gamma^-}\rra\lla3\Navy{\delta^2}\rra\red{\ep}$&3&38&10,19,21,32,35\\
&&&&$\lla\red{\alpha^+}1\red{\beta^-_3}\rra\lla2\red{\gamma^-}\rra\lla3\Navy{\delta^2}\rra\red{\ep}$&{\bf 4}&39&11,22,24,33,36,38\\
&&&&$\lla\red{\alpha^+}1\rra\lla\red{\beta^+_3}2\red{\gamma^-}\rra\lla3\Navy{\delta^2}\rra\red{\ep}$&{\bf 4}&40&12,25,27,34,37,38\\ 
\hline

1&$\varnothing$&\{2,3\}&$\varnothing$&$\lla\red{\ep}0\red{\alpha}\rra\Navy{\beta^-}2\lla\Navy{\gamma^+}3\Navy{\delta^2}\rra$&2&41&13,29,30\\ 
&&&\{2\}&$\lla\red{\ep}0\red{\alpha}\rra\lla\red{\beta^+}2\rra\lla\Navy{\gamma^+}3\Navy{\delta^2}\rra$&3&42&14,29,31,41\\ 
&&&\{3\}&$\lla\red{\ep}0\red{\alpha}\rra\Navy{\beta^-}\lla2\red{\gamma^-}\rra\lla3\Navy{\delta^2}\rra$&3&43&15,35,36,41\\ 
&&&\{2,3\}&$\lla\red{\ep}0\red{\alpha}\rra\lla\red{\beta^+}2\red{\gamma^-}\rra\lla3\Navy{\delta^2}\rra$&{\bf 4}&44&16,35,37,42,43\\ 
\hline%

2&\{1\}&$\varnothing$&$\varnothing$&$\lla\red{\ep}0\rra\Navy{\alpha^\circ_3}\lla1\red{\beta}\rra\lla\Navy{\gamma}3\Navy{\delta^2}\rra$&2&45&17,41,43\\ 
&&&&$\lla\red{\ep}0\red{\alpha^-_3}\rra\lla1\red{\beta}\rra\lla\Navy{\gamma}3\Navy{\delta^2}\rra$&3&46&18,41,43,45\\ 
&&&&$\lla\red{\ep}0\rra\lla\red{\alpha^+_3}1\red{\beta}\rra\lla\Navy{\gamma}3\Navy{\delta^2}\rra$&3&47&19,41,43,45\\ 
\hline%

3&\{1,2\}&$\varnothing$&$\varnothing$&$\lla\red{\ep}0\rra\Navy{\alpha^\circ_3}1\Navy{\beta^\circ_3}\lla2\red{\gamma}\rra\Navy{\delta^2}$&2&48&20,45\\ 
&&&&$\lla\red{\ep}0\red{\alpha^-_3}\rra1\Navy{\beta^\circ_3}\lla2\red{\gamma}\rra\Navy{\delta^2}$&3&49&21,46,48\\ 
&&&&$\lla\red{\ep}0\rra\lla\red{\alpha^+_3}1\rra\Navy{\beta^\circ_3}\lla2\red{\gamma}\rra\Navy{\delta^2}$&3&50&22,47,48\\ 
&&&&$\lla\red{\ep}0\rra\Navy{\alpha^\circ_3}\lla1\red{\beta^-_3}\rra\lla2\red{\gamma}\rra\Navy{\delta^2}$&3&51&23,45,48\\ 
&&&&$\lla\red{\ep}0\red{\alpha^-_3}\rra\lla1\red{\beta^-_3}\rra\lla2\red{\gamma}\rra\Navy{\delta^2}$&{\bf 4}&52&24,46,49,51\\
&&&&$\lla\red{\ep}0\rra\lla\red{\alpha^+_3}1\red{\beta^-_3}\rra\lla2\red{\gamma}\rra\Navy{\delta^2}$&{\bf 4}&53&25,47,50,51\\ 
&&&&$\lla\red{\ep}0\rra\Navy{\alpha^\circ_3}1\lla\red{\beta^+_3}2\red{\gamma}\rra\Navy{\delta^2}$&3&54&26,45,48\\ 
&&&&$\lla\red{\ep}0\red{\alpha^-_3}\rra1\lla\red{\beta^+_3}2\red{\gamma}\rra\Navy{\delta^2}$&{\bf 4}&55&27,46,49,54\\ 
&&&&$\lla\red{\ep}0\rra\lla\red{\alpha^+_3}1\rra\lla\red{\beta^+_3}2\red{\gamma}\rra\Navy{\delta^2}$&{\bf 4}&56&28,47,50,54\\ 

\hline
\hline
\end{tabular}}
\caption{$X_I$, $I=\{0,1,2,3\}$, from $T^{9}$. Part 2: $J=\{0\}$.}\label{T:93}
\end{center}
\end{table}


Tables \ref{T:111} and \ref{T:112} detail handle decompositions of $X_I$, $I=\{0,1,2,4\}$ from the sexasection of $T^{11}$. The parts of these tables with $i_*=4$ and $0\notin J$ feature a complication that does not appear in dimensions $n\leq 9$. Also see Tables \ref{T:13} and \ref{T:15} for more complicated examples of this pattern.

\begin{table}[H]
\begin{center}
\scalebox{.925}{\begin{tabular}{||cc|ccc|c|ccc||}
\hline
$J$&$i_*$&$U$&$V$&$V^-$&$Y_z^*$&$h$&$z$&glue to\\ 
\hline
$\varnothing$&4&$\varnothing$&\{1,2\}&$\varnothing$&$0\lla\Navy{\alpha^+}1\rra\lla\Navy{\beta^+}2\Navy{\gamma^3}\rra\Navy{\ep^3}$&0&1&\\
&&&&$\{1\}$&$\lla0\red{\alpha^-}\rra1\lla\Navy{\beta^+}2\Navy{\gamma^3}\rra\Navy{\ep^3}$&1&2&1\\
&&&&$\{1,2\}$&$\lla0\Navy{\alpha^-}\rra\lla1\red{\beta^-}\rra\lla2\Navy{\gamma^3}\rra\Navy{\ep^3}$&1&3&2\\
&&&&$\{2\}$&$0\lla\red{\alpha^+}1\red{\beta^-}\rra\lla2\Navy{\gamma^3}\rra\Navy{\ep^3}$&2&4&1,3\\
&0&$\varnothing$&\{1,2\}&$\varnothing$&$\Navy{\alpha^-}1\lla\Navy{\beta^+}2\Navy{\gamma^3}\rra\lla4\Navy{\ep^3}\rra$&0&5&\\
&&&&\{1\}&$\lla\red{\alpha^+}1\rra\lla\Navy{\beta^+}2\Navy{\gamma^3}\rra\lla4\Navy{\ep^3}\rra$&1&6&5\\
&&&&\{2\}&$\Navy{\alpha^-}\lla1\red{\beta^-}\rra\lla2\Navy{\gamma^3}\rra\lla4\Navy{\ep^3}\rra$&1&7&5\\
&&&&\{1,2\}&$\lla\red{\alpha^+}1\red{\beta^-}\rra\lla2\Navy{\gamma^3}\rra\lla4\Navy{\ep^3}\rra$&2&8&5,6\\
&1&$\varnothing$&$\varnothing$&$\varnothing$&$\lla0\red{\alpha}\rra\lla\Navy{\beta}2\Navy{\gamma^3}\rra\lla4\Navy{\ep^3}\rra$&1&9&5,7\\
&2&$\{1\}$&$\varnothing$&$\varnothing$&$0\Navy{\alpha^\circ_3}\lla1\red{\beta}\rra\Navy{\gamma^3}\lla4\Navy{\ep^3}\rra$&1&10&9\\
&&&&&$\lla0\red{\alpha^-_3}\rra\lla1\red{\beta}\rra\Navy{\gamma^3}\lla4\Navy{\ep^3}\rra$&2&11&9,10\\
&&&&&$0\lla\red{\alpha^+_3}1\red{\beta}\rra\Navy{\gamma^3}\lla4\Navy{\ep^3}\rra$&2&12&9,10\\
\hline
$\{4\}$&4&$\varnothing$&\{1,2\}&$\varnothing$&$0\lla\Navy{\alpha^+}1\rra\lla\Navy{\beta^+}2\Navy{\gamma^2}\rra\red{\delta}\Navy{\ep^3}$&1&13&1,10,12\\
&&&&$\{1\}$&$\lla0\red{\alpha^-}\rra1\lla\Navy{\beta^+}2\Navy{\gamma^2}\rra\red{\delta}\Navy{\ep^3}$&2&14&2,10,11,13\\
&&&&$\{1,2\}$&$\lla0\Navy{\alpha^-}\rra\lla1\red{\beta^-}\rra\lla2\Navy{\gamma^2}\rra\red{\delta}\Navy{\ep^3}$&2&15&3,10,11,14\\
&&&&$\{2\}$&$0\lla\red{\alpha^+}1\red{\beta^-}\rra\lla2\Navy{\gamma^2}\rra\red{\delta}\Navy{\ep^3}$&3&16&4,10,12,13,15\\
&0&$\varnothing$&\{1,2\}&$\varnothing$&$\Navy{\alpha^-}1\lla\Navy{\beta^+}2\Navy{\gamma^2}\rra\lla\Navy{\delta^+}4\Navy{\ep^3}\rra$&0&17&\\
&&&&&$\Navy{\alpha^-}1\lla\Navy{\beta^+}2\Navy{\gamma^2}\rra\red{\delta^-}\lla4\Navy{\ep^3}\rra$&1&18&5,17\\
&&&&\{1\}&$\lla\red{\alpha^+}1\rra\lla\Navy{\beta^+}2\Navy{\gamma^2}\rra\lla\Navy{\delta^+}4\Navy{\ep^3}\rra$&1&19&17\\
&&&&&$\lla\red{\alpha^+}1\rra\lla\Navy{\beta^+}2\Navy{\gamma^2}\rra\red{\delta^-}\lla4\Navy{\ep^3}\rra$&2&20&6,18,19\\
&&&&\{2\}&$\Navy{\alpha^-}\lla1\red{\beta^-}\rra\lla2\Navy{\gamma^2}\rra\lla\Navy{\delta^+}4\Navy{\ep^3}\rra$&1&21&19\\
&&&&&$\Navy{\alpha^-}\lla1\red{\beta^-}\rra\lla2\Navy{\gamma^2}\rra\red{\delta^-}\lla4\Navy{\ep^3}\rra$&2&22&7,20,21\\
&&&&\{1,2\}&$\lla\red{\alpha^+}1\red{\beta^-}\rra\lla2\Navy{\gamma^2}\rra\lla\Navy{\delta^+}4\Navy{\ep^3}\rra$&2&23&19,21\\
&&&&&$\lla\red{\alpha^+}1\red{\beta^-}\rra\lla2\Navy{\gamma^2}\rra\red{\delta^-}\lla4\Navy{\ep^3}\rra$&3&24&8,20,22,23\\
&1&$\varnothing$&$\{4\}$&$\varnothing$&$\lla0\red{\alpha}\rra\lla\Navy{\beta}2\Navy{\gamma^2}\rra\lla\Navy{\delta^+}4\Navy{\ep^3}\rra$&1&25&17,21\\
&&&&$\{4\}$&$\lla0\red{\alpha}\rra\lla\Navy{\beta}2\Navy{\gamma^2}\rra\red{\delta^-}\lla4\Navy{\ep^3}\rra$&2&26&9,18,22,25\\
&2&$\{1\}$&$\{4\}$&$\varnothing$&$0\Navy{\alpha^\circ_3}\lla1\red{\beta}\rra\Navy{\gamma^2}\lla\Navy{\delta^+}4\Navy{\ep^3}\rra$&1&27&25\\
&&&&&$\lla0\red{\alpha^-_3}\rra\lla1\red{\beta}\rra\Navy{\gamma^2}\lla\Navy{\delta^+}4\Navy{\ep^3}\rra$&2&28&25,27\\
&&&&&$0\lla\red{\alpha^+_3}1\red{\beta}\rra\Navy{\gamma^2}\lla\Navy{\delta^+}4\Navy{\ep^3}\rra$&2&29&25,27\\
&&&&$\{4\}$&$0\Navy{\alpha^\circ_3}\lla1\red{\beta}\rra\Navy{\gamma^2}\red{\delta^-}\lla4\Navy{\ep^3}\rra$&2&30&10,26,27\\
&&&&&$\lla0\red{\alpha^-_3}\rra\lla1\red{\beta}\rra\Navy{\gamma^2}\red{\delta^-}\lla4\Navy{\ep^3}\rra$&3&31&11,26,28,30\\
&&&&&$0\lla\red{\alpha^+_3}1\red{\beta}\rra\Navy{\gamma^2}\red{\delta^-}\lla4\Navy{\ep^3}\rra$&3&32&12,26,29,30\\
\hline

\end{tabular}}
\caption{Part 1 of $X_I$, $I=\{0,1,2,4\}$, from $T^{11}$.}\label{T:111}
\end{center}
\end{table}

\begin{table}[H]
\begin{center}
\scalebox{.925}{\begin{tabular}{||cc|ccc|c|ccc||}
\hline
$J$&$i_*$&$U$&$V$&$V^-$&$Y_z^*$&$h$&$z$&glue to\\ 
\hline
$\{0\}$&4&$\{1\}$&\{2\}&$\varnothing$&$\lla\red{\zeta}0\rra\Navy{\alpha^\circ_3}1\lla\Navy{\beta^+}2\Navy{\gamma^3}\rra\Navy{\ep^2}$&1&33&1,2\\
&&&&&$\lla\red{\zeta}0\rra\lla\red{\alpha^+_3}1\rra\lla\Navy{\beta^+}2\Navy{\gamma^3}\rra\Navy{\ep^2}$&2&34&1,33\\
&&&&&$\lla\red{\zeta}0\red{\alpha^-_3}\rra1\lla\Navy{\beta^+}2\Navy{\gamma^3}\rra\Navy{\ep^2}$&2&35&2,33\\
&&&&$\{2\}$&$\lla\red{\zeta}0\rra\Navy{\alpha^\circ_3}\lla1\red{\beta^-}\rra\lla2\Navy{\gamma^3}\rra\Navy{\ep^2}$&2&36&3,4,33\\
&&&&&$\lla\red{\zeta}0\rra\lla\red{\alpha^+_3}1\red{\beta^-}\rra\lla2\Navy{\gamma^3}\rra\Navy{\ep^2}$&3&37&3,34,36\\
&&&&&$\lla\red{\zeta}0\red{\alpha^-_3}\rra\lla1\red{\beta^-}\rra\lla2\Navy{\gamma^3}\rra\Navy{\ep^2}$&3&38&4,35,36\\
&0&$\varnothing$&\{1,2\}&$\varnothing$&$\red{\zeta}\Navy{\alpha^-}1\lla\Navy{\beta^+}2\Navy{\gamma^3}\rra\lla4\Navy{\ep^2}\rra$&1&39&2,5\\
&&&&\{1\}&$\red{\zeta}\lla\red{\alpha^+}1\rra\lla\Navy{\beta^+}2\Navy{\gamma^3}\rra\lla4\Navy{\ep^2}\rra$&2&40&1,6,39\\
&&&&\{2\}&$\red{\zeta}\Navy{\alpha^-}\lla1\red{\beta^-}\rra\lla2\Navy{\gamma^3}\rra\lla4\Navy{\ep^2}\rra$&2&41&3,7,39\\
&&&&\{1,2\}&$\red{\zeta}\lla\red{\alpha^+}1\red{\beta^-}\rra\lla2\Navy{\gamma^3}\rra\lla4\Navy{\ep^2}\rra$&3&42&4,8,40\\
&1&$\varnothing$&$\varnothing$&$\varnothing$&$\lla\red{\zeta}0\red{\alpha}\rra\lla\Navy{\beta}2\Navy{\gamma^3}\rra\lla4\Navy{\ep^2}\rra$&2&43&9,39,41\\
&2&$\{1\}$&$\varnothing$&$\varnothing$&$\lla\red{\zeta}0\rra\Navy{\alpha^\circ_3}\lla1\red{\beta}\rra\Navy{\gamma^3}\lla4\Navy{\ep^2}\rra$&2&44&10,43\\
&&&&&$\lla\red{\zeta}0\red{\alpha^-_3}\rra\lla1\red{\beta}\rra\Navy{\gamma^3}\lla4\Navy{\ep^2}\rra$&3&45&11,43,44\\
&&&&&$\lla\red{\zeta}0\rra\lla\red{\alpha^+_3}1\red{\beta}\rra\Navy{\gamma^3}\lla4\Navy{\ep^2}\rra$&3&46&12,43,44\\
\hline

$\{0,4\}$&4&$\{1\}$&\{2\}&$\varnothing$&$\lla\red{\zeta}0\rra\Navy{\alpha^\circ_3}1\lla\Navy{\beta^+}2\Navy{\gamma^2}\rra\red{\delta}\Navy{\ep^2}$&2&47&13,14,33,44\\
&&&&&$\lla\red{\zeta}0\rra\lla\red{\alpha^+_3}1\rra\lla\Navy{\beta^+}2\Navy{\gamma^2}\rra\red{\delta}\Navy{\ep^2}$&3&48&13,34,45,47\\
&&&&&$\lla\red{\zeta}0\red{\alpha^-_3}\rra1\lla\Navy{\beta^+}2\Navy{\gamma^2}\rra\red{\delta}\Navy{\ep^2}$&3&49&14,35,46,47\\
&&&&$\{2\}$&$\lla\red{\zeta}0\rra\Navy{\alpha^\circ_3}\lla1\red{\beta^-}\rra\lla2\Navy{\gamma^2}\rra\red{\delta}\Navy{\ep^2}$&3&50&15,16,36,44,47\\
&&&&&$\lla\red{\zeta}0\rra\lla\red{\alpha^+_3}1\red{\beta^-}\rra\lla2\Navy{\gamma^2}\rra\red{\delta}\Navy{\ep^2}$&4&51&16,37,45,48,50\\
&&&&&$\lla\red{\zeta}0\red{\alpha^-_3}\rra\lla1\red{\beta^-}\rra\lla2\Navy{\gamma^2}\rra\red{\delta}\Navy{\ep^2}$&4&52&15,38,46,49,50\\
&0&$\varnothing$&\{1,2\}&$\varnothing$&$\red{\zeta}\Navy{\alpha^-}1\lla\Navy{\beta^+}2\Navy{\gamma^2}\rra\lla\Navy{\delta^+}4\Navy{\ep^2}\rra$&1&53&14,17\\
&&&&&$\red{\zeta}\Navy{\alpha^-}1\lla\Navy{\beta^+}2\Navy{\gamma^2}\rra\red{\delta^-}\lla4\Navy{\ep^2}\rra$&2&54&14,18,39,53\\
&&&&\{1\}&$\red{\zeta}\lla\red{\alpha^+}1\rra\lla\Navy{\beta^+}2\Navy{\gamma^2}\rra\lla\Navy{\delta^+}4\Navy{\ep^2}\rra$&2&55&13,19,53\\
&&&&&$\red{\zeta}\lla\red{\alpha^+}1\rra\lla\Navy{\beta^+}2\Navy{\gamma^2}\rra\red{\delta^-}\lla4\Navy{\ep^2}\rra$&3&56&13,20,40,54,55\\
&&&&\{2\}&$\red{\zeta}\Navy{\alpha^-}\lla1\red{\beta^-}\rra\lla2\Navy{\gamma^2}\rra\lla\Navy{\delta^+}4\Navy{\ep^2}\rra$&2&57&15,21,53\\
&&&&&$\red{\zeta}\Navy{\alpha^-}\lla1\red{\beta^-}\rra\lla2\Navy{\gamma^2}\rra\red{\delta^-}\lla4\Navy{\ep^2}\rra$&3&58&15,22,41,54,57\\
&&&&\{1,2\}&$\red{\zeta}\lla\red{\alpha^+}1\red{\beta^-}\rra\lla2\Navy{\gamma^2}\rra\lla\Navy{\delta^+}4\Navy{\ep^2}\rra$&3&59&16,23,55,57\\
&&&&&$\red{\zeta}\lla\red{\alpha^+}1\red{\beta^-}\rra\lla2\Navy{\gamma^2}\rra\red{\delta^-}\lla4\Navy{\ep^2}\rra$&4&60&16,24,42,56,58,59\\
&1&$\varnothing$&$\{4\}$&$\varnothing$&$\lla\red{\zeta}0\red{\alpha}\rra\lla\Navy{\beta}2\Navy{\gamma^2}\rra\lla\Navy{\delta^+}4\Navy{\ep^2}\rra$&2&61&25,53,57\\
&&&&$\{4\}$&$\lla\red{\zeta}0\red{\alpha}\rra\lla\Navy{\beta}2\Navy{\gamma^2}\rra\red{\delta^-}\lla4\Navy{\ep^2}\rra$&3&62&26,43,54,58,61\\
&2&$\{1\}$&$\{4\}$&$\varnothing$&$\lla\red{\zeta}0\rra\Navy{\alpha^\circ_3}\lla1\red{\beta}\rra\Navy{\gamma^2}\lla\Navy{\delta^+}4\Navy{\ep^2}\rra$&2&63&27,61\\
&&&&&$\lla\red{\zeta}0\red{\alpha^-_3}\rra\lla1\red{\beta}\rra\Navy{\gamma^2}\lla\Navy{\delta^+}4\Navy{\ep^2}\rra$&3&64&{28},61,63\\
&&&&&$\lla\red{\zeta}0\rra\lla\red{\alpha^+_3}1\red{\beta}\rra\Navy{\gamma^2}\lla\Navy{\delta^+}4\Navy{\ep^2}\rra$&3&65&{29},61,63\\
&&&&$\{4\}$&$\lla\red{\zeta}0\rra\Navy{\alpha^\circ_3}\lla1\red{\beta}\rra\Navy{\gamma^2}\red{\delta^-}\lla4\Navy{\ep^2}\rra$&3&66&{30},44,62,63\\
&&&&&$\lla\red{\zeta}0\red{\alpha^-_3}\rra\lla1\red{\beta}\rra\Navy{\gamma^2}\red{\delta^-}\lla4\Navy{\ep^2}\rra$&4&67&{31},44,62,64,66\\
&&&&&$\lla\red{\zeta}0\rra\lla\red{\alpha^+_3}1\red{\beta}\rra\Navy{\gamma^2}\red{\delta^-}\lla4\Navy{\ep^2}\rra$&4&68&32,45,62,65,66\\

\hline
\end{tabular}}
\caption{Part 2 of $X_I$, $I=\{0,1,2,4\}$, from $T^{11}$.}\label{T:112}
\end{center}
\end{table}

Table \ref{T:15} details the start of the handle decomposition of $X_I$ from $T^{15}$ with $I=\{0,1,2,3,4,6\}$, focusing on the first few pieces $Y_z$. Those pieces have $J=\varnothing$, $i_*=6$, $U=\varnothing$, $V=\{1,2,3,4\}$. The interesting feature of this example is the ordering of these pieces.  Compare to (\ref{E:order}) and Tables \ref{T:13}, \ref{T:111}, \ref{T:112}.
\begin{table}[H]
\begin{center}
\scalebox{1}{\begin{tabular}{||c|c|ccc||}
\hline
$V^-$&$Y_z^*$&$h$&$z$&glue to\\ 
\hline
$\varnothing$&$0\lla\Navy{\alpha^+}1\rra\lla\Navy{\beta^+}2\rra\lla\Navy{\gamma^+}3\rra\lla\Navy{\delta^+}4\Navy{\ep^3}\rra\Navy{\eta^3}$&0&1&\\
$\{1\}$&$\lla0\red{\alpha^-}\rra1\lla\Navy{\beta^+}2\rra\lla\Navy{\gamma^+}3\rra\lla\Navy{\delta^+}4\Navy{\ep^3}\rra\Navy{\eta^3}$&1&2&1\\
$\{1,2\}$&$\lla0\Navy{\alpha^-}\rra\lla1\red{\beta^-}\rra2\lla\Navy{\gamma^+}3\rra\lla\Navy{\delta^+}4\Navy{\ep^3}\rra\Navy{\eta^3}$&1&3&2\\
$\{2\}$&$0\lla\red{\alpha^+}1\red{\beta^-}\rra2\lla\Navy{\gamma^+}3\rra\lla\Navy{\delta^+}4\Navy{\ep^3}\rra\Navy{\eta^3}$&2&4&1,3\\
$\{2,3\}$&$0\lla\Navy{\alpha^+}1\Navy{\beta^-}\rra\lla2\red{\gamma^-}\rra3\lla\Navy{\delta^+}4\Navy{\ep^3}\rra\Navy{\eta^3}$&1&5&4\\
$\{1,2,3\}$&$\lla0\red{\alpha^-}\rra\lla1\Navy{\beta^-}\rra\lla2\red{\gamma^-}\rra3\lla\Navy{\delta^+}4\Navy{\ep^3}\rra\Navy{\eta^3}$&2&6&3,5\\
$\{1,3\}$&$\lla0\Navy{\alpha^-}\rra1\lla\red{\beta^+}2\red{\gamma^-}\rra3\lla\Navy{\delta^+}4\Navy{\ep^3}\rra\Navy{\eta^3}$&2&7&2,6\\
$\{3\}$&$0\lla\red{\alpha^+}1\rra\lla\red{\beta^+}2\red{\gamma^-}\rra3\lla\Navy{\delta^+}4\Navy{\ep^3}\rra\Navy{\eta^3}$&3&8&1,5,7\\
$\{3,4\}$&$0\lla\Navy{\alpha^+}1\rra\lla\Navy{\beta^+}2\Navy{\gamma^-}\rra\lla3\red{\delta^-}\rra\lla4\Navy{\ep^3}\rra\Navy{\eta^3}$&1&9&8\\
$\{1,3,4\}$&$\lla0\red{\alpha^-}\rra1\lla\Navy{\beta^+}2\Navy{\gamma^-}\rra\lla3\red{\delta^-}\rra\lla4\Navy{\ep^3}\rra\Navy{\eta^3}$&2&10&7,9\\
$\{1,2,3,4\}$&$\lla0\Navy{\alpha^-}\rra\lla1\red{\beta^-}\rra\lla2\Navy{\gamma^-}\rra\lla3\red{\delta^-}\rra\lla4\Navy{\ep^3}\rra\Navy{\eta^3}$&2&11&6,10\\
$\{2,3,4\}$&$0\lla\red{\alpha^+}1\red{\beta^-}\rra\lla2\Navy{\gamma^-}\rra\lla3\red{\delta^-}\rra\lla4\Navy{\ep^3}\rra\Navy{\eta^3}$&3&12&5,9,11\\
$\{2,4\}$&$0\lla\Navy{\alpha^+}1\Navy{\beta^-}\rra2\lla\red{\gamma^+}3\red{\delta^-}\rra\lla4\Navy{\ep^3}\rra\Navy{\eta^3}$&2&13&4,12\\
$\{1,2,4\}$&$\lla0\red{\alpha^-}\rra\lla1\Navy{\beta^-}\rra2\lla\red{\gamma^+}3\red{\delta^-}\rra\lla4\Navy{\ep^3}\rra\Navy{\eta^3}$&3&14&3,11,13\\
$\{1,4\}$&$\lla0\Navy{\alpha^-}\rra1\lla\red{\beta^+}2\rra\lla\red{\gamma^+}3\red{\delta^-}\rra\lla4\Navy{\ep^3}\rra\Navy{\eta^3}$&3&15&2,10,14\\
\{4\}&$0\lla\red{\alpha^+}1\rra\lla\red{\beta^+}2\rra\lla\red{\gamma^+}3\red{\delta^-}\rra\lla4\Navy{\ep^3}\rra\Navy{\eta^3}$&4&16&1,9,13,15\\
\hline
\end{tabular}}
\caption{Start of the handle decomposition from $T^{15}$ with $I=\{0,1,2,3,4,6\}$, $J=\varnothing$, $i_*=6$, $U=\varnothing$, $V=\{1,2,3,4\}$.}\label{T:15}
\end{center}
\end{table}

Tables \ref{T:YzrNoSing}, \ref{T:YzrRe}, and \ref{T:YzrSing} list the possible forms for $\xi_r(z)$. Table \ref{T:YzrNoSing} lists those {\it with no} singleton factor. Table \ref{T:YzrRe} lists those  {\it with} a singleton factor $\{i\}$, where $i\in V^+$ and $i+1\in U^\circ\cup U^+\cup V^+$, or $i\in U^-\cup U^\circ\cup V^-$ and $i+1\in V^-$; the class of this case depends on the parity of $\#(V^-\cap \{i+1,\hdots,\max I_{s}\})$, where $i\in I_s$. Table \ref{T:YzrSing} lists the remaining possibilities for $\xi_r(z)$.

\begin{table}[H]
\begin{center}
\begin{tabular}{||c|c|c||}
\hline
class&$\xi_r(z)$&conditions\\
\hline
$\red{\text{(A)}}$& ${[i_*-1,i_*]}$ & $i_*\in J$\\
$\red{\text{(A)}}$& ${\left[i-1,i-\frac{1}{2}\right]}$ & $i\in J\cap V^-\implies i\neq i_*$, $i+1\notin I$\\
\hline
$\Navy{\text{(B)}}$& ${\left[i_*,i_*+\frac{1}{2}\right]}$ & $a\leq i_*\leq b-2$, $i_*+1\in V^-$\\ 
$\Navy{\text{(B)}}$& ${\left[i-\frac{2}{3},i-\frac{1}{3}\right]}$ & $i\in U^\circ$\\
$\Navy{\text{(B)}}$& $ \prod_{j=i_*+1}^{c-1}[i_*,j]^2$ & $i_*=b$, $c\in J$\\
$\Navy{\text{(B)}}$& $ \prod_{j=i_*+1}^{c-2}[i_*,j]^2[i_*,c-1]^3$ & $i_*=b$, $c\notin J$\\
\hline
\end{tabular}
\caption{The possible forms for $\xi_r(z)$ with no singleton factor, where $i_*\in I_{s}$, $a=\min I_s$, $b=\max I_s$, $c=\min I_{s+1}$.
}\label{T:YzrNoSing}
\end{center}
\end{table}

\begin{table}[H]
\begin{center}
\scalebox{.945}{\begin{tabular}{||c|c|cc|c||}
\hline
class&$\xi_r(z)$&conditions on $i$&conditions on $i+1$&parity\\ \hline
$\red{\text{(A)}}$ & ${\left[i-\frac{1}{2},i\right]} \{i\}$ & $i\in V^+$& $i+1\in U^\circ\cup U^+\cup V^+$ & odd\\
 $\red{\text{(A)}}$ & $\{i\} {\left[i,i+\frac{1}{2}\right]}$ & $i\in U^-\cup U^\circ\cup V^-$& $i+1\in V^-$ & odd\\
 $\red{\text{(A)}}$ & ${\left[i-\frac{1}{2},i\right]} \{i\} {\left[i,i+\frac{1}{2}\right]}$ & $i\in V^+$& $i+1\in V^-$ & odd\\
 \hline
 $\Navy{\text{(B)}}$ &  ${\left[i-\frac{1}{2},i\right]} \{i\}$ & $i\in V^+$& $i+1\in U^\circ\cup U^+\cup V^+$ & even\\
 $\Navy{\text{(B)}}$ &  $\{i\} {\left[i,i+\frac{1}{2}\right]}$ & $i\in U^-\cup U^\circ\cup V^-$& $i+1\in V^-$ & even\\
 $\Navy{\text{(B)}}$ &  ${\left[i-\frac{1}{2},i\right]} \{i\} {\left[i,i+\frac{1}{2}\right]}$ & $i\in V^+$& $i+1\in V^-$ & even\\
\hline
\end{tabular}}
\caption{The possible forms for each $\xi_r(z)$ containing a singleton factor $\{i\}$, where $i\in V^+$ and $i+1\in U^\circ\cup U^+\cup V^+$, or $i\in U^-\cup U^\circ\cup V^-$ and $i+1\in V^-$; the class depends on the parity of $\#(V^-\cap \{i+1,\hdots,\max I_{s}\})$, where $i\in I_s$.
}\label{T:YzrRe}
\end{center}
\end{table}

\begin{table}[H]
\begin{center}
\begin{tabular}{||c|c|c||}
\hline
class&$\xi_r(z)$&conditions on $i$\\
\hline
$\red{\text{(A)}}$& ${[i-1,i]}\{i\}$ & $i\in J$,  $i+1\in U^\circ\cup U^+\cup V^+$\\
$\red{\text{(A)}}$& ${[i-1,i]}\{i\}{\left[i,i+1\right]}$ & $i\in J$, $i_*=i+1$\\
$\red{\text{(A)}}$& ${[i-1,i]}\{i\}{\left[i,i+\frac{1}{3}\right]}$ & $i\in J$, $i+1\in U^-$\\
$\red{\text{(A)}}$& ${[i-1,i]}\{i\}{\left[i,i+\frac{1}{2}\right]}$ & $i\in J$, $i+1\in V^-$\\
 $\red{\text{(A)}}$& ${\left[i-\frac{1}{3},i\right]}\{i\}$ & $i\in U^+$, $i+1\in U^\circ\cup U^+\cup V^+$\\
 $\red{\text{(A)}}$&  $\{i\}{\left[i,i+\frac{1}{3}\right]}$ & $i+1\in U^-$, $i\in U^-\cup U^\circ\cup V^-$\\
 $\red{\text{(A)}}$& ${\left[i-\frac{1}{3},i\right]}\{i\}{\left[i,i+\frac{1}{3}\right]}$ & $i\in U^+$,  $i+1\in U^-$\\
 $\red{\text{(A)}}$& ${\left[i-\frac{1}{3},i\right]}\{i\}{\left[i,i+\frac{1}{2}\right]}$,  & $i\in U^+$, $i+1\in V^-$\\
&& $\implies i+1=\max I_s\neq i_*$\\
  $\red{\text{(A)}}$& ${\left[i-\frac{1}{2},i\right]}\{i\}{\left[i,i+\frac{1}{3}\right]}$,  & $i\in V^+$, $i+1\in U^-$\\
  && $\implies i=i_*+1\leq \max I_s-1$\\
  $\red{\text{(A)}}$ & $\{i\}$ & $ i\in (T\setminus J)\cup U^-\cup U^\circ\cup V^-$,\\
&& 
$i+1\in U^\circ\cup U^+\cup V^+$\\
 \hline
$\Navy{\text{(B)}}$& $ [i-1,i]\{i\} \prod_{j=i+1}^{c-2}[i,j]^2[i,c-1]^q$ &  $i_*=\min I_s= i-1=\max I_s-1$\\
$\Navy{\text{(B)}}$& $ \left[i-\frac{1}{2},i\right]\{i\} \prod_{j=i+1}^{c-2}[i,j]^2[i,c-1]^q$ & $ i=\max I_s\in V^+$\\
$\Navy{\text{(B)}}$&$ \{i\} \prod_{j=i+1}^{c-2}[i,j]^2[i,c-1]^q$ & $ i=\max I_s\in V^-$\\
\hline
\end{tabular}
\caption{The possible forms for each $\xi_r(z)$ not listed in Tables \ref{T:YzrNoSing}, \ref{T:YzrRe}.  Each contains a singleton factor $\{i\}$, $i_*\neq i\in I_s$, $s\in\Z_m$. Denote $c=\min I_{s+1}$ with $q\in\{2,3\}$.}\label{T:YzrSing}
\end{center}
\end{table}

\section*{Appendix 2: Four other attempts to multisect $T^n$ for $n$ odd}\label{S:Fail}


\subsection*{From the handle decomposition}\label{S:Fail1}
The $n$-torus has a natural handle decomposition, with $\binom{n}{h}$ $h$-handles for each $h=0,\hdots, n$, which one can construct as follows. View $T^n$ as $(\R/2\Z)^n$, and decompose it into the $2^n$ subcubes with vertices in $(\Z/2\Z)^n$. Then, using notation \ref{N:alpha}, for each $h=0,\hdots,n$, the $h$-handles are the subcubes which are permutations of $\alpha^{n-h}\beta^{h}$.\footnote{Note that this handle decomposition is optimal in the sense that it has the minimum possible number of handles of each index, since $H_h(T^n)$ has rank $\binom{n}{h}$.}  

One might hope that $X_i=\lla \alpha^{n-i}\beta^i\rra\cup\lla\alpha^{n+1-i}\beta^{i-1}\rra$ determines a multisection.\footnote{Note that $n$ is odd throughout Appendix 2.} Indeed, in dimension 3, this is the Heegaard splitting shown in Figure \ref{Fi:T3Intro}.  Yet, the construction does not work beyond dimension 3, as one can see by noting, e.g., that $X_0\cap X_{k-1}=\bigcup_{r=0}^{n-2}\lla\alpha\beta0^r1^{n-2-r}\rra$ is always 2-dimensional.  

\subsection*{By gluing pairs of balls}\label{S:Fail2}
Instead, one might attempt to generalize the following construction. See Figure \ref{Fi:MultiAttempt}. View $T^n$ as $(\R/2k\Z)^n=[0,2k]^n/\sim$. Partition the $(2k)^n$ unit cubes with vertices in the lattice $\left(\Z/2k\Z\right)^n$ so as to form $V_0,\hdots, V_n$ subject to the following conditions:\footnote{These conditions uniquely determine $V_0,\hdots,V_n$.}

\begin{figure}
\begin{center}
\includegraphics[width=\textwidth]{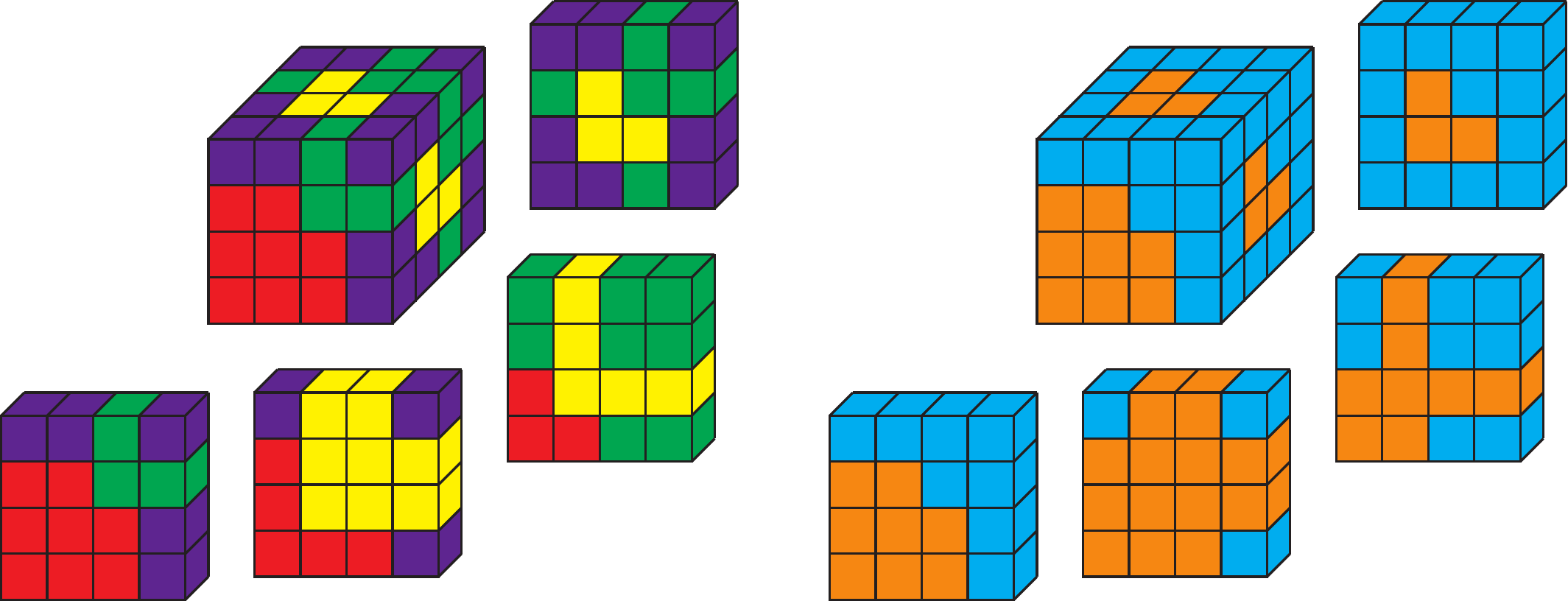}
\caption{Another construction of the minimal genus Heegaard splitting of $S^3$}
\label{Fi:MultiAttempt}
\end{center}
\end{figure}

\begin{itemize}
\item If $\x\in V_0$, then $\x+(r,\hdots,r)\in V_r$;
\item The permutation action on the indices fixes each $V_r$; 
\item $V_0$ contains $[0,1]^n$, is star-shaped about $(0,\hdots,0)$, and contains no points with any coordinate in $(n-1,n)$.
\end{itemize}
Then, for $i=0,\hdots, k=\frac{n+1}{2}$, let $X_i=V_{2i}\cup V_{2i+1}$.  Figure \ref{Fi:MultiAttempt} shows that this construction does in fact give a genus 3 Heegaard splitting of $T^3$.

In higher dimensions,
this construction is promising for many of the same reasons as the construction behind Theorem \ref{T:Main}.  This construction has at least one additional advantage, namely that each $V_i$ is a ball. This makes it easy to check that each $X_i$ is indeed an $n$-dimensional handlebody of genus $n$.  Unfortunately, the complexity of this construction grows much more rapidly than the construction behind Theorem \ref{T:Main}, making it hard to check the other details, even in dimension 5. Indeed, see Figure \ref{Fi:MultiAttemptT5}.

\begin{question}
Does this construction also give a trisection of $T^5$? Does it give a multisection of $T^n$ for arbitrary $n=2k-1$?
\end{question}

\begin{figure}
\begin{center}
\includegraphics[width=\textwidth]{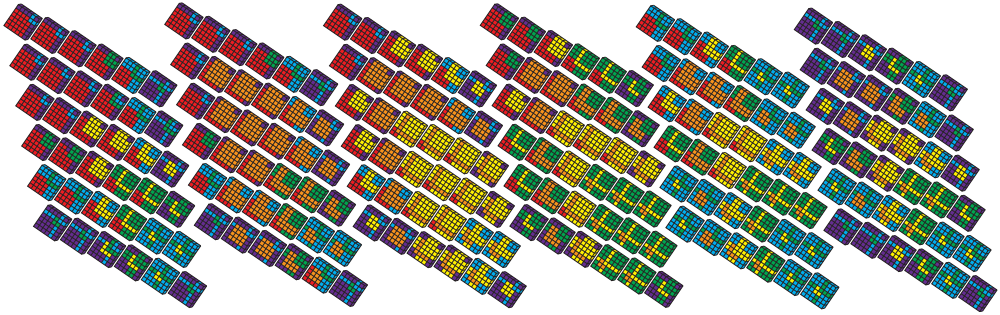}
\caption{Decomposing $T^5=[0,6]^5/\sim$ as $V_0\cup\cdots\cup V_5$. Does $(V_0\cup V_1,V_2\cup V_3,V_4\cup V_5)$ determine a trisection?}
\label{Fi:MultiAttemptT5}
\end{center}
\end{figure}

\subsection*{By summing coordinates}

As shown in Figure \ref{Fi:T3Sum}, the genus 3 Heegaard splitting of $T^3=
[0,2]^3/\sim$ can be constructed as $T^3=X_0\cup X_1$ where each
\[X_i=\{(x_1,x_2,x_3):3i\leq x_1+x_2+x_3\leq 3(i+1)\}/\sim.\]
The splitting surface consists of the hexagon $\{(x_1,x_2,x_3):x_1+x_2+x_3=3\}/\sim$ together with three other hexagons. One is $\{(0,x_2,x_3):1\leq x_2+x_3\leq 5\}/\sim$, and the others are obtained from this one by permuting coordinates. A co-core of one 1-handle in $X_0$ is the triangle $\{(0,x_2,x_3): x_2+x_3\leq 1\}/\sim$, and a co-core of a 1-handle in $X_0$ is the triangle $\{(0,x_2,x_3): 5\leq x_2+x_3\}/\sim$; the other 1-handles of $X_0$ and $X_1$ are related to these by permuting coordinates.

\begin{figure}
\begin{center}
\includegraphics[width=.4\textwidth]{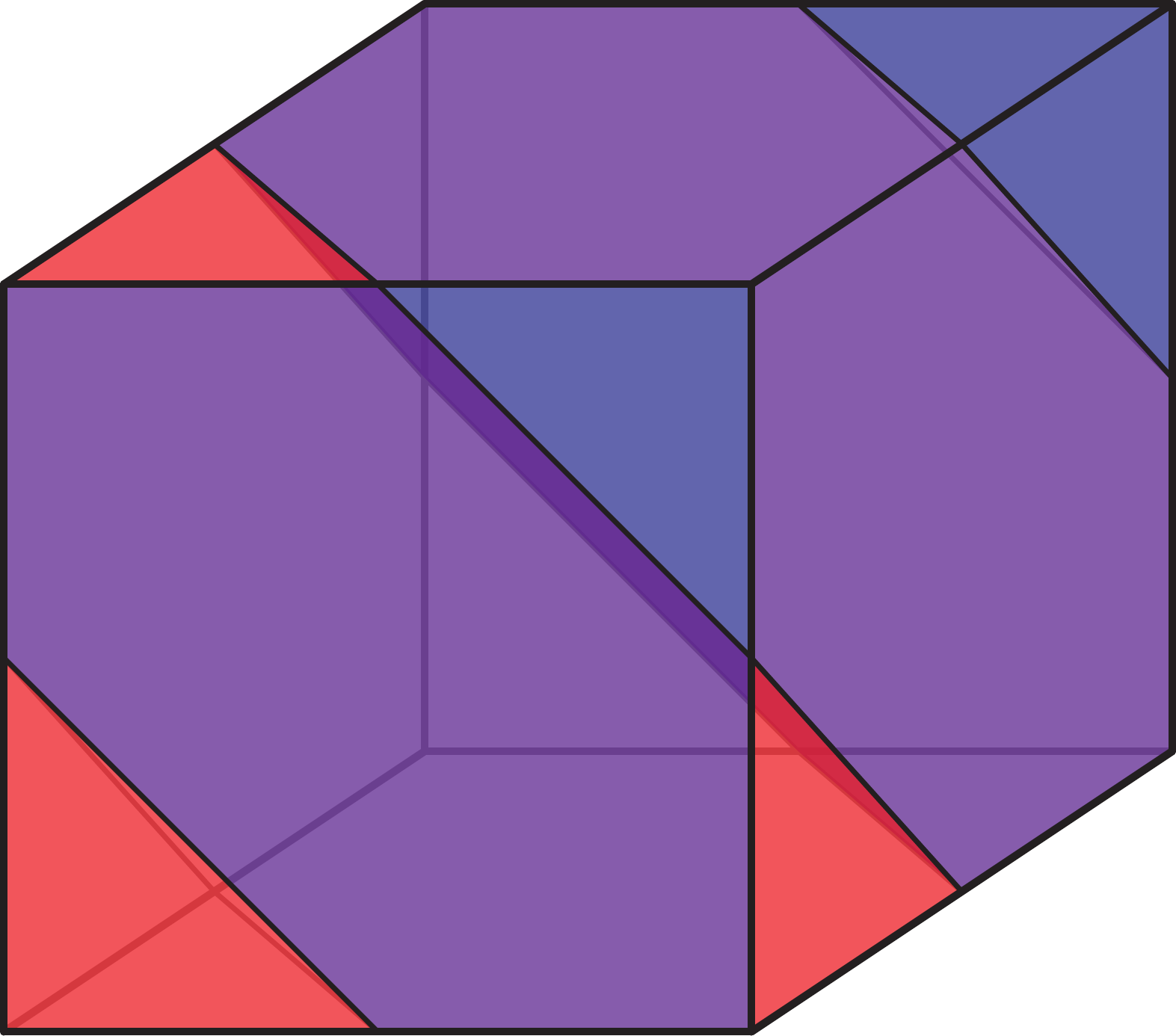}
\caption{The efficient Heegaard splitting $T^3=X_0\cup X_1$ constructed by summing coordinates.  Four purple hexagons comprise the splitting surface. Red and blue triangles are co-cores of the 1-handles in $X_0$ and $X_1$, respectively.}
\label{Fi:T3Sum}
\end{center}
\end{figure}
One might attempt to trisect $T^5=[0,3]^5/\sim$ as $T^5=X_0\cup X_1\cup X_2$ with
\[X_i=\{(x_1,\hdots,x_5):5i\leq x_1+\cdots+x_5\leq 5(i+1)\}/\sim.\]
Then each $X_i$ is in fact a 4-dimensional 1-handlebody of genus 4: a co-core of a 1-handle of $X_0$ is the 4-simplex  $\{(0,x_2,x_3,x_4,x_5): x_2+x_3+x_4+x_5\leq 2\}/\sim$, a co-core of a 1-handle of $X_1$ is $\{(0,x_2,x_3,x_4,x_5): 5\leq x_2+x_3+x_4+x_5\leq 7\}/\sim$, and a co-core of a 1-handle of $X_2$ is $\{(0,x_2,x_3,x_4,x_5): 12\leq x_2+x_3+x_4+x_5\}/\sim$; the other 1-handles of $X_0$, $X_1$, and $X_2$ are related to these by permuting coordinates. 

Yet, this is not a trisection, because 
\[X_0\cap X_2=\{(x_1,x_2,x_3,0,0)_\sigma:~4\leq x_1+x_2+x_3\leq 5,~\sigma\in S_5\}/\sim\]
is 3-dimensional, not 4-.

To fix this problem, one could choose $0=a_0<a_1<a_2<a_3=15$ differently and define each
\[X_i=\{(x_1,\hdots,x_5):a_i\leq x_1+\cdots+x_5\leq a_{i+1}\}.\]
Then $X_0\cap X_2$ will be 4-dimensional if and only if $a_2-a_1<3$.  This creates a new problem: if $a_2-a_1<3$, then $X_1$ is contractible, hence a 5-ball. It now follows from Proposition \ref{P:SmoothEfficient1} that no choice of $a_1$ and $a_2$ produces a trisection of $T^5$.  The same difficulty prevails in all other dimensions $n>3$ (including even dimensions).

\subsection*{Using the symmetric space $T^n/S_n$}\label{S:Fail3}

Given a triangulation $K$ of an $n$-manifold $X$, Rubinstein--Tillmann multisect $X$ by mapping each $n$-simplex of $K$ to the standard 
$(k-1)$-simplex 
\begin{equation}\label{E:Simplex}
\Delta_{k-1}=[\vv_0,\hdots,\vv_{k-1}]=\biggl\{\sum_{j\in \Z_k}a_j\vv_j:~0\leq a_j,~\sum_{j\in\Z_k}a_j=1\biggr\},
\end{equation}
decomposing $\Delta_{k-1}=\bigcup_{i\in\Z_k}Z_i$ where each
\begin{equation}\label{E:RTDelta}
Z_i=\{\x\in \Delta_{k-1}:~|\x-\vv_i|\leq|\x-\vv_j|~\forall j\in\Z_k\},
\end{equation}
(see Figure \ref{Fi:DeltaDecomp}), and pulling back.  Their maps from the $n$-simplices of $K$ to $\Delta_{k-1}$ are simplest to construct in odd dimension $n=2k-1$. Namely:
\begin{itemize}
\item map the barycenter of each $r$-face to $\vv_j\in\Delta_{k-1}$, $j=2r,2r+1$; and
\item extend linearly in the first barycentric subdivision of $K$.
\end{itemize}
The even-dimensional case is similar, but with an extra move.

\begin{figure}
\begin{center}
\labellist
\small\hair 4pt
\pinlabel {$\red{v_0}$} [t] at 25 45
\pinlabel {$\NavyBlue{v_1}$} [t] at 295 45
\pinlabel {$\red{v_0}$} [t] at 440 5
\pinlabel {$\NavyBlue{v_1}$} [t] at 700 5
\pinlabel {$\yellow{v_2}$} [t] at 560 320
\pinlabel {$\red{v_0}$} [t] at 780 120
\pinlabel {$\NavyBlue{v_1}$} [t] at 935 -120
\pinlabel {$\yellow{v_2}$} [t] at 1330 60
\pinlabel {$v_3$} [t] at 1100 350
\endlabellist
  \includegraphics[width=.1875\textwidth]{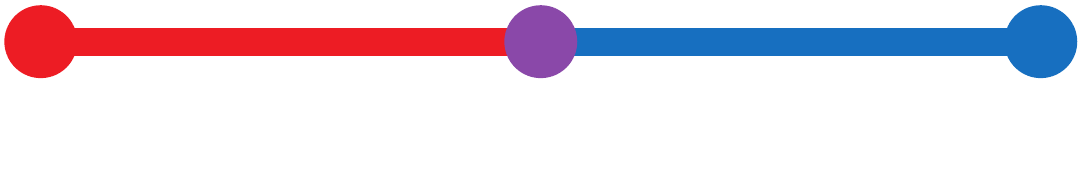}
\hspace{.04\textwidth}
 \raisebox{-.025\textwidth}{\includegraphics[width=.1875\textwidth]{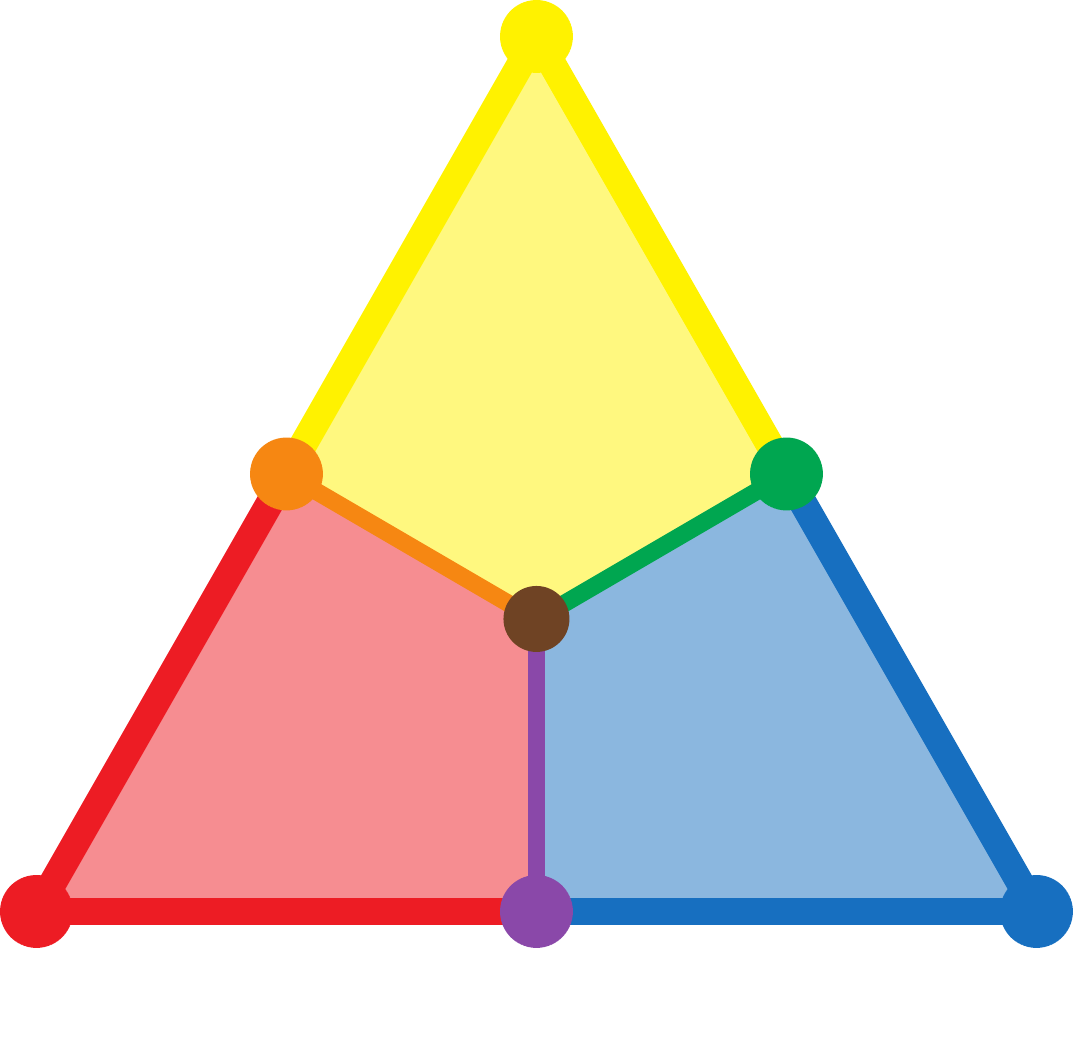}}
\hspace{.04\textwidth}
 \raisebox{-.1\textwidth}{\includegraphics[width=.3\textwidth]{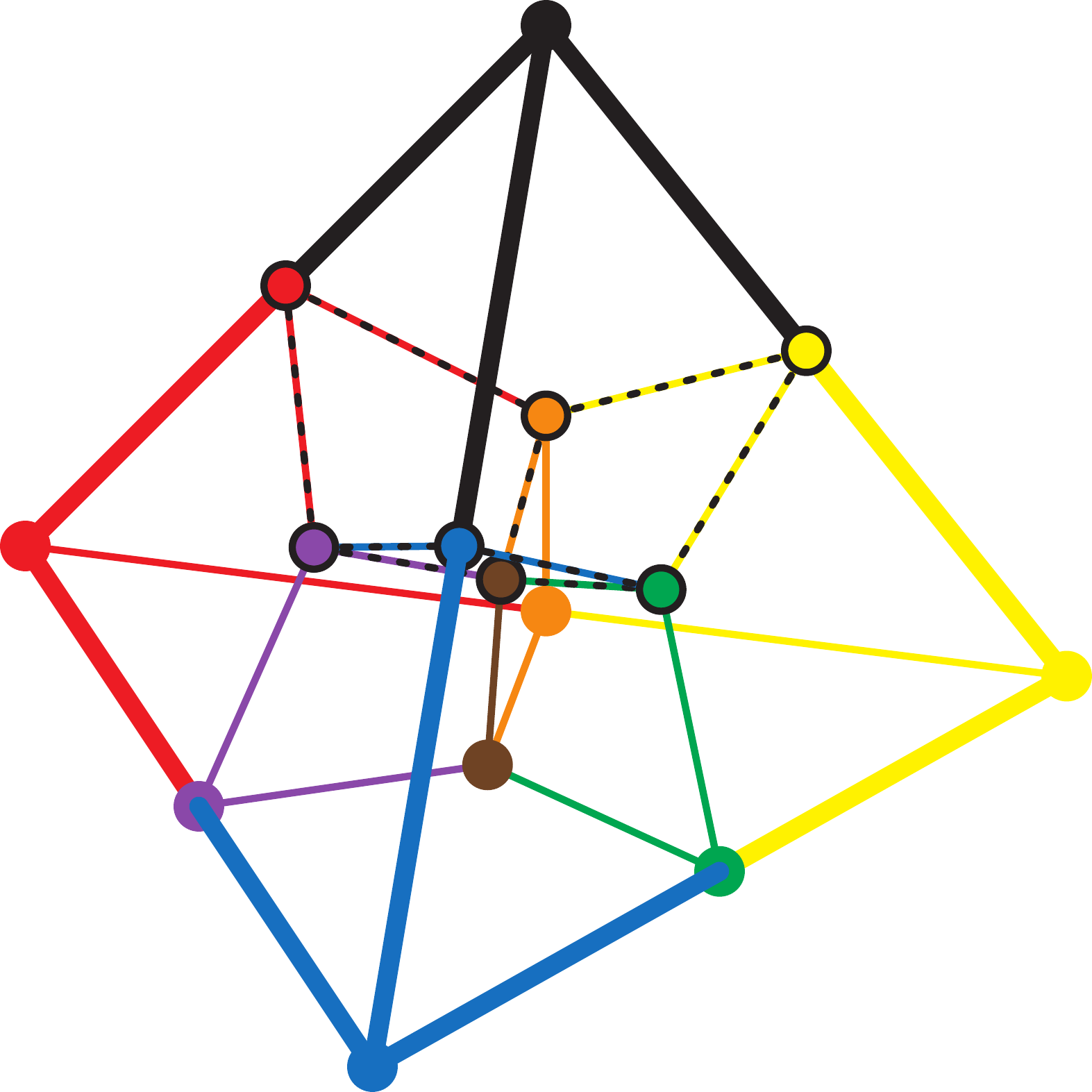}}
\caption{The decompositions $\Delta_{k-1}=\bigcup_{i\in\Z_k}Z_i$ of the $1$-, $2$-, and $3$-simplices following Rubinstein--Tillmann.}
\label{Fi:DeltaDecomp}
\end{center}
\end{figure}
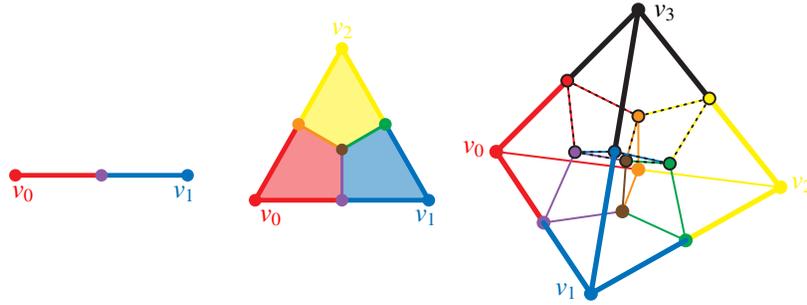

For example, the triangulation of $S^3$ with two 3-simplices gives a genus 3 Heegaard splitting, as shown in Figure \ref{Fi:S3PL}.

\begin{figure}
\begin{center} 
\labellist
\small\hair 4pt
\pinlabel {$\red{v_0}$} [t] at 650 260
\pinlabel {$\NavyBlue{v_1}$} [t] at 930 260
\endlabellist
\includegraphics[width=.95\textwidth]{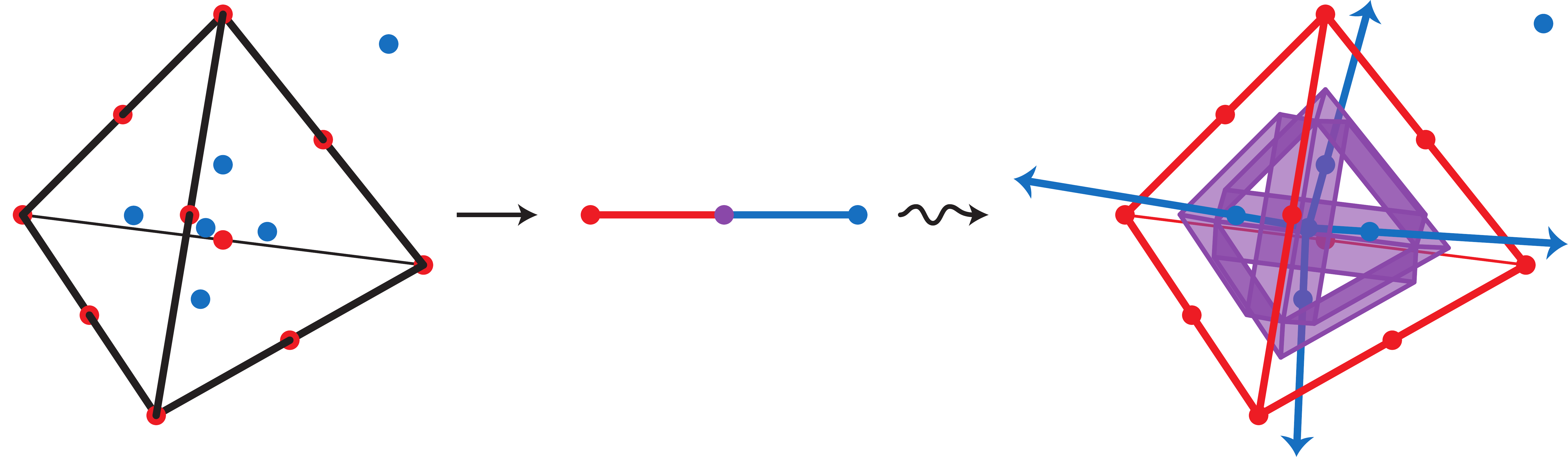}
\caption{A genus 3 Heegaard splitting (right) of $S^3$, following Rubinstein--Tillmann's construction.}
\label{Fi:S3PL}
\end{center}
\end{figure}

Following Rubinstein-Tillmann, one might try to construct a, say PL, multisection of $T^n$ using the symmetric space $T^n/S_n$, which is homeomorphic to a disk-bundle over the circle; this bundle is twisted when $n$ is even and untwisted when $n$ is odd. 

One can also view the symmetric space $T^n/S_n$ 
as an $n$-simplex $\Delta_n=[\vv_0,\hdots,\vv_n]$ with certain faces identified.  When $n=2k-1$, one can also view $\Delta_n$ as an iterated join of intervals,
\[\Delta_n=[\vv_0,\vv_1]*\cdots*[\vv_{2(k-1)},\vv_{2k-1}].\]
Hence, there is a map 
$\phi:\Delta_n\to\Delta_{k-1}=[\vv_0,\hdots,\vv_n]$ given by 
\[\phi:\x=\sum_{i=0}^{k-1}w_i(c_i\vv_{2i}+(1-c_i)\vv_{2i+1})\mapsto \sum_{i=0}^{k-1}w_i\vv_i.\]
One can then decompose $\Delta_{k-1}$ symmetrically into $k$ pieces using barycentric coordinates as in (\ref{E:RTDelta}) and Figure \ref{Fi:RT}. Following Rubinstein--Tillmann's construction of PL multisections from triangulations \cite{rt}, one might attempt to construct a multisection of $T^n$ by pulling back each $X_i$ via $\phi$, mapping forward by the quotient map $\Delta_n\to T^n/S_n$, and pulling back by the quotient map $T^n\to T^n/S_n$.

\begin{figure}
\begin{center}
\labellist
\small\hair 4pt
\pinlabel {$\red{v_0}$} [t] at 10 30
\pinlabel {$\red{v_1}$} [t] at 210 30
\pinlabel {$\NavyBlue{v_2}$} [t] at 350 70
\pinlabel {$\NavyBlue{v_3}$} [t] at 445 160
\pinlabel {$\yellow{v_4}$} [t] at 480 320
\pinlabel {$\yellow{v_5}$} [t] at 480 460
\pinlabel {$\red{v_0}$} [t] at 605 105
\pinlabel {$\NavyBlue{v_1}$} [t] at 850 105
\pinlabel {$\yellow{v_2}$} [t] at 725 410
\endlabellist
\includegraphics[width=.6\textwidth]{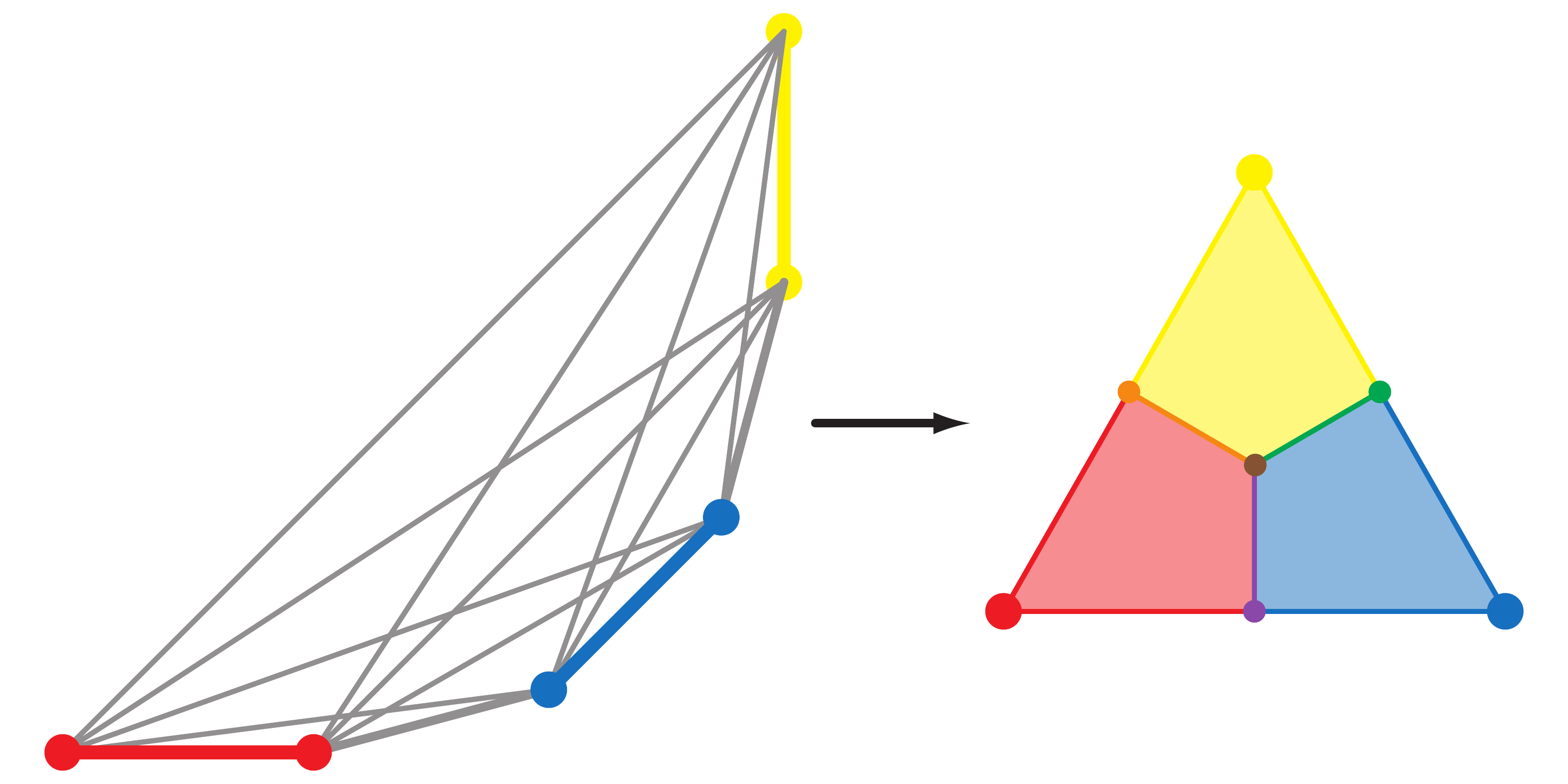}
\caption{Try viewing $T^n/S_n$ as $\Delta_n/\sim$ and $\Delta_n$ as an iterated join of $k$ intervals.  Then map $\Delta_n\to\Delta_{k-1}$, decompose $\Delta_{k-1}$, and pull back. It fails, even for $n=5$, shown.}
\label{Fi:RT}
\end{center}
\end{figure}
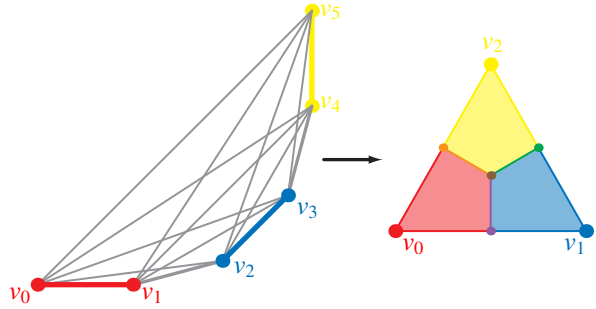

This construction works for $T^3$ and cuts any $T^n$ into $k$ 1-handlebodies of genus $n$.  
Unfortunately, the needed intersection properties fail, even for $T^5$, so the decomposition is not a multisection. Note that by writing 
\[\Delta_n=[\vv_0,\vv_1]*\cdots*[\vv_{2(k-1)},\vv_{2k-1}]\]
we made an asymmetric choice, and that the resulting decomposition is generally different than the one obtained by writing 
\[\Delta_n=[\vv_{\sigma(0)},\vv_{\sigma(1)}]*\cdots*[\vv_{\sigma(2k-2)},\vv_{\sigma(2k-1)}]\]
for arbitrary $\sigma\in S_n$ and then following the same procedure.

\end{document}